\documentclass[pdftex,a4paper,reqno]{amsart}
\pdfoutput=1

\usepackage{amsmath,amssymb,amsthm}
\usepackage{mathtools}

\usepackage{prftree}

\usepackage[foot]{amsaddr}

\usepackage{biblatex}
\usepackage{hyperref}
\usepackage{cleveref}
\usepackage{mleftright}
\usepackage{comment}

\usepackage{tikz}
\usetikzlibrary{cd}

\crefname{enumi}{}{}
\creflabelformat{enumi}{#2(#1)#3}
\crefname{section}{Section}{Sections}
\crefname{subsection}{Section}{Sections}

\theoremstyle{definition}
\newtheorem{definition}{Definition}[section]
\crefname{definition}{Definition}{Definitions}
\newtheorem{notation}[definition]{Notation}
\crefname{notation}{Notation}{Notations}

\theoremstyle{plain}
\newtheorem{theorem}[definition]{Theorem}
\crefname{theorem}{Theorem}{Theorems}
\newtheorem{proposition}[definition]{Proposition}
\crefname{proposition}{Proposition}{Propositions}
\newtheorem{lemma}[definition]{Lemma}
\crefname{lemma}{Lemma}{Lemmas}
\newtheorem{corollary}[definition]{Corollary}
\crefname{corollary}{Corollary}{Corollaries}

\theoremstyle{remark}
\newtheorem{remark}[definition]{Remark}
\crefname{remark}{Remark}{Remarks}
\numberwithin{equation}{section}

\newcommand{\boundary}{\partial}
\newcommand{\coloniff}{\;\defcolon\Longleftrightarrow\;}
\newcommand{\defcolon}{\mathrel{:}}
\DeclareMathSymbol{\mhyph}{\mathalpha}{operators}{`-}

\newcommand{\metanats}{\mathord{\mathbb{N}}}
\newcommand{\relmiddle}[1]{\mathrel{}\middle#1\mathrel{}}
\NewDocumentCommand{\set}{m o}{%
  \left\{#1\IfValueT{#2}{\relmiddle| #2}\right\}}
\NewDocumentCommand{\mtCard}{s m}{
  \mathop{\#}\IfBooleanTF{#1}{\mleft(#2\mright)}{#2}
}

\newcommand{\mtcCat}{\mathord{\mathbf{Cat}}} 
\newcommand{\mtcSimpx}{\mathord{\mathbf{\Delta}}} 
\newcommand{\mtcSimpxAug}{\mathord{\mathbf{\Delta}_a}} 
\newcommand{\mtcPoset}{\mathord{\mathbf{Poset}}} 
\newcommand{\mtcSet}{\mathord{\mathbf{Set}}} 
\newcommand{\mtcPsh}[2]{{#2}_{#1}} 
\DeclareMathOperator{\mtynd}{\mathbf{y}} 
\newcommand{\mtcSSet}{\mtcPsh{\mtcSimpx}{\mtcSet}} 
\newcommand{\mtcSSSet}{\mtcPsh{\mtcSimpx\times\mtcSimpx}{\mtcSet}} 
\newcommand{\dualCat}[1]{#1^{\mathord{\mathrm{op}}}} 
\newcommand{\overcat}[2]{#1/#2} 
\newcommand{\functCat}[2]{{#2}^{#1}} 
\DeclareMathOperator{\Ob}{Ob} 
\DeclareMathOperator{\Hom}{Hom}  
\DeclareMathOperator{\Mor}{Mor} 
\DeclareMathOperator{\dom}{dom} 
\NewDocumentCommand{\HomOf}{o m m}{%
  \Hom\IfValueT{#1}{_{#1}}\mleft(#2,#3\mright)} 
\newcommand{\deltaBra}[1]{\mleft[#1\mright]} 
\DeclareMathOperator{\mtyndSpx}{\Delta} 
\DeclareMathOperator{\mtBdrySpx}{\boundary\Delta} 
\newcommand{\mtsHorn}[2]{\Lambda_{#2}[#1]} 
\newcommand{\compos}{\mathbin{\circ}} 
\newcommand{\whiskr}{\triangleright} 
\newcommand{\whiskl}{\triangleleft} 
\newcommand{\uniqmor}{\mathord{\boldsymbol{!}}} 
\NewDocumentCommand{\forwardinduce}{s m}{%
  \IfBooleanTF{#1}{{(#2)}_*}{#2_*}}
\NewDocumentCommand{\backwardinduce}{s m}{%
  \IfBooleanTF{#1}{{(#2)}^*}{#2^*}}
\NewDocumentCommand{\idmor}{o}{\mathord{\mathrm{id}}\IfValueT{#1}{_{#1}}}
\newcommand{\catNerve}{\mathop{\mathbf{N}}\nolimits} 
\DeclareMathOperator{\nerve}{N} 
\DeclareMathOperator*{\mtColim}{colim} 
\DeclareMathOperator{\mtLan}{Lan} 
\newcommand{\mtsSk}[1]{\operatorname{Sk}_{#1}} 
\DeclareMathOperator{\Down}{Down}
\newcommand{\Downstar}{\Down_*}
\NewDocumentCommand{\upset}{o m}{%
  \left[#2,\infty\right[\IfValueT{#1}{_{#1}}%
}
\NewDocumentCommand{\cintv}{o m m}{%
  \left[#2,#3\right]\IfValueT{#1}{_{#1}}%
}
\newcommand{\proj}{\mathord{\mathrm{pr}}}
\NewDocumentCommand{\last}{o}{%
  \mathop{\mathfrak{last}}\nolimits\IfValueT{#1}{_{#1}}%
}
\NewDocumentCommand{\weqlast}{o}{%
  W^{\last}\IfValueT{#1}{_{#1}}%
}
\DeclareMathOperator{\mtsDcp}{Dcp}
\newcommand{\mtpDcp}{\mtsDcp_P}
\newcommand{\mtcDcp}{\mtsDcp_C}
\DeclareMathOperator{\mtsDcpI}{DcpI}

\newcommand{\mtcDcpI}{\mtsDcpI_C}
\DeclareMathOperator{\mtsESd}{ESd}
\DeclareMathOperator{\mtsESdI}{ESdI}
\newcommand{\mtsESdIc}{\mtsESdI_{\mtcSimpx}}
\newcommand{\mtsESdp}{\mtsESd'}
\newcommand{\mtsESdpI}{\mtsESdI'}
\newcommand{\mtpESdp}{\mtsESdp_P}
\newcommand{\mtpESdpI}{\mtsESdpI_P}
\DeclareMathOperator{\mtsDcpC}{DcpC}
\newcommand{\mtsESdpC}{\operatorname{ESdC}'}

\NewDocumentCommand{\idtype}{o m m}{%
  {{#2} =\IfValueT{#1}{_{#1}} {#3}}}

\NewDocumentCommand{\tytrunc}{o m}{\left\lVert #2 \right\rVert
  \IfValueT{#1}_{#1}}

\NewDocumentCommand{\iterSig}{o m}{%
  \mathop{\mathsf{Sig}}\mleft[#2\mright]\IfValueT{#1}{_{#1}}}
\NewDocumentCommand{\iterSigCon}{o m}{%
  \mathop{\sigma}\mleft[#2\mright]\IfValueT{#1}{_{#1}}}
\NewDocumentCommand{\iterPi}{o m m}{%
  \mleft[#2\mright]\IfValueT{#1}{_{#1}} \to #3}
\NewDocumentCommand{\iterFct}{o m m}{%
  \mathop{\lambda}\mleft[#2\mright]\IfValueT{#1}{_{#1}}.\,{#3}}

\NewDocumentCommand{\refl}{o m}
{\mathsf{refl}\IfValueT{#1}{^{#1}}_{#2}}

\makeatletter
\DeclareRobustCommand%
  \tightvdots{\vbox{\baselineskip4\p@ \lineskiplimit\z@%
    \hbox{.}\hbox{.}\hbox{.}}}
\makeatother

\makeatletter
\NewDocumentCommand{\prooftree@my@one}{m}{%
  \IfValueTF{#1}{\prftree[#1]}{\prftree}}

\NewDocumentCommand{\prooftree@my@two}{m m}{%
  \IfValueTF{#2}{\prooftree@my@one{#1}[l]}{\prooftree@my@one{#1}}}

\NewDocumentCommand{\prooftree@my@three}{m m m} {%
  \IfValueTF
      {#3}%
      {\prooftree@my@two{#1}{#2}[r]{#3}}%
      {\prooftree@my@two{#1}{#2}}}

\NewDocumentCommand{\prooftree@my@four}{m m m} {%
  \IfValueTF
      {#2}%
      {\prooftree@my@three{#1}{#2}{#3}{#2}}%
      {\prooftree@my@three{#1}{#2}{#3}}}

\NewDocumentCommand\prooftree{o d\l\endl d\r\endr d\l\endl m m}{
  \IfValueTF{#2}{\prooftree@my@four{#1}{#2}}
  {\prooftree@my@four{#1}{#4}}{#3}#5{#6}\relax}

\makeatother

\title[Finite Reedy categories as localizations of finite direct categories]{
  Presentation of finite Reedy categories as localizations of finite direct categories}

\author{Genki Sato}

\email{gettaplacetogo@gmail.com}

\begin{document}

\begin{abstract}
  In this paper, we present a construction from a Reedy category $C$
  of a direct category $\Down(C)$ and a functor $\Down(C) \to C$, 
  which exhibits $C$ as an $(\infty,1)$-categorical localization of $\Down(C)$. 
  This result refines previous constructions in the literature by ensuring 
  finiteness of the direct category $\Down(C)$ whenever $C$ is finite,
  which is not guaranteed by existing approaches.
  The finiteness property is useful when we want to embed the construction into
  the syntax of a (non-infinitary) logic: the author expects the construction may be used to
  develop a meta-theory of finitely truncated simplicial types for homotopy type theory.
\end{abstract}

\maketitle
\tableofcontents

\section{Introduction}
\label{sec:introduction}

In this paper, we prove the following result:

\begin{theorem}\label{thm:main}
  Let $C$ be a Reedy category. Then there is a concrete construction 
  (to be seen in \cref{def:down-c,def:down-fctr-last}) of
  a direct category $\Down(C)$ and a functor $\Down(C) \to C$
  which exhibits $C$ as an $(\infty, 1)$-categorical localization of $\Down(C)$.
  Furthermore, the category $\Down(C)$ is finite whenever $C$ is.
\end{theorem}

Without the mention of finiteness, the proof of this result appears
in literature in a stronger form:

\begin{theorem}[Preceding result]\label{thm:preceding}
  Let $X$ be any simplicial set. Then there is a well-founded partial order $P$
  and a functor $P \to X$, i.e., a simplicial map $\nerve(P) \to X$ exhibiting
  $X$ as an $(\infty,1)$-localization of $P$.
\end{theorem}

A direct proof of \cref{thm:preceding} may be seen, for example, in the Lurie's
online textbook \emph{Kerodon} as \cite[\href{https://kerodon.net/tag/02MD}{Theorem 02MD}]{kerodon}.
The proof of a claim ``equivalent to \cref{thm:preceding} up to homotopy'' can also
be found in an earlier literature: in 
\cite{barwick2011-relcat-model,barwick2011-thomasonlike-quilleneqv-quasicategories},
Barwick and Kan state that the cofibrant objects of a model category of relative categories,
which models all $(\infty,1)$-categories, are necessarily relative posets; therefore
any $(\infty,1)$-category is a localization of a poset.

As you can see, the statement of Barwick and Kan does not refer to well-foundedness
(or equivalently directness). Strictly speaking, Lurie's statement does not either.
However, with a closer inspection of Lurie's proof, it is easily seen that the poset
he has constructed is well-founded. I, the author, have not inspected Barwick and Kan's proof,
but I roughly expect that the poset they have constructed is well-founded as well,
as they use a modified nerve functor.

Although the preceding \cref{thm:preceding} is almost stronger than our \cref{thm:main},
there is an improvement in the latter: finiteness. If $C$ is a finite Reedy category,
the category $\Down(C)$, which has $C$ as a localization, is in fact a finite direct category.
The preceding constructions do not seem to guarantee the finiteness of the poset $P$
except in very limited cases.
For example, Lurie's construction of $P$ may be taken to be finite precisely when $X$ is a finite
simplicial set. If we wish to take $X = \nerve(C)$, the nerve of a 1-category, the simplicial set
$X$ is finite if and only if $C$ is already finite and \emph{direct}. Therefore, if you wish to
get a finite direct category from \cref{thm:preceding}, you need to start with a finite direct
category, which renders \cref{thm:preceding} useless for the purpose.

On the other hand, our \cref{thm:main} provides a finite direct category from a finite
Reedy category. Furthermore, by applying our construction and then Lurie's construction,
we obtain the following corollary:
\begin{corollary}
  Let $C$ be a finite Reedy category. Then there is a finite partial order $P$
  and a functor $P \to C$ exhibiting $C$ as an $(\infty,1)$-localization of $P$.
\end{corollary}

The author expects that the result of this paper may be useful when one wishes 
to embed the construction into the syntax of a (non-infinitary) logic. 
In particular, it may have applications to the meta-theory of 
homotopy type theory (HoTT)~\cite{hottbook}. In HoTT, there is a meta-notion 
of finite direct presheaves (or finite inverse diagrams): for each finite direct category $C$, 
there is a well-defined notion of the type of presheaves on $C$ 
\cite{SHULMAN_2014,kraus2017-spacevalueddiagrams}. 
The result of this paper suggests that these notions can be extended to finite Reedy presheaves:
if $C$ is a finite Reedy category, a presheaf over $C$ may be defined as a presheaf over 
$\Down(C)$ equipped with a point of an appropriate subcontractible space. 
A key example is that of finitely truncated simplicial types. 
For each $n$, there is a theory of $n$-truncated semisimplicial types; however, 
the theory of $n$-truncated simplicial types for an arbitrary meta-level natural number $n$
has not yet been developed. Our result suggests that the theory of $n$-simplicial types can be developed 
in a straightforward way.

Kraus and Sattler have already worked along these lines in their extended abstract 
\cite{kraus2017-spacevalueddiagrams}. As a part of their work,
they have proposed a sketch of a in-HoTT theory of diagrams over a fixed Reedy category
with certain conditions via a ``direct replacement''. An important example of their work is a simple infinite
direct replacement of the untruncated simplex category, which is far simpler than the present construction.
This approach yields a definition of simplicial types in terms of countably infinite meta-series of type judgements in HoTT.
However, in private communication,
Kraus and the author have confirmed that truncated simplex categories do not satisfy the conditions 
proposed in \cite{kraus2017-spacevalueddiagrams}, contrary to the claim in it.
Our present construction applies to all Reedy categories, including truncated simplex categories.

Let us roughly outline the construction of $\Down(C)$.
Given a finite Reedy category $C=(C, C_{-}, C_{+})$, the construction of $\Down(C)$ begins with
the direct subcategory $C_{+}$. If we added morphisms from $C_{-}$ to this, we would obtain $C$, but
it would break directness. Instead, for each non-identity $f\colon x \to y$ in $C_{-}$, 
we would like to add a new object $c_f$ and two morphisms, as in:
\[
\begin{tikzcd}
  x \ar[r] & c_f & y \ar[l, "\sim"']
\end{tikzcd}
\]
If we make the morphism $c_f \leftarrow y$ a weak equivalence, we would get a ``factorization'' of $f$.
If we make $c_f$ have a higher degree than $x$ and $y$, the directness or the finiteness would not be broken. 
However, since we need to ensure coherence, just adding these objects $c_f$ and connecting morphisms is not enough.
We actually need to add an arbitraily long composable chains in $C_{-}$ as objects
and we consider any series of morphisms in $C_{+}$ that connect these chains as
morphisms; the resulting category is called $\int \catNerve^{{-},{+}}(C)$ in \cref{sec:def-cat-down}. 
However this does not guarantee finiteness, because this construction allows
arbitrarily long chains of idenitities. These arbitrary long chains actually breaks the directness
(it is merely Reedy).
By quotienting out these idenities, we obtain a finite direct
category $\Down(C)$, because $C_{-}$ is finite inverse. This is the rough idea of the construction.

Let us now give an outline of the paper. In \cref{sec:cat-prel}, we wrap up some terminology and some
results in category and $\infty$-category theory. In \cref{sec:def-cat-down}, we define the category
$\Down(C)$ and some other related categories.
In \cref{sec:groth-cat-partial-order}, we investigate the structure of hom-set of $\Down(C)$. 
From this, the directness and the finiteness of $\Down(C)$ will be shown in \cref{sec:groth-cat-direct}.
In \cref{sec:groth-cat-comparison}, we shall study some functors among $\Down(C)$ and other categories
constructed in \cref{sec:def-cat-down}. This helps us to show that $\Down(C)\to C$, among other $C$-valued
functors, are 1-localizations in \cref{sec:down-last-localiz}.
The functor $\Down(C) \to C$ itself, called $\last$, will also be constructed in \cref{sec:down-last-localiz}.
After that, \cref{sec:down-last-infty-loc-shape,sec:down-last-infty-loc} will be devoted to the proof
that $\Down(C)\to C$ is an $(\infty,1)$-localization. In \cref{sec:down-last-infty-loc-shape}, we shall
define some endofunctors on the category of simplicial sets and investigate their properties. These
endofunctors will describe the shapes of the $(\infty,1)$-diagrams in the actual proof
in \cref{sec:down-last-infty-loc} that $\Down(C)\to C$ is an $(\infty,1)$-localization.
Our main theorem, \cref{thm:main}, will be proven in \cref{sec:wrapping-up}. The content of the theorem
shall have been proven part by part in the previous sections, and the proof in \cref{sec:wrapping-up}
will be a concise summary of these results.

Throughout this paper except for \cref{sec:cat-prel}, we shall fix a Reedy category $C=(C, C_{-}, C_{+})$.

\section*{Acknowledgements}

I would like to thank my advisors, Toshitake Kohno and Nariya Kawazumi, 
for their invaluable guidance and encouragement.
I am also grateful to Yohsuke Matsuzawa for organizing regular progress meetings, 
despite not sharing my mathematical background, and to Jun Yoshida and to Satoshi Sugiyama
for their helpful conversations.
I appreciate the great previous research by Nicolai Kraus and Christian Sattler,
and I am grateful to Nicolai Kraus for the wounderful discussion on their work.
I appreciate the insights provided by Andrej Bauer, Selgei Burkin, Simon Henry,
Peter LeFanu Lamsdaine, Maxime Ramzi, Mike Shulman, Paul Taylor, and David White
through MathOverflow.
My thanks also go to my collegemates who have supported and delighted me,
during my course at the University of Tokyo. Finally,
although this paper was completed after my departure from the university,
my doctoral research was supported by a JSPS Research Fellowship
and the FMSP program at the University of Tokyo.

\section{Preliminaries}
\label{sec:cat-prel}

In this section, we review some category-theoretic preliminaries.
In \cref{sec:cat-prel-grossaries}, we list some standard category-theoretic terminology and notations,
including the weak and the strict definitions of 1-localizations.
In \cref{sec:cat-prel-foundations}, we remark on the set-theoretic foundations of this paper.
Since the first half of this paper is based on the language of constructive mathematics, we shortly discuss
the possible formalization of the part. The subsection also includes the definition of well-founded relations, which is
foundational for defining direct categories.
In \cref{sec:cat-prel-direct-reedy}, we discuss the definitions of direct categories and Reedy categories.
Those notions need explicit reconsideration under the constructive settings, we shall discuss them in detail.
In \cref{sec:cat-prel-simpx}, we recall the definition of our most important Reedy category,
the simplex category $\mtcSimpx$, and list some related notations and terminologies.
Finally in \cref{sec:cat-prel-simplicial-quasicat}, we shall cite some terminology, notation and results from
$(\infty, 1)$-category theory based on simplicial sets as models. This will be vital in
\cref{sec:down-last-infty-loc-shape,sec:down-last-infty-loc}
for formulating and proving that $\Down(C) \to C$ is an $(\infty,1)$-localization.

\subsection{Grossaries}
\label{sec:cat-prel-grossaries}

We first list some category-theoretic terminology and notations.
This is long, but it is a standard glossary in category theory,
so the readers may wish to skip this list.
See, for instance, \cite{awodey-cat,Borceux_1994,maclane-cat} for reference.

\begin{definition}[Category-theoretic glossary]
  \ \label{def:cat-glossary}
  \begin{itemize}
    \item Let $C$ be a category. Then we write $\Ob C$ for the collection of
    objects of $C$, $\HomOf[C]{x}{y}$ for the collection of morphisms
    $x \to y$ in $C$, and $\Mor C$ for the collection of all morphisms
    of $C$. If $f\colon x \to y$ and $g\colon y \to z$ are morphisms in a
    category,
    their composition $x \to z$ will be denoted by $g \compos f$.
    The identity morphism of an object $x$ in a category is written
    as $\idmor[x]$.
    \item If there is no danger of confusion, we may write $f \in C$
    to mean $f \in \Mor C$, and $x \in C$ to say $x \in \Ob C$.
    \item A \emph{functor} $F\colon C\to D$, as usual, stands for 
    a covariant one: a pair of
    a function $F\colon \Ob C \to \Ob D$ and a family of functions
    $F\colon \HomOf[C]{x}{y} \to \HomOf[D]{F(x)}{F(y)}$ preserving
    identities and compositions. If $F\colon C\to D$ and $G\colon D\to E$
    are functors, their composition $x\mapsto G(F(x))$ is denoted by
    $G\compos F\colon C\to E$. The identity functor of a category $C$ is
    denoted by $\idmor[C]$.
    \item A functor $F\colon C\to D$ is \emph{faithful} if, for any pair
    of objects $x,y$ in $C$, the function
    $F\colon \HomOf[C]{x}{y} \to \HomOf[D]{F(x)}{F(y)}$ is injective.
    The functor is \emph{full} if this function is surjective,
    and \emph{fully faithful} if it is an isomorphism of sets.
    \item A category $C'$ is said to be a \emph{subcategory} of a category $C$,
    denoted by $C' \subseteq C$, if there is a faithful functor $C' \to C$
    solely consisting of inclusions of subsets on objects and morphisms.
    A subcategory $C'$ is said to be \emph{full} if the functor is fully faithful;
    it is said to be \emph{wide} if the functor is the identity on objects.
    \item A functor $F\colon C\to D$ is said to \emph{reflect identity}
    if for any morphism $f\colon x \to y$ in $C$ with $F(x) = F(y)$ and
    $F(f) = \idmor[F(x)]$, we have $x=y$ and $f = \idmor[x]$.
    \item A \emph{natural transformation} $\alpha\colon F\Rightarrow G$
    between functors $F,G\colon C\to D$, as always, consists of a morphism
    $\alpha_x\colon F(x) \to G(x)$ in $D$ for each object $x$ in $C$ satisfying
    suitable commutativity condition.
    \item The \emph{vertical composition} of natural transformations
    $\alpha\colon F\Rightarrow G$ and $\beta\colon G\Rightarrow H$, is denoted by
    $\beta\compos\alpha\colon F\Rightarrow H$. 
    The identity natural transformation of a functor $F$ is
    denoted by $\idmor[F]\colon F\Rightarrow F$.
    \item The \emph{functor category} $\functCat{C}{D}$ has functors $C\to D$
    as objects, natural transformations between them as morphisms, and
    the vertical composition of natural transformations as composition.
    \item The \emph{left whiskering} of a natural transformation
    $\alpha\colon F\Rightarrow G\colon C \to D$ by a functor $H\colon D \to D'$
    is denoted by $H\whiskl\alpha\colon H\compos F\Rightarrow H\compos G$.
    \item The \emph{right whiskering} of a natural transformation
    $\alpha\colon F\Rightarrow G\colon C \to D$ by a functor $H\colon C' \to C$
    is denoted by $\alpha\whiskr H\colon F\compos H\Rightarrow G\compos H$.
    \item Let $C$ and $D$ be a category. If $F\colon C' \to C$ is a functor,
    precomposition with $F$ and the left whiskering by $F$ define the
    \emph{precomposition functor} $\functCat{C}{D} \to \functCat{C'}{D}$,
    denoted by $\backwardinduce{F} = ({-} \compos F)$.
    Similarly for a functor $G\colon D \to D'$, the \emph{postcomposition functor}
    $\functCat{C}{D} \to \functCat{C}{D'}$ is denoted by
    $\forwardinduce{G} = (G \compos {-})$.
    \item The \emph{dual} of a category $C$ will be denoted by $\dualCat{C}$.
    \item The category $\functCat{\dualCat{C}}{D}$ of functors
    $\dualCat{C}\to D$ may be denoted by $\mtcPsh{C}{D}$. The objects of this
    category may also be called \emph{$D$-valued presheaves} over $C$ (or
    \emph{contravaraint functors} from $C$ to $D$, but we will not use this term).
    \item If $x,y$ are objects in a category $C$ and the hom-set $\HomOf[C]{x}{y}$ is a
    singleton, then we write $\uniqmor=\uniqmor_{xy}\colon x \to y$ for the unique morphism in $\HomOf[C]{x}{y}$.
    \item Let $C$ and $D$ be categories. The \emph{join} $C \star D$ is the category whose objects are
    $\Ob(C \star D) \coloneqq \Ob(C) \amalg \Ob(D) = (\set{0} \times \Ob(C)) \cup (\set{1} \times \Ob(D))$
    and whose morphisms are
    \[
    \HomOf[C \star D]{(i,x)}{(j,y)} = \begin{cases}
      \HomOf[C]{x}{y} & \text{if\ } i=j=0,\\
      \HomOf[D]{x}{y} & \text{if\ } i=j=1,\\
      \set{\uniqmor} & \text{if\ } i=0 \text{\ and\ } j=1,\\
      \emptyset & \text{if\ } i=1 \text{\ and\ } j=0.
    \end{cases}
    \]
    The composition in $C \star D$ is given by the compositions in $C$ and $D$ when possible,
    and by the unique possible choice when composing with $\uniqmor$.
    \item We use the word \emph{semicategory} for categories without the
    assumption of identities; in other words, a quiver with associative
    compositions. We can forget the structure of identities in a
    category to get a semicategory, and can freely adjoin identities to
    a semicategory to obtain a category.
    \item A finite category is a small category
    whose sets of objects its morphisms are both finite, i.e. bijective
    to some set of the form $\set{0,1,\dotsc,n-1}$.
    \item A preordered set $\left(P,\mathord{\le}\right)$
    is canonically considered
    as a category by defining $\Ob P$ to be $P$ regarded as a set and
    setting
    \[
    \HomOf[P]{x}{y}\coloneqq
    \begin{cases}
      \set{\uniqmor_{xy}} & \text{if\ } x\le y,\\
      \emptyset    & \text{otherwise.}
    \end{cases}
    \]
    \item If a category $C$ has every hom-set a subset of a singleton, then $C$ is identified
    with the preordered set of its objects. In particular, the notation for joins of categories
    and of preordered sets are identified. Also, the notation $\functCat{P}{Q}$ for the category
    of functors $P\to Q$ is identified as the preordered set of order-preserving functions $P\to Q$.
    \item We employ some examples of \emph{large categories}: the category of
    small categories $\mtcCat$, that of sets $\mtcSet$, that of (reflexively)
    partially ordered sets $\mtcPoset$, and the category of functors from a
    small category to these three categories. The most important
    example of such functor categories in this paper is the category of simplicial
    sets $\mtcSSet$. The term ``category,'' when used without qualification or reference to
    these examples, will always mean a \emph{small} category.
    \item If $C$ is a small category, the \emph{Yoneda embedding} $C \to \mtcPsh{C}{\mtcSet}$
    is denoted by $\mtynd$: $\mtynd(c) = \HomOf[C]{-}{c}$ for $c \in \Ob C$.
    \item The \emph{colimit} of the \emph{diagram} $F\colon D\to C$ is denoted by
    $\mtColim_{d\in D} F(d)$.
    \item As a specific case of colimits, if $F\colon \Lambda\to C$ and $G\colon \Lambda\to D$
    are functors, the \emph{pointwise left Kan extension of $G$ along $F$} is denoted by
    $(\mtLan_G F)(c) = \mtColim_{F(\lambda) \to c} G(\lambda)$.
    \item If $C$ is a category and $F\colon \dualCat{C} \to \mtcCat$ is a functor,
    the \emph{Grothendieck construction} of $F$ is denoted by $\int F \to C$. We will clarify
    the actual construction in its specialized form in \cref{def:groth-cat-ladder-all}.
    \item We use some terminology from enriched category theory for the specific case of
    enrichment over $\mtcPoset$. When we endow a $\mtcPoset$-enrichment on a category,
    we do not introduce notation for hom-objects; instead, we simply endow hom-sets with
    order structures. A notable usage of the term from the theory is the change of enriching base
    by a lax monoidal functor.
  \end{itemize}
\end{definition}

We also need to revisit the weak and strict definitions of (1-)localizations of categories.
For the strict definition, we refer the reader to \cite[Section 5.2]{Borceux_1994} and
\cite[Section I.1]{gabriel1967calculus}. The definition of weak localizations may be found
in \cite[\href{https://kerodon.net/tag/01M9}{Definition 01M9}]{kerodon}.

\begin{definition}\label{def:cat-1-localization}
  Let $F\colon C\to D$ be a functor between categories, and let $W\subseteq \Mor C$.
  Consider the following conditions:
  \begin{enumerate}
    \item \label{item:def-cat-1-localization:weq-send}%
    $F$ sends morphisms in $W$ to isomorphisms in $D$.
    \item \label{item:def-cat-1-localization:strict}%
    If a functor $G\colon C\to E$ sends morphisms in $W$ to isomorphisms in $E$,
    then there exists a unique functor $H\colon D\to E$ such that $H\compos F = G$.
    \item \label{item:def-cat-1-localization:weak-ex}%
    If a functor $G\colon C\to E$ sends morphisms in $W$ to isomorphisms in $E$,
    then there exists a functor $H\colon D\to E$ and a natural isomorphism
    $\theta\colon H\compos F \Rightarrow G$.
    \item \label{item:def-cat-1-localization:weak-uniq}%
    If $H, H'\colon D\to E$ are functors and
    $\theta\colon H\compos F \Rightarrow H'\compos F$ is a natural isomorphism,
    then there is a unique
    natural isomorphism $\hat\theta\colon H\Rightarrow H'$ such that
    $\hat\theta \whiskr F = \theta$.
    \item \label{item:def-cat-1-localization:ff}%
    For any category $E$, the precomposition functor
    $\backwardinduce{F}\colon \functCat{D}{E} \to \functCat{C}{E}$
    is fully faithful.
  \end{enumerate}
  We say that $F$ \emph{exhibits} $D$ as the \emph{strict (1-)localization of} $C$ \emph{at}
  (or \emph{with respect to}) $W$ if
  it satisfies the conditions \labelcref{item:def-cat-1-localization:weq-send,%
  item:def-cat-1-localization:strict}. We may instead say that the pair $(D,F)$ is a
  strict (1-)localization of $C$ at, or with respect to, $W$.
  We say $F$ \emph{exhibits} $D$ as the
  \emph{weak (1-)localization of} $C$ \emph{at} (or \emph{with respect to}) $W$
  if it satisfies the conditions \labelcref{item:def-cat-1-localization:weq-send,%
  item:def-cat-1-localization:weak-ex,item:def-cat-1-localization:weak-uniq},
  or equivalently the conditions \labelcref{item:def-cat-1-localization:weq-send,%
  item:def-cat-1-localization:weak-ex,item:def-cat-1-localization:ff}.
  In this case, we may instead say that the pair $(D,F)$ is a weak (1-)localization of
  $C$ at, or with respect to, $W$.
\end{definition}

\begin{remark}\label{rmk:cat-1-localization-weak-strict-equiv}
  A strict (1-)localization is a weak (1-)localization. More specifically, a pair
  $(D, F\colon C \to D)$ of a category and a functor is a weak 1-localization of $C$ 
  at $W\subseteq\Mor C$ if
  and only if there is a strict 1-localization $(D', F')$ of $C$ at $W$, an equivalence
  $H\colon D'\overset{\sim}{\to} D$ of categories,
  and a natural isomorphism $H\compos F' \cong F$. On the other hand,
  the pair $(D, F)$ is a strict 1-localization of $C$ at $W$ if and only if
  it is a weak 1-localization of $C$ at $W$ and $F\colon \Ob C \to \Ob D$ is a bijection.
\end{remark}

\subsection{Remarks on foundations}
\label{sec:cat-prel-foundations}

Although we need to enter the realm of classical axiomatic set theory in
\cref{sec:down-last-infty-loc-shape,sec:down-last-infty-loc} to enable
our discussion of $(\infty,1)$-category theory, we shall use the language of
constructive mathematics in \cref{sec:def-cat-down,sec:groth-cat-partial-order,sec:groth-cat-direct,%
sec:groth-cat-comparison,sec:down-last-localiz}.
We only place restrictions on the use of non-constructive principles,
and we do not additionally assume constructive principles that contradict
classical mathematics.

If $P$ is a proposition or a property, we say ``$P$ is decidable'' to mean that
``either $P$ is true or $P$ is false.'' If $P$ is a property of some object, we may say
``We have dichotomy of $x$ into $P$ and non-$P$'' to mean that ``$x$ either satisfies $P$ or does not.'' 

In order to formalize our results in \cref{sec:def-cat-down,sec:groth-cat-partial-order,%
sec:groth-cat-direct,sec:groth-cat-comparison,sec:down-last-localiz}, it should
be sufficient to have, as the foundation, Aczel's Constructive Zermelo-Fraenkel (CZF) set theory~\cite{aczel1978}.
CZF is an intuitionistic first-order theory with equality whose unique non-logical symbol is
the membership binary relation $\in$. Its axioms (and axiom schemata) are Axioms of Extensionality, Pairing, Union,
Empty Set and Infinity, and the Axiom Schemata of $\Delta_0$-Separation, Strong Collection, Subset Collection,
and $\in$-Induction. Note that CZF does not include the dedicate notion of classes, and the large categories
should be treated through the well-known trick of metatheoretically regarding properties of sets as classes.

In fact, since our discussion in \cref{sec:def-cat-down,sec:groth-cat-partial-order,sec:groth-cat-direct,%
sec:groth-cat-comparison,sec:down-last-localiz} is almost completely finitistic, the author strongly
conjectures that the foundation for the formalization of these sections can be weakened to
Intuitionistic Kripke-Platek set theory with the Axiom of Infinity ($\mathrm{IKP\omega}$)~\cite{lubarsky-ikp}.
The theory $\mathrm{IKP\omega}$ is axiomatized by Axioms of Extensionality, Pairing, Union, Empty Set and Infinity,
as well as the Axiom Schemata of $\Delta_0$-Separation, $\Delta_0$-Collection, and $\in$-Induction.
The author has not checked the details of the formalization, but the categories and functors constructed in
\cref{sec:def-cat-down,sec:groth-cat-partial-order,%
sec:groth-cat-direct,sec:groth-cat-comparison,sec:down-last-localiz} may be built using finite-domain
function sets, which are $\Delta_1$-constructible in $\mathrm{IKP\omega}$. Since IKP admits the Theorem Scheme of
Strong $\Sigma_1$-Collection, this provides a strong indication that the entire formalization can
indeed be carried out in $\mathrm{IKP\omega}$.

The author furthermore conjectures that we may use Intuitionistic Kripke-Platke set theory
without the Axiom of Infinity (IKP) by reformulating
and omitting some of the results in \cref{sec:groth-cat-direct,sec:down-last-localiz}. Specifically, the author suspects that
\cref{sec:def-cat-down,sec:groth-cat-partial-order,sec:groth-cat-comparison} remain valid if we allow large categories in our
construction, and that the following results admit modified proofs:
\begin{itemize}
  \item the categories $\int\catNerve^{{--},{+}}_{+}(C)$ from \cref{def:groth-cat-ladder-strict} and $\Down(C)$
    from \cref{def:down-c} are small categories;
  \item the two categories are direct (\cref{cor:groth-cat-direct,prop:down-c-direct});
  \item the two categories are finite if $C$ is finite (\cref{lem:groth-cat-finite,lem:down-c-finite});
  \item the functor $\Down(C) \to C$ is a 1-localization functor (\cref{thm:down-last-localiz}).
\end{itemize}
However, without the Axiom of Infinity, the arguments cannot be executed as smoothly as in the present paper,
since we then only have the \emph{class} of natural numbers. Consequently the following categories are not necessarily \emph{small},
meaning that we need to meta-theoretically treat them using classes:
\begin{itemize}
  \item the simplex category $\mtcSimpx$;
  \item the category $\int\catNerve(C)$ in \cref{def:groth-cat-ladder-all};
  \item the category $\int\catNerve^{{-},{+}}(C)$ in \cref{def:groth-cat-ladder-weak};
  \item the category $\Downstar(C)$ in \cref{def:down-c}.
\end{itemize}
The natural definitions of Reedy structures on these large categories, their directness, or their localizations
naturally involve univeral quantification over all classes satisfying certain conditions, which is not possible in
this set theory. In the present paper, \cref{prop:groth-cat-reedy} and several from \cref{sec:down-last-localiz}
show such properties of these categories. We might be able to reformulate them as meta-theorems or treat them
with workarounds, or might need to give them up. The author has not investigated this further.

As mentioned at the beginning of this subsection, to formalize the results in
\cref{sec:down-last-infty-loc-shape,sec:down-last-infty-loc}, we need the language of classical mathematics
and a stronger set-theoretic foundation. In these sections, we shall use the language of $(\infty,1)$-category theory
to prove that $\Down(C)\to C$ is an $(\infty,1)$-localization. Since the author is not aware of any constructive theory
of $(\infty,1)$-categories, we shall not attempt to formalize these sections in a constructive setting. One possible
choice of foundation for these sections is Zermelo-Fraenkel set theory with the Axiom of Choice (ZFC), but that does
not provide an internalized notion of classes, and therefore some lemmas (especially those in
\cref{sec:cat-prel-simplicial-quasicat}) would have to be treated as meta-theorems. Alternatively, it can be more
convenient to use a foundation that includes or allows a definition of classes, such as Morse-Kelley set theory (MK),
or Tarski-Grothendieck set theory (TG). See Shulman's survey~\cite{shulman2008settheorycategorytheory} for a discussion
on this topic.

As we treat direct and Reedy categories, we shall need to use the notion of well-foundedness.
We simply recall the definition here:

\begin{definition}[Well-foundedness]\label{def:well-founded}
  Let $X$ be a set and $<$ be a binary relation on $X$. We say that $<$ is a
  \emph{well-founded relation} on $X$ if $X$ is the only subset $Y$ of $X$ satisfying
  the following property: for any $x\in X$, if $\set{y \in X}[y < x] \subseteq Y$, then $x\in Y$.
  The set $X$ equipped with such $<$ is said to be a \emph{well-founded set}.
\end{definition}

However, the treatment of well-foundedness in this paper is abstract, and we shall only need
the following facts for this paper:

\begin{itemize}
  \item If $X$ is a set and $<$ is a well-founded relation on $X$, then there is no
  $<$-cycle: $x_0 < x_1 < \dotsb < x_n < x_0$ for some $n \ge 0$.
  \item The set $\metanats$ of natural numbers is well-founded under the usual order.
  \item Let $(X, {<}_X)$ and $(Y, {<}_Y)$ be sets with relations, and assume that the latter
    is well-founded. If there is a function $f\colon X\to Y$ such that $x_1 <_X x_2$ implies
    $f(x_1) <_Y f(x_2)$, then $<_X$ is well-founded.
  \item Let $(\Lambda, <)$ be a well-founded set, and $(X_\lambda, <_\lambda)$ be a well-founded set
    for each $\lambda \in \Lambda$. Endow the disjoint union 
    \[ 
      \coprod_{\lambda\in\Lambda} X_\lambda = \set{(\lambda, x)}[\lambda\in\Lambda, x\in X_\lambda]
    \]
    with the lexicographic relation, i.e., $(\lambda_1, x_1) < (\lambda_2, x_2)$ if and only if
    either $\lambda_1 < \lambda_2$ or $\lambda_1 = \lambda_2$ and $x_1 <_{\lambda_1} x_2$.
    Then the disjoint union is well-founded.
\end{itemize}

\subsection{Direct and Reedy categories}
\label{sec:cat-prel-direct-reedy}

In this subsection, we review the definitions of direct and Reedy categories constructively.
Textbook account with classic foundation can be found, for example,
in \cite[Sections 5.1 and 5.2]{hovey2007model}.

We first begin with the definition of direct categories.
The constructive definition of direct categories suffers from ambiguity; the correct categorification
of well-founded sets would be direct \emph{semicategories}, and it is not clear what is the vaild
extension of the notion to categories with identities. The following definition is the one supported by
Shulman's helpful answer~\cite{shulman-constructive-direct-categories} to my MathOverflow question:

\begin{definition}
  Let $C$ be a semicategory. For any $x,y\in \Ob C$, let us write $x\prec_C y$ if
  there exists morphism $x\to y$ in $C$. We say that $C$ is \emph{direct} if
  $\prec_C$ is a well-founded relation on $\Ob C$.
  A category is said to be \emph{direct} if it is isomorphic to the
  one obtained by freely adjoining identities to a direct semicategory.
  A semicategory or a category is said to be \emph{inverse} if its dual is direct.
\end{definition}

This definition is good in that it allows us to construct a presheaf over any direct category by
the most intuitive form of induction. The definition is also equivalent to the inductive definition
by Shulman in \cite{shulman2015reedycategoriesgeneralizations}.
Bartels, on the other hand, raised a remarkable opposition in comments to the same question:
the definition above does not include the canonical reflexively ordered class of ordinals, even if
we see the irreflexively ordered class of ordinals as well-founded.

The following lemma proves that this constructive definition is equivalent to that from the classical
mathematics:

\begin{lemma}\label{lem:direct-cat-def-eqvt-cond}
  Let $C$ be a category. For any $x,y\in \Ob C$, let us write $x\prec y$ if
  there is a non-identity morphism $x\to y$. Assume that $\prec$ is a well-founded
  relation on $\Ob C$. Then $C$ is a direct category if and only if any morphism in
  $C$ is either an identity or non-identity.
\end{lemma}
\begin{proof}
  The ``only if'' part is clear. For the ``if'' part, assume the dichotomy of morphisms
  into identities and non-identities. We need to show that $C$ is direct.
  Let the wide sub-semicategory $C'$ consist of all objects in $C$ and all
  non-identity morphisms in $C$. For well-definedness, we need to show that
  $C'$ is closed under compositions. Let $f\colon x\to y$ and $g\colon y\to z$
  be non-identity morphisms in $C$. If $g\compos f$ were an identity,
  then we would have $x\prec y \prec z=x$, which contradicts the well-foundedness of $\prec$.
  Therefore, $g\compos f$ is a non-identity, and $C'$ is a semicategory.
  By the assumption of the lemma, $C'$ is direct. It suffices to see that $C$ is
  isomorphic to the category obtained by freely adjoining identities to $C'$,
  which is immediate from our dichotomy assumption.
\end{proof}

\begin{remark}
  If $C$ is a finite category, the following three conditions are equivalent:
  \begin{enumerate}
  \item $C$ is a direct category.
  \item $C$ is an inverse category.
  \item Any endomorphisms and isomorphisms of $C$ are identities.
  \end{enumerate}
\end{remark}

We now move on to the definition of Reedy categories.

\begin{definition}[Reedy categories]
  \label{def:reedy-cat}
    Let $C$ be a category. A \emph{Reedy structure} on $C$ is a
    pair $\left(C_{-}, C_{+}\right)$ of wide subcategories of $C$ satisfying
    the following conditions:
    \begin{enumerate}
    \item Every morphism in $C$ can be uniquely factored as the
      composition of a morphism in $C_{-}$ followed by one in $C_{+}$:
      for each morphism $f: x \to y$ in $C$, there is a unique triple
      $\left(z, g, h\right)$ of $z \in \Ob(C)$, $g: x \to z$ in $C_{-}$,
      and $h: z \to y$ in $C_{+}$ with $f = h \compos g$.
      \label{item:def-reedy-cat:facto}
    \item Each of $C_{-}$ and $C_{+}$ is isomorphic to the category obtained
      by freely adjoining identities to some semicategory.
      \label{item:def-reedy-cat:id}
    \item Define a relation $<'$ on $\Ob(C)$ by setting $x <' y$ if and only
      if there is a non-identity $x \to y$ in $C_{+}$ or a
      non-identity $x \leftarrow y$ in $C_{-}$. Then $<'$ is a well-founded
      relation on $\Ob(C)$.
      \label{item:def-reedy-cat:wf}
    \end{enumerate}
    A \emph{Reedy category} $C=\left(C, C_{-}, C_{+}\right)$ is a category
    $C$ equipped with a Reedy structure $\left(C_{-}, C_{+}\right)$ on it.
  \end{definition}
  
  The definition may look different from the usual definition (for example, see
  Hovey~\cite{hovey2007model}),
  but it is in fact classically equivalent.
  With classical logic, the condition \eqref{item:def-reedy-cat:id}
  in \cref{def:reedy-cat} is derivable from
  \eqref{item:def-reedy-cat:wf}. To be constructively precise,
  see the following lemma:
  \begin{lemma}\label{lem:reedy-cat-iddec-equiv}
    Let $C$ be a category, and let $C_{-}, C_{+} \subseteq C$ be wide subcategories.
    Then we have the following:
    \begin{enumerate}
      \item Under the assumption \eqref{item:def-reedy-cat:wf} in
        \cref{def:reedy-cat}, the subcategory $C_{-}$ (resp.\@ $C_{+}$)
        is isomorphic to the category obtained by freely adjoining identities
        to some semicategory if and only if any morphism in $C_{-}$ (resp.\@
        $C_{+}$) is either an identity or non-identity.
      \item Under the assumption \eqref{item:def-reedy-cat:facto} in
        \cref{def:reedy-cat}, the dichotomy of morphisms into identities
        and non-identities in $C_{-}$ (resp.\@ $C_{+}$) is equivalent to the
        decidability of the membership of a morphism in $C_{+}$ (resp.\@ $C_{-}$).
    \end{enumerate}
  \end{lemma}
  \begin{proof}
    Follows from \cref{lem:direct-cat-def-eqvt-cond} and the fact that $C_{-}\cap C_{+}$
    consists only of identities.
  \end{proof}
  \begin{corollary}\label{cor:reedy-cat-iddec-equiv}
    If a category $C$ and a pair of wide subcategories $C_{-}, C_{+}\subseteq C$ satisfies
    \labelcref{item:def-reedy-cat:facto,item:def-reedy-cat:wf} in \cref{def:reedy-cat},
    then the followings are equivalent:
    \begin{enumerate}
      \item $(C, C_{-}, C_{+})$ is a Reedy category.
      \item The membership properties of a morphism in $C_{-}$ and in $C_{+}$ are decidable.
      \item Any morphism in $C$ is either an identity or non-identity.
      \item Any morphism in $C_{-}$ or $C_{+}$ is either an identity or non-identity.
    \end{enumerate}
  \end{corollary}
  \begin{proof}
    Follows from the previous lemma.
  \end{proof}
  The condition \eqref{item:def-reedy-cat:wf} in \cref{def:reedy-cat} is equivalently stated as
  the existence of a \emph{degree function}. A degree function of $C$,
  in the sense of Reedy categories, is by definition a function
  $F\colon \Ob(C) \to W$, where $W$ is some well-ordered set,
  such that all the non-identities $x \to y$ in $C_{+}$ and
  all the non-identities $x \leftarrow y$ in $C_{-}$ satisfy
  $F(x) < F(y)$. Some literature, including Hovey~\cite{hovey2007model},
  requires the equipment of a degree function, but its existence is enough
  for us.
  
  With this definition of Reedy categories,
  $\Mor\mleft(C_{-}\mright)$ and $\Mor\mleft(C_{+}\mright)$ are
  decidable subsets of $\Mor\mleft(C\mright)$. We also see that
  $C_{-}$ is an inverse category, and $C_{+}$ a direct category,
  as can be proven classically.
  
  In later sections, we will frequently draw a commutative
  diagram in a Reedy category. Within a diagram in a fixed Reedy category
  $C=\left(C, C_{-}, C_{+}\right)$, we depict morphisms in $C_{-}$ and
  $C_{+}$ with $\twoheadrightarrow$ and $\rightarrowtail$, respectively,
  to maintain the conciseness of these diagrams.
  Note that $\twoheadrightarrow$ and $\rightarrowtail$ will not necessarily
  stand for epimorphisms and monomorphisms when we consider a diagram
  in general Reedy categories, deviating from the general convention.
  
  We shall cite a useful lemma concerning Reedy categories:
  
  \begin{lemma}[Lemma 3.9 from \cite{bergner-rezk_2011}]\label{lem:reedy-idem-split}
    All idempotents in a Reedy category are split.
    For more precision, let $(C, C_{-}, C_{+})$ be a Reedy category and
    suppose that a morphism $\alpha\colon c \to c$ satisfies
    $\alpha \compos \alpha = \alpha$. Then there exists a unique pair of
    $\sigma \in \Mor C_{-}$ and $\delta \in \Mor C_{+}$ with
    $\delta \compos \sigma = \alpha$ and $\sigma \compos \delta = \idmor$.
  \end{lemma}
  \begin{proof}
    Uniqueness is clear. Let the following be the $(C_{-}, C_{+})$-factorizations:
    \begin{alignat}{3}
      \alpha &= \delta \compos \sigma,\,
      & \sigma &\in \Mor C_{-},\, &\delta &\in \Mor C_{+};
      \label{eq:lem-reedy-idem-split:1}\\
      \sigma \compos \delta &= \delta' \compos \sigma',\,
      & \sigma' &\in \Mor C_{-},\, &\delta' &\in \Mor C_{+}.
      \label{eq:lem-reedy-idem-split:2}
    \end{alignat}
    Since $\alpha$ is idempotent, we have:
    \[
      \sigma\compos\delta
      = \alpha
      = \alpha \compos \alpha
      = \sigma \compos \delta \compos \sigma \compos \delta
      = \sigma \compos \sigma' \compos \delta' \compos \delta.
    \]
    The uniqueness of $(C_{-}, C_{+})$-factorization, the
    directness of $C_{+}$, and the inverseness of $C_{-}$ together
    imply $\sigma' = \delta' = \idmor$, which reduces
    \cref{eq:lem-reedy-idem-split:2} to $\sigma \compos \delta = \idmor$.
    Combining this with \cref{eq:lem-reedy-idem-split:1}, we obtain the
    desired result.
  \end{proof}

As we declared in \cref{sec:introduction},
we shall fix a Reedy category $C=(C, C_{-}, C_{+})$
in \cref{sec:def-cat-down} and onward.

\subsection{The simplex category}
\label{sec:cat-prel-simpx}

We will frequently use the simplex category $\mtcSimpx$ in this paper. Here
we recall its definition and some of its properties. Remember we write $\mtcPoset$ for
the category of reflexively partially ordered sets (simply posets) and order-preserving
maps between them.

\begin{definition}[The simplex category]
  \label{def:simpx-cat}
  The category $\mtcSimpxAug$, called the \emph{augmented simplex category}, is the full subcategory of
  the $\mtcPoset$ spanned by finite ordinals: the linearly ordered sets
  $\deltaBra{n} \coloneqq n+1 = \set{0 < 1 < \dotsb < n}$ for integers $n \ge -1$.
  The \emph{simplex category} $\mtcSimpx$ is the full subcategory of $\mtcSimpxAug$
  spanned by inhabited finite ordinals: $\deltaBra{n}$ for $n \ge 0$.
\end{definition}

\begin{definition}[The canonical Reedy structure]
  \label{def:reedy-simpx-cat}
  We shall write $(\mtcSimpxAug)_{-}$ and $(\mtcSimpxAug)_{+}$ for the wide subcategories
  of $\mtcSimpxAug$ consisting of surjections
  (or \emph{degeneracies}) and injections (or \emph{face maps}), respectively. The pair
  $((\mtcSimpxAug)_{-}, (\mtcSimpxAug)_{+})$ is a Reedy structure on $\mtcSimpxAug$,
  which is considered as the canonical one on $\mtcSimpxAug$. The restrictions
  $\mtcSimpx_{-} \coloneqq (\mtcSimpxAug)_{-} \cap \mtcSimpx$ and $\mtcSimpx_{+} \coloneqq
  (\mtcSimpxAug)_{+} \cap \mtcSimpx$ form the canonical Reedy structure on $\mtcSimpx$.
\end{definition}

\begin{notation}
  We shall employ the conventional notation for morphisms in $\mtcSimpxAug$ (and hence in
  $\mtcSimpx$). 
  If $0\le k \le n$, the morphisms $\delta^n_k\colon \deltaBra{n-1} \rightarrowtail \deltaBra{n}$
  in $(\mtcSimpxAug)_{+}$ and $\sigma^n_k\colon\deltaBra{n+1} \twoheadrightarrow \deltaBra{n}$ in
  $(\mtcSimpxAug)_{-}$ are given by:
  \begin{align*}
    \delta^n_k(i) &\coloneqq \begin{cases}
      i &\text{if } 0\le i < k;\\
      i+1 &\text{if } k\le i < n;
    \end{cases}\\
    \sigma^n_k(i) &\coloneqq \begin{cases}
      i &\text{if } 0 \le i \le k;\\
      i-1 &\text{if } k < i \le n+1.
    \end{cases}
  \end{align*}
  Also, We add some non-conventional notation. Let $\deltaBra{n}\in\mtcSimpxAug$ and say we
  have the following subset:
  \[ S=\set{s_0 < s_1 < s_2 < \dotsc < s_{k-1}}\subseteq \deltaBra{n}. \]
  Then we write
  $\iota^n_S=\iota_S\colon \deltaBra{k-1} \rightarrowtail \deltaBra{n}$ for the morphism in
  $(\mtcSimpxAug)_{+}$ defined by $\iota_S(i) = s_i$.
  The particular case where $S$ is a singleton $S=\set{s}$ is also denoted by
  $\iota^n_s = \iota_s \coloneqq \iota^n_{\set{s}}\colon \deltaBra{0}\to\deltaBra{n}$.
\end{notation}

\subsection{Simplicial sets, quasicategories and \texorpdfstring{$(\infty,1)$}{(∞,1)}-categories}
\label{sec:cat-prel-simplicial-quasicat}

In this subsection, we partially review the theory of simplicial sets and quasicategories.
The notable textbook references for this are Lurie's \cite{kerodon,lurie2009higher}. Especially,
\cite{kerodon} will be repeatedly referred to in
\cref{sec:down-last-infty-loc-shape,sec:down-last-infty-loc}. The content of
this subsection will only be used in \cref{sec:down-last-infty-loc-shape,sec:down-last-infty-loc},
so we will use classical logic here.

\begin{definition}[Simplicial sets]
  A \emph{simplicial set} is a functor $\dualCat{\mtcSimpx} \to \mtcSet$. The category of
  simplicial sets, which is the functor category, is denoted by $\mtcSSet$. The morphisms in
  the category are called \emph{simplicial maps}. If $X\in\mtcSSet$, an $n$-\emph{simplex} of $X$
  is an element of $X_n$, which stands for $X$ evaluated at $\deltaBra{n}\in\mtcSimpx$. An
  $n$-simplex $\sigma$ is \emph{degenerate} if it is in the image of the map induced by some
  non-identity morphism in $\mtcSimpx_{-}$; it is \emph{non-degenerate} otherwise.
  A simplicial set is \emph{finite} if it has finitely many non-degenerate simplices across
  all dimensions.
\end{definition}

\begin{notation}[Simplices]
  We set ${\mtyndSpx}\coloneqq {\mtynd}\colon \mtcSimpx \to \mtcSSet$ for the special case of
  the Yoneda embedding.
  We may extend this functor to $\mtcSimpxAug$ by setting $\mtyndSpx$ to be:
  \[
  \begin{tikzcd}
    \mtcSimpxAug \ar[r, "{\mtynd}"] & \mtcPsh{\mtcSimpxAug}{\mtcSet} \ar[r, "\text{restrict}"]
    &[3em] {\mtcSSet.}
  \end{tikzcd}
  \]
  In other words, $\mtyndSpx\deltaBra{-1} \coloneqq \emptyset$.
\end{notation}

\begin{notation}[degenerate edge]
  Let $X$ be a simplicial set.  If $x \in X_0$
  is a vertex of $X$, its corresponding degenerate edge $\backwardinduce*{\sigma^0_0}(x)=X(\sigma^0_0)(x)$
  is also denoted by $\idmor[x]$.
\end{notation}

\begin{notation}[Boundaries and horns]
  The simplicial sets $\mtBdrySpx\deltaBra{n}\subseteq\mtyndSpx\deltaBra{n}$, for $n\ge 0$,
  and $\mtsHorn{n}{k}\subseteq\mtBdrySpx\deltaBra{n}$, for $0\le k\le n$, are defined as follows:
  \begin{align*}
    \mtBdrySpx\deltaBra{n}_l &\coloneqq \set{\alpha\colon\deltaBra{l}\to\deltaBra{n}}[\text{$\alpha$ is not surjective}],\\
    \mtsHorn{n}{k}_l &\coloneqq \set{\alpha\colon\deltaBra{l}\to\deltaBra{n}}%
    [ \set{0,\dotsc,l-1,l+1,\dotsc,n} \not\subseteq \operatorname{Im}\alpha].
  \end{align*}
\end{notation}

\begin{definition}[lifting property]
  Let $C$ be a (possibly large) category and $f\colon A \to B$, $g\colon X \to Y$ be morphisms in $C$.
  We say that the morphism $f$ has \emph{left lifting property} with respect to $g$, or
  that $g$ has \emph{right lifting property} with respect to $f$, if for any morphism $u\colon A \to X$
  and $v\colon B\to Y$ commutating the following outer square, there is a morphism that fits into the
  dotted line, commutating the two triangles:
  \[
  \begin{tikzcd}
    A \arrow[r, "u"] \arrow[d, "f"'] & X \arrow[d, "g"] \\
    B \arrow[r, "v"'] \arrow[ur, dotted] & Y
  \end{tikzcd}
  \]
\end{definition}

\begin{definition}[inner fibration, inner anodyne morphism, and quasicategories]
  Let $f\colon X\to Y$ be a morphism of simplicial sets. We say that $f$ is an \emph{inner fibration}
  if it has the right lifting property with respect to all horn inclusions $\mtsHorn{n}{k}\hookrightarrow\deltaBra{n}$,
  for all pair of integers $0 < k < n$. We say that a simplicial set $X$ is a \emph{quasicategory} if
  the canonical morphism $X \to \ast$ to the terminal object is an inner fibration. We say that a morphism
  $i\colon A \to B$ of simplicial sets is \emph{inner anodyne} if it has the left lifting property with respect
  to all inner fibrations.
\end{definition}

We refer the reader to \cite{kerodon,lurie2009higher} for:
\begin{itemize}
  \item \emph{equivalences} (or \emph{isomorphisms}) \emph{in a quasicategory};
  \item \emph{equivalences of quasicategories} and its generalization \emph{categorical equivalences};
  \item \emph{joins of simplicial sets}, denoted $X\star Y$.
\end{itemize}

\begin{notation}[{function complex}]
  The exponential object in the category of simplicial sets will be denoted by
  $\operatorname{Fun}(\bullet,\bullet)$, i.e.:
  \[
  \operatorname{Fun}(X,Y)_n \coloneqq \HomOf[\mtcSSet]{\mtyndSpx\deltaBra{n}\times X}{Y}.
  \]
\end{notation}

We now proceed to the definition of localization. In \cite{kerodon}, the following simplicial subset is denoted by
$\operatorname{Fun}(X[E^{-1}],Q)$. However, to avoid the misleading impression that it involves a simplicial set 
called $X[E^{-1}]$, we instead adopt a modified version of the notation used for the theory of marked simplicial sets 
in \cite{lurie2009higher}.

\begin{notation}[{complex of maps inverting some edges, cf.\@  \cite[Chapter~3]{lurie2009higher}}]
  Let $X\in\mtcSSet$ be a simplicial set, $E\subseteq X_1$ a set of edges, and
  $Q$ a quasicategory. We define a simplicial subset 
  $\operatorname{Fun}((X,E),Q^\natural)\subseteq\operatorname{Fun}(X,Q)$.
  The set $\operatorname{Fun}((X,E),Q^\natural)_0$ of vertices is defined to be the set of
  simplicial maps $X \to Q$ that send edges in $E$ to equivalences in $Q$. The set
  $\operatorname{Fun}((X,E),Q^\natural)_n$ of $n$-simplices is defined to be the set of
  all $n$-simplices $f\in\operatorname{Fun}(X,Q)_n$ whose vertices
  are in $\operatorname{Fun}((X,E),Q^\natural)_0$.
\end{notation}

This is the definition of localization that appears in our main theorem:

\begin{definition}[{localization \cite[\href{https://kerodon.net/tag/01MP}{Definition 01MP}]{kerodon}}]
  Let $f \colon X\to Y$ be a simplicial map, and $W\subseteq X_1$ be a collection of edges.
  We say that $f$ \emph{exhibits $Y$ as} an ($(\infty,1)$-)\emph{localization of $X$ at}, or \emph{with respect to},
  $W$, if for any quasicategory $Q$, the precomposition functor 
  $\backwardinduce{f}\colon \operatorname{Fun}(Y,Q) \to \operatorname{Fun}(X,Q)$ factors through
  $\operatorname{Fun}((X,W),Q^\natural) \subseteq \operatorname{Fun}(X,Q)$, and induces an equivalence
  of quasicategories 
  \[ \operatorname{Fun}(Y,Q) \to \operatorname{Fun}((X,W),Q^\natural). \]
  An ($(\infty,1)$-)\emph{localization map} is a simplicial map that exhibits the target as a localization
  of the source at some collection of edges.
\end{definition}

Of course, this notion is a refinement of the 1-localization: according to
\cite[\href{https://kerodon.net/tag/01MV}{Remark 01MV}]{kerodon},
if $f$ exhibits $Y$ as an $(\infty,1)$-localization of $X$ at $W\subseteq X_1$
and if $\operatorname{h}X$ and $\operatorname{h}Y$ denote the \emph{homotopy category} of $X$ and $Y$,
respectively, then the induced functor $\operatorname{h}f\colon \operatorname{h}X \to \operatorname{h}Y$
is a weak 1-localization of $\operatorname{h}X$ at the image of $W$ under the canonical map $X \to \nerve(\operatorname{h}X)$.

To facilitate our argument, we introduce a stronger notion of localization:

\begin{definition}[{universal localization \cite[\href{https://kerodon.net/tag/02M0}{Definition 02M0}, \href{https://kerodon.net/tag/02M1}{Proposition 02M1}]{kerodon}}]
  Let $f\colon X\to Y$ be a simplicial map. Then $f$ is said to be \emph{universally localizing} or
  to be a \emph{universal localization} if one, and hence both, of the following equivalent conditions hold:
  \begin{enumerate}
    \item Let $\varphi\colon \mtyndSpx\deltaBra{n} \to Y$ be any simplex of $Y$. Then the pullback
      $X\times_Y\mtyndSpx\deltaBra{n} \to \mtyndSpx\deltaBra{n}$ of $f$ along $\varphi$ is a localization
      map.
    \item Let $\varphi\colon Z\to Y$ be any simplicial map from an arbitrary simplicial set. Then the
      pullback $X\times_Y Z \to Z$ of $f$ along $\varphi$ exhibits $Z$ as an $(\infty,1)$-localization of $X\times_Y Z$
      at the preimage of the degeneracy edges of $Z$.
  \end{enumerate}
\end{definition}

We need some closure properties of localizations. The localizations and universal localizations
are closed under filtered
colimits~\cite[Propositions \href{https://kerodon.net/tag/01N6}{01N6} and \href{https://kerodon.net/tag/02M9}{02M9}]{kerodon}
and certain
pushouts~\cite[Propositions \href{https://kerodon.net/tag/01N7}{01N7} and \href{https://kerodon.net/tag/02MA}{02MA}]{kerodon}
in $\functCat{\deltaBra{1}}{(\mtcSSet)}$. Easier to verify is the closure under products:

\begin{lemma}[{closure under product of localizations, \cite[\href{https://kerodon.net/tag/02LV}{Proposition 02LV}]{kerodon}}]
  \label{lem:inf-loc-closure-product}
  Let $K$ be a simplicial set. Assume that a simplicial map $f\colon X\to Y$ exhibits $Y$ as a localization
  of $X$ at $W\subseteq X_1$. Then the induced map $f\times\idmor[K]\colon X\times K \to Y\times K$ on products
  exhibits $Y\times K$ as a localization of $X\times K$ at:
  \[
    W \times K_0 \coloneqq \set{(e, \idmor[x])}[e\in W, x\in K_0] \subseteq (X\times K)_1.
  \]
\end{lemma}
\begin{proof}
  See
  \cite[\href{https://kerodon.net/tag/02LV}{Proposition 02LV}]{kerodon}; this follows
  from the simple computation of mapping simplicial sets.
\end{proof}

\begin{lemma}[closure under product of universal localizations]\label{lem:univ-loc-closure-product}
  Let $f\colon X \to Y$ be a universally localizing simplicial map, and let $K$ be a simplicial set.
  Then the induced map $f\times\idmor[K]\colon X\times K \to Y\times K$ on products is universally localizing.
\end{lemma}
\begin{proof}
  If the left square in the following diagram is a pullback, then the whole rectangle is a pullback:
  \[
  \begin{tikzcd}
    Z \times_{Y\times K} (X\times K) \ar[r] \ar[d] & X\times K \ar[r] \ar[d, "f\times\idmor"']  & X \ar[d, "f"]\\
    Z \ar[r, "\text{arbitrary}"'] & Y\times K \ar[r] & Y
  \end{tikzcd}
  \]
\end{proof}

The proof of the closure under join is contributed by Maxime Ramzi as an answer \cite{Ramzi-localization-join-closure}
to my MathOverflow question: 

\begin{lemma}[closure under join of localizations, \cite{Ramzi-localization-join-closure}]\label{lem:inf-loc-closure-join}
  Let $K$ be a simplicial set. Assume that a simplicial map $f\colon X\to Y$ exhibits $Y$ as a localization
  of $X$ at $W\subseteq X_1$. Then the induced map $f\star\idmor[K]\colon X\star K \to Y\star K$ on joins
  exhibits $Y\star K$ as a localization of $X\star K$ at:
  \[
    W \subseteq X_1 = (X \star \emptyset)_1 \subseteq (X \star K)_1.
  \]
\end{lemma}
\begin{proof}
  \cite[\href{https://kerodon.net/tag/01HN}{Construction 01HN}]{kerodon}
  constructs the following commutative cube of simplicial sets:
  \[
    \begin{tikzcd}[column sep=0]
      X \times \mtBdrySpx\deltaBra{1} \times K \ar[rr, hook] \ar[dr, "f\times\idmor"'] \ar[dd]
      && X \times \mtyndSpx\deltaBra{1} \times K \ar [dd] \ar[dr, "f\times\idmor"] &\\
      & Y \times \mtBdrySpx\deltaBra{1} \times K \ar[rr, hook, crossing over]
      && Y \times \mtyndSpx\deltaBra{1} \times K \ar[dd] \\
      (X \times \set{0}) \amalg (\set{1} \times K) \ar[dr, "f\amalg{\idmor}"'] \ar[rr, hook]
      && X \star K \ar[dr, "f\star{\idmor}"] &\\
      & (Y \times \set{0}) \amalg (\set{1} \times K) \ar[rr, hook] \ar[from=uu, crossing over]
      && Y \star K
    \end{tikzcd}
  \]
  As \cite[\href{https://kerodon.net/tag/01HP}{Proposition 01HP}]{kerodon} states, the back and the front faces
  are \emph{categorical pushout} squares in the sense of \cite[\href{https://kerodon.net/tag/01F7}{Definition 01F7}]{kerodon},
  a.k.a.\@ homotopy pushout squares with respect to the Joyal model structure on $\mtcSSet$.
  Then \cite[\href{https://kerodon.net/tag/01N7}{Proposition 01N7}]{kerodon} implies that, in order to demonstrate
  our lemma, it suffices to check that the three maps that connect the pushout-defining spans are localization maps at
  appropriate sets of edges. We inspect the construction of the commutative cube to see that we need to show all of the
  following:
  \begin{itemize}
    \item $f\times{\idmor}\colon X\times\mtBdrySpx\deltaBra{1}\times K \to Y\times\mtBdrySpx\deltaBra{1}\times K$
      is a localization at $W\times (\mtBdrySpx\deltaBra{1} \times K)_0$;
    \item $f\times{\idmor}\colon X\times\mtyndSpx\deltaBra{1}\times K \to Y\times\mtyndSpx\deltaBra{1}\times K$
      is a localization at $W\times (\mtyndSpx\deltaBra{1} \times K)_0$;
    \item $f\amalg{\idmor}\colon (X\times\set{0})\amalg(\set{1}\times K) \to (Y\times\set{0})\amalg(\set{1}\times K)$
      is a localization at $W\times\set{0}$.
  \end{itemize}
  All of these are trivial or follow from \cref{lem:inf-loc-closure-product}.
\end{proof}

\begin{corollary}[closure under join of universal localizations]\label{cor:univ-loc-closure-join}
  Let $f\colon X \to Y$ be a universally localizing simplicial map, and let $K$ be a simplicial set.
  Then the induced map $f\star\idmor[K]\colon X\star K \to Y\star K$ on joins is universally localizing.
\end{corollary}
\begin{proof}
  Let $\varphi\colon \mtyndSpx\deltaBra{n} \to Y\star K$ be any simplex of $Y\star K$. Then there are
  objects $\deltaBra{m_0},\deltaBra{m_1}\in\Ob\mtcSimpxAug$ with $m_0+m_1+1=n$, and simplicial maps
  $\varphi_0\colon \mtyndSpx\deltaBra{m_0} \to Y$ and $\varphi_1\colon \mtyndSpx\deltaBra{m_1} \to K$
  with:
  \[
  \varphi = \varphi_0\star\varphi_1\colon \mtyndSpx\deltaBra{n} = \mtyndSpx\deltaBra{m_0}\star\mtyndSpx\deltaBra{m_1} \to Y\star K.
  \]
  Now, note that the following diagram is a pullback of simplicial sets:
  \[
  \begin{tikzcd}
    (X \times_Y \mtyndSpx\deltaBra{m_0}) \star \mtyndSpx\deltaBra{m_1} \ar[r] \ar[d]
    & X \star K \ar[d, "{f\star\idmor[K]}"]\\
    \mtyndSpx\deltaBra{m_0} \star \mtyndSpx\deltaBra{m_1} \ar[r, "\varphi_0\star\varphi_1"']
    & Y \star K
  \end{tikzcd}
  \]
  Therefore, it suffices to show that the left vertical map is a localization map. This is a direct consequence
  of \cref{lem:inf-loc-closure-join}.
\end{proof}

We remember few more notations:

\begin{notation}[nerve]
  Let $C$ be a category. The \emph{nerve} $\nerve(C)$ of $C$ is the simplicial set defined by
  \[
  \nerve(C)_n \coloneqq \HomOf[\mtcCat]{\deltaBra{n}}{C}.
  \]
  It is common to write just $C$ for the nerve of $C$, but we will use $\nerve(C)$ to facilitate
  the notation for colimits.
\end{notation}

\begin{notation}[skeleton]
  Given a simplicial set $X$, the $n$-\emph{skeleton} of $X$ is denoted by $\mtsSk{n} X$. To be specific,
  let $\mtcSimpx_{\le n}$ denote the full subcategory of $\mtcSimpx$ spanned by $\deltaBra{0},\dotsc,\deltaBra{n}$,
  and consider the restriction functor $\mathrm{tr}_n\colon\mtcSSet \to \mtcPsh{\mtcSimpx_{\le n}}{\mtcSet}$ and
  its left adjoint $\mathrm{sk}_n\colon \mtcPsh{\mtcSimpx_{\le n}}{\mtcSet} \to \mtcSSet$. Then we have
  ${\mtsSk{n}} = \mathrm{sk}_n \compos \mathrm{tr}_n$.
\end{notation}

To conclude this section, we list some references to lemmas that will be useful to prove that
a simplicial map is a monomorphism or an inner anodyne map.
We begin with the following definition, which is here merely for the formulation of
\cref{lem:wfs-pushout-left-class}:

\begin{definition}
  Let $C$ be a category, and $L, R \subseteq \Mor C$ be classes of morphisms. 
  We say that $(L, R)$ is a \emph{weak factorization system} on $C$ if the following conditions
  are satisfied:
  \begin{itemize}
    \item for any morphism $f\colon x \to y$ in $C$, there exist an object $z \in \Ob C$ and
      morphisms $l\colon x \to z$ in $L$ and $r\colon z \to y$ in $R$ such that $f = r \compos l$;
    \item $f\in \Mor C$ is in $L$ if and only if it has the right lifting property with respect to
      all morphisms in $R$;
    \item $f\in \Mor C$ is in $R$ if and only if it has the left lifting property with respect to
      all morphisms in $L$.
  \end{itemize}
\end{definition}

Notice that $L$ may be monomorphisms or inner anodyne maps in $\mtcSSet$.

\begin{lemma}\label{lem:wfs-pushout-left-class}
  Let $(L, R)$ be a weak factorization system on a bicomplete category $C$.
  Let the following be a commutative diagram in $C$:
  \[
  \begin{tikzcd}
    a \ar[d] & b \ar[l] \ar[r] \ar[d] & c \ar[d] \\
    x & y \ar[l] \ar[r] & z
  \end{tikzcd}
  \]
  If the morphisms $a \to x$, $b \to y$, and $y \cup_b c \to z$ induced 
  by the diagram are in $L$, then the morphism $a \cup_b c \to x \cup_y z$ between pushouts 
  belongs to $L$.
\end{lemma}
\begin{proof}
  We regard $C$ as a model category by taking $L$ as the class of cofibrations, $R$ as the class of fibrations,
  and all morphisms as weak equivalences. Consider the Reedy category $S=\set{0\twoheadleftarrow 1 \rightarrowtail 2}$,
  where the Reedy structure is shown by the shape of the arrows. As the commutative diagram in the statement
  may then be seen as a cofibration in the Reedy model category $\functCat{S}{C}$, the result follows if
  we can show that ${\mtColim}\colon \functCat{S}{C} \to C$ is a left Quillen functor. Since Proposition~15.10.2
  from Hirschhorn~\cite{hirschhorn2003model} demonstrates that $S$ has fibrant constants in the sense
  of Definition~15.10.1 from the same book, we can apply Theorem~15.10.8 from \cite{hirschhorn2003model} to conclude
  that ${\mtColim}$ is a left Quillen functor.
\end{proof}

The following lemma is useful and I believe it is a well-known folklore, but I could not find a reference.
Therefore I include a proof here. Note that the class $I$ of morphisms can be the class of injections or
the class of inner anodyne maps in $\mtcSSet$:

\begin{lemma}\label{lem:cospx-kan-ext-reedy-nat}
  Let $C$ be a cocomplete (large) category. Let $I \subseteq \Mor C$ be a class of morphisms,
  closed under transfinite compositions and pushouts. More specifically, assume that $I$ satisfies
  the following conditions:
  \begin{itemize}
    \item Let $\alpha > 0$ be an ordinal and $F\colon \alpha \to C$ be a functor
      with $F(\uniqmor_{\beta,\beta+1}) \in I$ for any $\beta$ with $\beta + 1 < \alpha$.
      If $F$ preserves all colimits of the diagrams $\beta\hookrightarrow\alpha$ for $\beta < \alpha$, then
      the canonical morphism $F(0) \to \mtColim_{\beta < \alpha} F(\beta)$ is in $I$.
    \item If $f\colon x \to y$ is a morphism in $I$ and $g\colon x\to z$ is a morphism in $C$,
      then the pushout $z \to y \cup_x z$ of $f$ along $g$ is in $I$.
  \end{itemize}
  Let $F,G\colon \mtcSimpx \to C$ be functors. Let the colimit-preserving functors
  $\hat{F},\hat{G}\colon \mtcSSet \to C$ be defined by the left Kan extensions of
  $F$ and $G$ along $\mtyndSpx$, as follows:
  \begin{align*}
    \hat{F}(X) &\coloneqq \mtColim_{\mtyndSpx\deltaBra{n} \to X} F\deltaBra{n},
    & \hat{G}(X) &\coloneqq \mtColim_{\mtyndSpx\deltaBra{n} \to X} G\deltaBra{n}.
  \end{align*}
  Assume that a natural transformation $\theta\colon F\Rightarrow G$ is Reedy in $I$, i.e.,
  for each $\deltaBra{n} \in \mtcSimpx$, the morphism
  $F\deltaBra{n} \cup_{\hat{F}(\mtBdrySpx\deltaBra{n})} \hat{G}(\mtBdrySpx\deltaBra{n})
  \to G\deltaBra{n}$ induced by $\theta$ belongs to $I$.
  Then, for any monomorphism $i\colon A\hookrightarrow B$ in $\mtcSSet$, the map
  $\hat{F}(B) \cup_{\hat{F}(A)} \hat{G}(A) \to \hat{G}(B)$ induced by $\theta$ and $i$ is in $I$.
  In particular, the consituent simplicial maps of
  the natural transformation $\hat{\theta}\colon \hat{F} \Rightarrow \hat{G}$
  induced by $\theta$ all belong to $I$.
\end{lemma}
\begin{proof}
  As is well known, the monomorphism $i\colon A \hookrightarrow B$ can be written as a transfinite
  composition of the pushouts of the maps of the form $\mtBdrySpx\deltaBra{n} \hookrightarrow \mtyndSpx\deltaBra{n}$.
  Let us take such a transfinite composition $X\colon \alpha + 1 \to \mtcSSet$.
  To be more explicit: the functor $X\colon \alpha + 1 \to \mtcSSet$ has
  the successor of an ordinal $\alpha$ as its domain,
  satisfies $X(0)=A$, $X(\alpha)=B$ and $X(\uniqmor_{0,\alpha})=i$,
  and preserves colimits of the diagrams $\beta\hookrightarrow\alpha\hookrightarrow\alpha+1$ for $\beta \le \alpha$.
  Also, for each $\beta < \alpha$, we have a pushout diagram in $\mtcSSet$ of the following form:
  \[
  \begin{tikzcd}
    \mtBdrySpx\deltaBra{n_\beta} \arrow[r, hook] \arrow[d] & \mtyndSpx\deltaBra{n_\beta} \arrow[d] \\
    X(\beta) \arrow[r, hook, "X(\uniqmor)"] & X(\beta+1)
  \end{tikzcd}
  \]

  Define a functor $Y\colon \alpha + 1 \to C$ by
  $Y(\beta) \coloneqq \hat{F}(B) \cup_{\hat{F}(X(\beta))} \hat{G}(X(\beta))$. Since $\hat{F}$ and $\hat{G}$
  preserve colimits and $X$ preserves colimits of the diagrams $\beta\hookrightarrow\alpha\hookrightarrow\alpha+1$
  for $\beta \le \alpha$, we see that $Y$ preserves colimits of such diagrams as well.
  We also see that $Y(\uniqmor_{0,\alpha})$ is the morphism $\hat{F}(B) \cup_{\hat{F}(A)} \hat{G}(A) \to \hat{G}(B)$
  in question. Therefore, it suffices to prove that $Y(\uniqmor)\colon Y(\beta) \to Y(\beta+1)$ is in $I$
  for each $\beta < \alpha$, for it and our assumption on $I$ together imply that $Y(\uniqmor_{0,\alpha})$ is in $I$.

  To see this remaining claim, let $\beta < \alpha$ and notice that there is the following pushout square in $C$:
  \[
  \begin{tikzcd}
    F\deltaBra{n_\beta} \cup_{\hat{F}(\mtBdrySpx\deltaBra{n_\beta})} \hat{G}(\mtBdrySpx\deltaBra{n_\beta})
    \ar[r] \ar[d] & G\deltaBra{n_\beta} \ar[d] \\
    Y(\beta) \ar[r] & Y(\beta+1)
  \end{tikzcd}
  \]
  Since the top vertical arrow is in $I$ by assumption, the bottom arrow is also in $I$, as $I$ is closed
  under pushouts. This completes the proof.
\end{proof}

In conbination with the previous lemma, the following lemmas are useful.
The proof can be found in Section~14.3 from Riehl~\cite{Riehl2014CatHtpy}:

\begin{definition}[{\cite[Section~14.3]{Riehl2014CatHtpy}}]
  Let $C$ be a small category and $F\colon \mtcSimpx \to \functCat{C}{\mtcSet}$ be a functor.
  Then $F$ is said to be \emph{unaugmentable} if the pullback of the two morphisms
  $F(\delta^1_0), F(\delta^1_1)\colon F\deltaBra{0} \rightrightarrows F\deltaBra{1}$
  is empty, i.e., the initial object.
\end{definition}

\begin{lemma}[{A part of \cite[Lemma~14.3.8]{Riehl2014CatHtpy}}]
  Let $C$ be a small category and $F\colon \mtcSimpx \to \functCat{C}{\mtcSet}$ be a functor.
  Then the unique natural transformation $\emptyset \Rightarrow F$ from the initial functor to $F$
  is a Reedy monomorphism if and only if $F$ is unaugmentable.
\end{lemma}

\begin{corollary}\label{cor:unaug-cospx-kan-ext-inj}
  Let $C$ be a small category and $F\colon \mtcSimpx \to \functCat{C}{\mtcSet}$ be a functor.
  Then the left Kan extension $\mtcSSet \to \functCat{C}{\mtcSet}$ of $F$ along the Yoneda embedding
  ${\mtyndSpx}\colon \mtcSimpx \to \mtcSSet$ preserves monomorphism if and only if $F$ is unaugmentable.
\end{corollary}

\begin{lemma}[{The rest of \cite[Lemma~14.3.8]{Riehl2014CatHtpy}}]
  Let $C$ be a small category and $F,G\colon \mtcSimpx \to \functCat{C}{\mtcSet}$ be
  unaugmentable functors. Then a natural transformation $F\Rightarrow G$ is a Reedy
  monomorphism if and only if it is a pointwise monomorphism.
\end{lemma}

\begin{corollary}\label{cor:unaug-cospx-kan-ext-nat-inj}
  Let $C$ be a small category and $F,G\colon \mtcSimpx \to \functCat{C}{\mtcSet}$ be
  unaugmentable functors. Let $\theta\colon F\Rightarrow G$ be a natural transformation.
  Let $\hat{F},\hat{G}\colon \mtcSSet \to \functCat{C}{\mtcSet}$ denote the left Kan extensions
  of $F$ and $G$ along the Yoneda embedding ${\mtyndSpx}\colon \mtcSimpx \to \mtcSSet$.
  Then the induced natural transformation $\hat\theta\colon \hat{F}\Rightarrow \hat{G}$
  is a pointwise monomorphism if and only if $\theta$ is so.
\end{corollary}

\section{Several categories related to a Reedy category}
\label{sec:def-cat-down}

We remind the reader that we will fix a Reedy category $C$ for the rest of the paper.
In this section, we construct the category $\Down(C)$ from $C$.
I, the author, appreciate Lumsdaine's helpful answer~\cite{Lumsdaine-groth-construction}
to my MathOverflow question, which has reformulated the categories constructed in this section
using Grothendieck construction and categorical nerves. This reformulation is explicitly included
in \cref{def:groth-cat-ladder-all} immediately below and further discussed in \cref{rmk:groth-cat-are-groth}.

Our construction begins with the following category, $\int\catNerve(C)$,
which appears to be standard, although we could not locate a convenient reference.

\begin{definition}[$\int\catNerve(C)$]
  \label{def:groth-cat-ladder-all}
  We shall denote by $\int \catNerve\mleft(C\mright)$ the total category
  of the Grothendieck construction of the categorical nerve
  \[
    \catNerve\mleft(C\mright)\colon
    \dualCat{\mtcSimpx} \to \mtcCat;\;
    \deltaBra{n} \mapsto \functCat{\deltaBra{n}}{C}
    \text{.}
  \]
  Explicitly, the category $\int\catNerve\mleft(C\mright)$ is described
  as follows:
  \begin{itemize}
  \item An object of $\int\catNerve\mleft(C\mright)$ is
    a pair $\left(\deltaBra{n}, X\right)$ of an object
    $\deltaBra{n} \in \Ob\mtcSimpx$ and a functor
    $X\colon \deltaBra{n} \to C$.
  \item A morphism $\left(\deltaBra{m}, X\right) \to
    \left(\deltaBra{n}, Y\right)$ is a pair
    $\left(\alpha, \theta\right)$ of a morphism
    $\alpha\colon  \deltaBra{m} \to \deltaBra{n}$ in $\mtcSimpx$
    and a natural transformation
    $\theta\colon X \Rightarrow Y \compos \alpha$.
  \item The composition of
    \begin{align*}
      (\alpha, \theta)
      &\colon (\deltaBra{l}, X) \to (\deltaBra{m}, Y)\text{, and}\\
      (\beta, \phi)
      &\colon (\deltaBra{m}, Y) \to (\deltaBra{n}, Z)
    \end{align*}
    is $(\beta\compos\alpha, (\phi\whiskr\alpha)\compos \theta)$.
  \end{itemize}
\end{definition}

In order to intuitively explain the category
$\int\catNerve\mleft(C\mright)$, we compare it to the
functor category $\functCat{\deltaBra{n}}{C}$.
Diagrammatically, the latter category consists of objects which are
vertical chains of $n$ composable morphisms in $C$, and morphisms which are
horizontal commutative ladders consisting of $n+1$ rungs connecting
between those objects, as depicted in:
\[
  \begin{tikzcd}
    \bullet \arrow[r] \arrow[d] & \bullet \arrow[d]\\
    \bullet \arrow[r] \arrow[d] & \bullet \arrow[d]\\
    \tightvdots \arrow[d] & \tightvdots \arrow[d] \\
    \bullet \arrow[r] & \bullet
  \end{tikzcd}
\]
By contrast, the category
$\int \catNerve\mleft(C\mright)$ has vertical chains of
any finite length as objects, and skew ladders as morphisms,
where by ``skew'' we mean that the rungs can be non-horizontal, as long as
each object in the source chain connects to exactly one rung and
the rungs do not cross midway, as shown in:
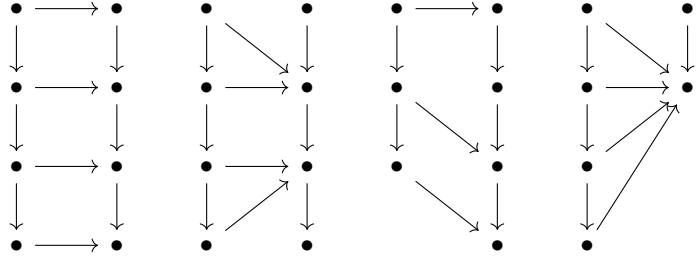
\[
  \begin{tikzcd}
    \bullet \arrow[d] \arrow[r] & \bullet \arrow[d] \\
    \bullet \arrow[d] \arrow[r] & \bullet \arrow[d] \\
    \bullet \arrow[d] \arrow[r] & \bullet \arrow[d] \\
    \bullet           \arrow[r] & \bullet
  \end{tikzcd}
  \qquad
  \begin{tikzcd}
    \bullet \arrow[d] \arrow[dr] & \bullet \arrow[d] \\
    \bullet \arrow[d] \arrow[r]  & \bullet \arrow[d] \\
    \bullet \arrow[d] \arrow[r]  & \bullet \arrow[d] \\
    \bullet           \arrow[ur] & \bullet
  \end{tikzcd}
  \qquad
  \begin{tikzcd}
    \bullet \arrow[d] \arrow[r]  & \bullet \arrow[d] \\
    \bullet \arrow[d] \arrow[rd] & \bullet \arrow[d] \\
    \bullet           \arrow[rd] & \bullet \arrow[d] \\
    & \bullet
  \end{tikzcd}
  \qquad
  \begin{tikzcd}
    \bullet \arrow[d] \arrow[rd]   & \bullet \arrow[d] \\
    \bullet \arrow[d] \arrow[r]    & \bullet \\
    \bullet \arrow[d] \arrow[ru]   & \\
    \bullet           \arrow[ruu]  &
  \end{tikzcd}
\]

Since our actual interest is in two subcategories of
$\int\catNerve(C)$, we proceed to their definitions.

\begin{definition}[$\int \catNerve^{{-},{+}}(C)$]
  \label{def:groth-cat-ladder-weak}
  The subcategory
  $\int \catNerve^{{-},{+}}\mleft(C\mright)$
  of $\int\catNerve(C)$
  is defined to consist of:
  \begin{itemize}
  \item the objects $\left(\deltaBra{n}, X\right)$ in which $X$ can be
    considered as a functor $\deltaBra{n} \to C_{-}$, and
  \item the morphisms
    $\mleft(\alpha, \theta\mright)\colon \left(\deltaBra{m}, X\right)
    \to \left(\deltaBra{n}, Y\right)$ such that each constituent morphism
    \[
      \theta_i\colon X(i) \to Y(\alpha(i))
    \]
    lies in $C_{+}$ for $i \in \deltaBra{m}=\set{0,1,\dotsc,m}$.
  \end{itemize}
\end{definition}

\begin{definition}[$\int \catNerve^{{--},{+}}_{+}(C)$]
  \label{def:groth-cat-ladder-strict}
  We will write
  $\int \catNerve^{{--},{+}}_{+}\mleft(C\mright)$
  for the subcategory of
  $\int \catNerve^{{-},{+}}\mleft(C\mright)$
  specified by:
  \begin{itemize}
  \item its objects are the objects
    $\left(\deltaBra{n}, X\right)\in
    \Ob\int \catNerve^{{-},{+}}\mleft(C\mright)$
    such that $X\colon \deltaBra{n} \to C_{-}$ is conservative,
    or equivalently reflects identities;
  \item its morphisms are the morphisms
    $\mleft(\alpha, \theta\mright)\colon \left(\deltaBra{m}, X\right)
    \to \left(\deltaBra{n}, Y\right)$ in
    $\int \catNerve^{{-},{+}}\mleft(C\mright)$
    such that $\alpha$ lies within $\mtcSimpx_{+}$.
  \end{itemize}
\end{definition}

Although we shall investigate $\int \catNerve^{{-},{+}}(C)$ and
$\int \catNerve^{{--},{+}}_{+}(C)$ mainly
using the explicit and combinatorial formulation as the categories of pairs,
it might be helpful to justify the symbols having categorical intent:

\begin{remark}[Lumsdaine~\cite{Lumsdaine-groth-construction}]
  \label{rmk:groth-cat-are-groth}
  The categories
  $\int \catNerve^{{-},{+}}(C)$ and
  $\int \catNerve^{{--},{+}}_{+}(C)$,
  as the symbols indicate,
  are Grothendieck constructions of functors
  \begin{align*}
    \catNerve^{{-},{+}}(C)
    &\colon \dualCat{\mtcSimpx} \to \mtcCat\text{, and}\\
    \catNerve^{{--},{+}}_{+}(C)
    &\colon \dualCat{\mtcSimpx_{+}} \to \mtcCat\text{.}
  \end{align*}
  Here, the functor $\catNerve^{{-},{+}}(C)$ is a subfunctor of
  $\catNerve(C)$, and
  $\catNerve^{{--},{+}}_{+}(C)$ is also a subfunctor of the
  restriction
  \[
    \begin{tikzcd}
      \dualCat{\mtcSimpx_{+}} \arrow[r, hook]
      & \dualCat{\mtcSimpx}
      \arrow[rr, "\catNerve^{{-},{+}}(C)"]
      && \mtcCat\text{.}
    \end{tikzcd}
  \]
  We do not need the actual definitions of the two functors
  for the argument in this paper, but they should be clear from
  \cref{def:groth-cat-ladder-weak,def:groth-cat-ladder-strict}.
  As the two categories in question are Grothendieck constructions, they
  come equipped with Grothendieck fibrations
  \begin{align*}
    \textstyle\int \catNerve^{{-},{+}}(C) &\to \mtcSimpx\text{, and}\\
    \textstyle\int \catNerve^{{--},{+}}_{+}(C) &\to \mtcSimpx_{+}\text{.}
  \end{align*}
  We will utilize these two functors later in this section, but we need not
  see them as a part of Grothendieck construction. They can
  be considered as the mere first projections, in accordance with the
  explicit constructions in \cref{def:groth-cat-ladder-weak,def:groth-cat-ladder-strict}.
\end{remark}

The purpose of the next lemma is another remark we need to make.
Although \cref{def:groth-cat-ladder-strict} imposes a condition
on the morphisms of
$\int \catNerve^{{-},{+}}\mleft(C\mright)$
to define those of
$\int \catNerve^{{--},{+}}_{+}\mleft(C\mright)$,
it turns out that these conditions are always inherently satisfied;
hence the condiction is logically redundant:

\begin{lemma} \label{lem:groth-cat-full-sub}
  The subcategory
  $\int \catNerve^{{--},{+}}_{+}\mleft(C\mright)
  \subseteq
  \int \catNerve^{{-},{+}}\mleft(C\mright)$
  is full.
  Explicitly, if
  \[
    (\deltaBra{m}, X), (\deltaBra{n}, Y) \in
    \Ob\textstyle\int \catNerve^{{--},{+}}_{+}\mleft(C\mright)
    \subseteq
    \Ob\int \catNerve^{{-},{+}}\mleft(C\mright)
  \]
  are objects and
  $
    (\alpha, \theta)\colon (\deltaBra{m}, X) \to (\deltaBra{n}, Y)
  $
  is a morphism in
  $\int \catNerve^{{-},{+}}\mleft(C\mright)$,
  then the morphism $\alpha\colon \deltaBra{m} \to \deltaBra{n}$ in
  $\mtcSimpx$ also belongs to $\mtcSimpx_{+}$,
  qualifying $(\alpha, \theta)$ as an arrow in
  $\int \catNerve^{{--},{+}}_{+}\mleft(C\mright)$.
\end{lemma}
\begin{proof}
  Set $(\alpha, \theta)\colon (\deltaBra{m}, X) \to (\deltaBra{n}, Y)$
  as assumed in the statement. It is sufficient to show
  $\alpha\colon \deltaBra{m} \to \deltaBra{n}$ is injective.
  Assume that $0 \le i \le j \le m$ satisfy $\alpha(i) = \alpha(j)$.
  The naturality of $\theta$ implies the commutativity of the following
  diagram:
  \[
    \begin{tikzcd}
      X(i) \arrow[d, two heads, "X(\uniqmor)"'] \arrow[r, tail, "\theta_i"]
      & Y(\alpha(i)) \arrow[d, equal] \\
      X(j) \arrow[r, tail, "\theta_j"']
      & Y(\alpha(j))
    \end{tikzcd}
  \]
  Here, $\uniqmor = \uniqmor_{ij}$ is the unique morphism $i \to j$ in
  $\deltaBra{m}$.
  In the diagram, the uniqueness of
  $\left(C_{-}, C_{+}\right)$-fac\-tor\-iza\-tion applies and we obtain
  $X(\uniqmor_{ij}) = \idmor$. Since $X$ is assumed to reflect identity
  by the definition of
  $\Ob\int \catNerve^{{--},{+}}_{+}\mleft(C\mright)$,
  we have $i=j$, as required.
\end{proof}

We proceed to the definitions of $\Down(C)$ and its inherently related
category $\Downstar(C)$.
For that purpose,
we first need to enrich the three previously defined categories over
the category $\mtcPoset$ of (reflexively) partially ordered sets,
since our desired categories are obtained by changing the enriching base
of the two $\mtcPoset$-categories $\int \catNerve^{{-},{+}}\mleft(C\mright)$
and $\int \catNerve^{{--},{+}}_{+}\mleft(C\mright)$.

Remember the canonical $\mtcPoset$-enrichment of $\mtcSimpx$
as a subcategory of the cartesian closed category $\mtcPoset$.
Explicitly stated, for any parallel pair of morphisms
$\alpha, \alpha'\colon \deltaBra{m} \to \deltaBra{n}$,
we consider that $\alpha \le \alpha'$ if and only if
we have $\alpha(i) \le \alpha'(i)$ for all
$i \in \deltaBra{m} = \set{0,1,\dotsc,m}$.
This enrichment induces the following $\mtcPoset$-enrichments:

\begin{definition}\label{def:groth-cat-poset-enrich}
  Let $C=\left(C,C_{-},C_{+}\right)$ be as above.
  Let
  \[\mleft(\alpha, \theta\mright),
    \mleft(\alpha', \theta'\mright)\colon \left(\deltaBra{m}, X\right)
    \to \left(\deltaBra{n}, Y\right)\]
  be a parallel pair of morphisms in any of the three categories
  $\int\catNerve\mleft(C\mright)$,
  $\int \catNerve^{{-},{+}}\mleft(C\mright)$, and
  $\int \catNerve^{{--},{+}}_{+}\mleft(C\mright)$.
  We say $\left(\alpha, \theta\right)\le\left(\alpha', \theta'\right)$
  if and only if it holds that $\alpha \le \alpha'$ and the following
  diagram of natural transformations is commutative:
  \[
    \begin{tikzcd}[row sep=tiny, column sep=small]
      &  &  & Y\compos\alpha
      \arrow[dd, Rightarrow, "Y\whiskl\uniqmor"] \\
      X \arrow[rrru, Rightarrow, "\theta"]
      \arrow[rrrd, Rightarrow, "\theta'"'] &  &  & \\
      &  &  & Y\compos\alpha'
    \end{tikzcd}
  \]
  Here, the unique natural transformation
  $\uniqmor\colon \alpha \Rightarrow \alpha'$
  corresponds to the inequality $\alpha \le \alpha'$ and
  $Y\whiskl\uniqmor$
  stands for its left whiskering with $Y$. With this partial order, we regard
  the three categories as $\mtcPoset$-enriched categories, and hence
  as 2-categories.
\end{definition}

Let $\pi_0\colon \mtcPoset \to \mtcSet$
denote the left adjoint functor of the discrete poset functor
$\mtcSet \to \mtcPoset$. The functor
$\pi_0$ sends a poset $P$ to the quotient of the set $P$ by
the symmetric transitive closure of the partial order on $P$.
Since this functor is strongly cartesian monoidal, the following
\cref{def:down-c}, of the category we want to construct in this section,
makes sense.

\begin{definition} \label{def:groth-cat-eqvce}
  Let
  \[
    (\alpha, \theta), (\beta, \phi)
    \colon (\deltaBra{m}, X) \to (\deltaBra{n}, Y)
  \]
  be a parallel pair of morphisms in any of the three categories
  $\int\catNerve\mleft(C\mright)$,
  $\int \catNerve^{{-},{+}}\mleft(C\mright)$, and
  $\int \catNerve^{{--},{+}}_{+}\mleft(C\mright)$.
  We say they are \emph{equivalent} and write
  $(\alpha, \theta) \sim (\beta, \phi)$
  if the two arrows are related by the symmetric transitive closure of
  the order on the hom-set
  $\HomOf{(\deltaBra{m}, X)}{(\deltaBra{n}, Y)}$.
\end{definition}

\begin{definition}[$\Down(C)$] \label{def:down-c}
  Let $C = \left(C,C_{-},C_{+}\right)$ be a Reedy category.
  The category $\Down(C)$ is defined to be the $\mtcSet$-category obtained by
  changing the enriching base of
  $\int \catNerve^{{--},{+}}_{+}\mleft(C\mright)$
  by the functor $\pi_0\colon \mtcPoset \to \mtcSet$.
  To be more explicit, $\Down(C)$ has the same set of objects as
  $\int \catNerve^{{--},{+}}_{+}\mleft(C\mright)$,
  and we declare the hom-sets of $\Down(C)$ to be
  the quotients of the original hom-sets by the equivalence we
  detailed above. Similarly, we define $\Downstar(C)$ to be the
  change of enriching base by $\pi_0$ of
  $\int \catNerve^{{-},{+}}\mleft(C\mright)$.
\end{definition}

There are two paths of construction of the category $\Down(C)$
from $\int\catNerve^{{-},{+}}(C)$. As we have done, we may first take the full subcategory
$\int\catNerve^{{--},{+}}_{+}(C)\subseteq\int\catNerve^{{-},{+}}(C)$, and
then take the quotient in hom-sets to obtain $\Down(C)$.
On the other hand, as will be seen in \cref{rem:groth-cat-functors},
we can also regard $\Down(C)$ as a full subcategory of $\Downstar(C)$,
which is obtained by taking quotients in hom-sets
of $\int\catNerve^{{-},{+}}(C)$. The two intermediate categories $\Downstar(C)$
and $\int\catNerve^{{--},{+}}_{+}(C)$ serve different purposes in our investigation
of $\Down(C)$. The quotient $\Downstar(C)$ is better in its categorical nature, and
used in the proof that $\Down(C) \to C$ is a localization. The full subcategory
$\int\catNerve^{{--},{+}}_{+}(C)$ is better in its combinatorial simplicity, and
used in the proof that $\Down(C)$ is finite and direct.

Now we have defined the category $\Down(C)$;
however, the combinatorial nature of the category $\Down(C)$ is
still unclear. For a better understanding of this combinatorics,
the next section
will be devoted to the analysis of the order and the equivalence
relation on the Hom-sets of
$\int\catNerve\mleft(C\mright)$,
$\int \catNerve^{{-},{+}}\mleft(C\mright)$, and
$\int \catNerve^{{--},{+}}_{+}\mleft(C\mright)$.

\section{The hom-posets of the categories of skew ladders}
\label{sec:groth-cat-partial-order}

The symmetric transitive closure of a partial order,
used in the definitions of $\Down(C)$ and $\Downstar(C)$,
is generally
tricky to discuss and compute with.
However, the simplicity of the order on the hom-sets on
$\int \catNerve^{{--},{+}}_{+}\mleft(C\mright)$,
shown in this section, makes it manageable.

We will prepare some notations for the subsequent discussion,
in which we will analyze the
$\mtcPoset$-enrichment of the three categories
$\int\catNerve\mleft(C\mright)$,
$\int \catNerve^{{-},{+}}\mleft(C\mright)$, and
$\int \catNerve^{{--},{+}}_{+}\mleft(C\mright)$.
During the analysis, $\Gamma$ will denote any of these three categories.

If $p\in P$ is an element of a poset $P$, we shall write
\[
  \upset[P]{p} = \upset{p}
  \coloneqq \set{x\in P}[p \le x]
\]
for the upward unbounded interval above $p$. If $p, q \in P$, then we set
\[
  \cintv[P]{p}{q} = \cintv{p}{q}
  \coloneqq \set{x\in P}[p \le x \le q]
\]
to be the closed bounded interval between $p$ and $q$.
Also, let us remind ourselves of the notation from \cref{def:cat-glossary}:
if $x \le y \in P$ are elements in a poset, we write
$\uniqmor = \uniqmor_{xy}\colon x \to y$ for the unique morphism in $P$ seen as a
category, as we already did in
\cref{def:groth-cat-poset-enrich} and in the proof of
\cref{lem:groth-cat-full-sub}.

We can consider the three categories
$\int\catNerve(C)$,
$\int \catNerve^{{-},{+}}(C)$, and
$\int \catNerve^{{--},{+}}_{+}(C)$
as Grothendieck constructions, as stated in \cref{def:groth-cat-ladder-all,rmk:groth-cat-are-groth};
thus they come equipped with Grothendieck fibrations
\begin{align*}
  \textstyle\int\catNerve\mleft(C\mright)
  & \to \mtcSimpx\text{,} \\
  \textstyle\int \catNerve^{{-},{+}}\mleft(C\mright)
  & \to \mtcSimpx\text{, and}\\
  \textstyle\int \catNerve^{{--},{+}}_{+}\mleft(C\mright)
  & \to \mtcSimpx_{+}\text{.}
\end{align*}
They are the first projections in terms of the explicit construction
in \cref{def:groth-cat-ladder-all,def:groth-cat-ladder-weak,%
  def:groth-cat-ladder-strict}.
By post-composing the inclusion
$\mtcSimpx_{+} \hookrightarrow \mtcSimpx$ as necessary, we get a functor
$\proj\colon \Gamma \to \mtcSimpx$.

In the following series of lemmas, we will analyze upward unbounded intervals
of the hom-posets of $\Gamma$, by comparing such intervals with the intervals
of the hom-posets of $\mtcSimpx$. The functor $\proj$ will be central in this analysis.

\begin{lemma} \label{lem:groth-cat-up-inj-refl}
  Let
  $\mleft(\alpha, \theta\mright)
    \colon \left(\deltaBra{m}, X\right)
    \to \left(\deltaBra{n}, Y\right)$
  be a morphism in $\Gamma$.
  Then the canonical order-preserving map
  \[
    \proj\colon
    \upset[{%
      \HomOf[\Gamma]{%
        \left(\deltaBra{m}, X\right)}{\left(\deltaBra{n}, Y\right)}
      }]{(\alpha,\theta)}
    \to \upset[{\HomOf[\mtcSimpx]{\deltaBra{m}}{\deltaBra{n}}}]{\alpha}
    \text{,}
  \]
  induced by the functor $\proj$ above, is injective and reflects order.
\end{lemma}
\begin{proof}
  Follows from the definition of the order.
\end{proof}

\begin{lemma} \label{lem:groth-cat-up-lift}
  We consider the cases
  $\Gamma=\int \catNerve^{{-},{+}}\mleft(C\mright)$
  and
  $\Gamma=
  \int \catNerve^{{--},{+}}_{+}\mleft(C\mright)$.
  Let
  \[\mleft(\alpha, \theta\mright)
    \colon \left(\deltaBra{m}, X\right)
    \to \left(\deltaBra{n}, Y\right)\]
  be a morphism in $\Gamma$.
  Let $\beta\colon \deltaBra{m} \to \deltaBra{n}$
  be a morphism of $\mtcSimpx$ with $\alpha \le \beta$.
  Then the following conditions are
  equivalent.
  \begin{enumerate}
  \item There exists a morphism in $\Gamma$ of the form
    \[
      \mleft(\beta, \phi\mright) \colon \left(\deltaBra{m}, X\right)
      \to \left(\deltaBra{n}, Y\right)
    \]
    satisfying $(\alpha, \theta) \le (\beta, \phi)$.
    \label{item:lem-groth-cat-up-lift:lift}
  \item For each $0\le i \le m$, the composite $\phi_i$ in the
    following commutative diagram belongs to $C_{+}$:
    \[
      \begin{tikzcd}
        X(i) \arrow[r, tail, "\theta_i"] \arrow[rd,"\phi_i"']
        & Y(\alpha(i)) \arrow[d, two heads, "Y(\uniqmor)"] \\
        & Y\mleft(\alpha'(i)\mright)
      \end{tikzcd}
    \]
    \label{item:lem-groth-cat-up-lift:cond}
  \end{enumerate}
  Additionally, if the equivalent conditions hold, the morphism
  $(\beta, \phi)$ stated to exist in
  \eqref{item:lem-groth-cat-up-lift:lift} is unique, and is
  specified by $\phi = \left(\phi_i\right)_{0 \le i \le m}$ using
  the notation from \eqref{item:lem-groth-cat-up-lift:cond}.
\end{lemma}
\begin{proof}
  The implication $\eqref{item:lem-groth-cat-up-lift:lift}
  \implies \eqref{item:lem-groth-cat-up-lift:cond}$ and
  the uniqueness part is a direct
  consequence of the definition of the order on
  $\HomOf[\Gamma]{\mleft(\deltaBra{m}, X\mright)}{%
    \mleft(\deltaBra{n}, Y\mright)}$.
  It just remains to demonstrate $\eqref{item:lem-groth-cat-up-lift:cond}
  \implies \eqref{item:lem-groth-cat-up-lift:lift}$.

  Assume the condition \eqref{item:lem-groth-cat-up-lift:cond}.
  Remember \cref{lem:groth-cat-full-sub};
  in order to construct a legitimate morphism $(\beta, \phi)$,
  it suffices to prove that
  $\phi = \left(\phi_0,\phi_1,\dotsc,\phi_m\right)$
  is a natural transformation.
   Let $0 \le i \le j \le m$.
  Consider the following diagram:
  \[
    \begin{tikzcd}
      X(i) \arrow[ddd, two heads, "X(\uniqmor)"']
      \arrow[rr, tail, "\phi_i"] \arrow[rd, tail, "\theta_i"]
      & & Y(\beta(i)) \arrow[ddd, two heads, "Y(\uniqmor)"] \\
      & Y(\alpha(i)) \arrow[ru, two heads, "Y(\uniqmor)"]
        \arrow[d, two heads, "Y(\uniqmor)"] & \\
      & Y(\alpha(j)) \arrow[rd, two heads, "Y(\uniqmor)"] & \\
      X(j) \arrow[ru, tail, "\theta_j"] \arrow[rr, tail, "\phi_j"']
      & & Y(\beta(j))
    \end{tikzcd}
  \]
  The commutativity of the two triangles follows
  from the definition of $\phi$. The naturality of $\theta$ commutes
  the left small quadrilateral. The functoriality of $Y$ ensures the
  commutativity of the right small quadrilateral. Combining them,
  we obtain the commutativity of the outer square, which is the
  naturality we desire.

  Now we have ensured that
  \[
    \mleft(\beta, \phi\mright) \colon \left(\deltaBra{m}, X\right)
    \to \left(\deltaBra{n}, Y\right)
  \]
  is a well-defined morphism in $\Gamma$, and the inequality
  $(\alpha, \theta) \le (\beta, \phi)$ is already included in
  our assumption. This establishes
  \eqref{item:lem-groth-cat-up-lift:lift}, and completes the proof.
\end{proof}

\begin{lemma} \label{lem:groth-cat-up-max}
  Let
  $\mleft(\alpha, \theta\mright)
    \colon \left(\deltaBra{m}, X\right)
    \to \left(\deltaBra{n}, Y\right)$
  be a morphism in $\Gamma$.
  Then there is the largest element in
  the upward unbounded interval
  \[
    \upset[{\HomOf[\Gamma]{(\deltaBra{m}, X)}{(\deltaBra{n}, Y)}}]
    {(\alpha, \theta)}.
  \]
\end{lemma}
\begin{proof}
  For $\Gamma=\int\catNerve\mleft(C\mright)$,
  the obvious surjectivity of
  \[
    \proj\colon
    \upset[{\HomOf[\int\catNerve(C)]{(\deltaBra{m}, X)}{(\deltaBra{n}, Y)}}]
    {(\alpha, \theta)}
    \to \upset[{\HomOf[\mtcSimpx]{\deltaBra{m}}{\deltaBra{n}}}]
    {\alpha}
  \]
  and \cref{lem:groth-cat-up-inj-refl} together imply that the
  unique inverse image of the maximum element of
  $\HomOf[\mtcSimpx]{\deltaBra{m}}{\deltaBra{n}}$ is the required
  largest element.

  We proceed to the cases of
  $\Gamma=\int\catNerve^{{-},{+}}(C)$
  and
  $\Gamma=\int\catNerve^{{--},{+}}_{+}(C)$.
  For each $i \in \deltaBra{m}$, define a natural number
  $\alpha'(i) \in \deltaBra{n} = \set{0,1,\dotsc,n}$ as the biggest
  $\alpha(i) \le k \le n$ for which the following composition 
  lies within $C_{+}$:
  \[
    \begin{tikzcd}
      X(i) \arrow[r, tail, "\theta_i"]
      & Y(\alpha(i)) \arrow[r, two heads, "Y(\uniqmor)"]
      & Y(k).
    \end{tikzcd}
  \]
  Remember that
  $\Mor(C_{+})$ is a decidable subset of
  $\Mor(C)$; therefore the maximum $k$ does exist.
  This defines a map $\alpha' \colon \deltaBra{m} \to \deltaBra{n}$ of
  sets.

  We need to verify that $\alpha'$ preserves order.
  Let $0 \le i \le j \le m$. We wish to show $\alpha'(i) \le \alpha'(j)$.
  If $\alpha'(i) < \alpha(j)$, there is nothing to prove; we consider
  the case $\alpha'(i) \ge \alpha(j)$.
  Consider the following commutative diagram:
  \[
    \begin{tikzcd}
      X(i) \arrow[r, tail, "\theta_i"]
      \arrow[d, two heads, "X\mleft(\uniqmor\mright)"']
      & Y\mleft(\alpha(i)\mright)
      \arrow[d, two heads, "Y\mleft(\uniqmor\mright)"]\\
      X(j) \arrow[r, tail, "\theta_{j}"]
      \arrow[d, two heads, "\sigma"']
      & Y\mleft(\alpha(j)\mright)
      \arrow[d, two heads, "Y\mleft(\uniqmor\mright)"]\\
      Z\arrow[r, tail] & Y\mleft(\alpha'(i)\mright)
    \end{tikzcd}
  \]
  where $X(j) \twoheadrightarrow Z \rightarrowtail Y(\alpha'(i))$ is the
  $\left(C_{-}, C_{+}\right)$-factorization of the composite
  $X(j) \rightarrowtail Y(\alpha(j)) \twoheadrightarrow Y(\alpha'(i))$.
  By the definition of $\alpha'(i)$,
  the morphism
  $X(i) \to Y\mleft(\alpha'(i)\mright)$
  in the diagram belongs to $C_{+}$;
  therefore the uniqueness of $\left(C_{-}, C_{+}\right)$-factorization
  implies that $X(i) \twoheadrightarrow Z$ is an identity.
  It follows from the
  inverseness of $C_{-}$ that
  $X\mleft(\uniqmor_{ij}\mright) = \sigma = \idmor$.
  Therefore $\alpha'(i) = \alpha'(j)$ by definition.

  Now we have proved that $\alpha'$ is a morphism of $\mtcSimpx$.
  By \cref{lem:groth-cat-up-lift}, there is a unique morphism in $\Gamma$
  of the form
  \[
    (\alpha',\theta')\colon
    \left(\deltaBra{m}, X\right)
    \to \left(\deltaBra{n}, Y\right)
  \]
  satisfying $(\alpha,\theta) \le (\alpha',\theta')$.
  By the construction of $\alpha'$ and \cref{lem:groth-cat-up-lift},
  we see that this is the maximum we desired.
\end{proof}

\begin{proposition} \label{prop:groth-cat-up-set-isom}
  Let
  $(\alpha, \theta)\colon \left(\deltaBra{m}, X\right)
    \to \left(\deltaBra{n}, Y\right)$
  be a morphism in $\Gamma$.
  Then the canonical order-preserving map
  \[
    \proj\colon\upset[{%
      \HomOf[\Gamma]{(\deltaBra{m}, X)}{(\deltaBra{n}, Y)}%
    }]{(\alpha,\theta)}
    \to \upset[{
      \HomOf[\mtcSimpx]{\deltaBra{m}}{\deltaBra{n}}}]{\alpha}
  \]
  has the largest element $\alpha'$ in its image. Furthermore,
  the map is an order isomorphism when restricted as
  \[
    \upset[{
      \HomOf[\Gamma]{(\deltaBra{m}, X)}{(\deltaBra{n}, Y)}
    }]{(\alpha, \theta)}
    \to \cintv[{\HomOf[\mtcSimpx]{\deltaBra{m}}{\deltaBra{n}}}]{\alpha}{\alpha'}
    \text{.}
  \]
\end{proposition}
\begin{proof}
  Since $\proj$ in the statement preserves order,
  the map sends the maximum element $(\alpha', \theta')\in
  \upset[{\HomOf[\Gamma]{(\deltaBra{m}, X)}{(\deltaBra{n}, Y)}}]{(\alpha, \theta)}$,
  shown to exist in \cref{lem:groth-cat-up-max},
  to the biggest element
  $\alpha'$ in the image of the map.
  Thanks to \cref{lem:groth-cat-up-inj-refl},
  all that is left to prove is the surjectivity of
  \[
    \proj\colon
    \upset[{\HomOf[\Gamma]{(\deltaBra{m}, X)}{(\deltaBra{n}, Y)}}]
    {(\alpha, \theta)}
    \to \cintv[{\HomOf[\mtcSimpx]{\deltaBra{m}}{\deltaBra{n}}}]{\alpha}{\alpha'}
    \text{,}
  \]
  which is obvious for $\Gamma=\int\catNerve(C)$.
  We prove this for
  $\Gamma=\int \catNerve^{{-},{+}}(C)$
  and
  $\Gamma=
  \int \catNerve^{{--},{+}}_{+}(C)$.

  Let $\beta\colon \deltaBra{m} \to \deltaBra{n}$ in $\mtcSimpx$
  satisfy $\alpha \le \beta \le \alpha'$. We would like to show that
  there is a morphism in $\Gamma$ having the form
  \[
    (\beta, \phi)\colon (\deltaBra{m}, X) \to (\deltaBra{n}, Y)
  \]
  such that $(\alpha, \theta) \le (\beta, \phi) \le (\alpha', \theta')$.
  By \cref{lem:groth-cat-up-lift,lem:groth-cat-up-inj-refl}, it suffices
  to prove, for each $0 \le i \le m$, that the composite
  \[
    \begin{tikzcd}
      X(i) \arrow[r, tail, "\theta_i"]
      & Y(\alpha(i)) \arrow[r, two heads, "Y(\uniqmor)"]
      & Y(\beta(i))
    \end{tikzcd}
  \]
  resides in $C_{+}$. Examine the following commutative diagram,
  formed by successively taking $(C_{-},C_{+})$-factorization twice:
  \[
    \begin{tikzcd}
      X(i) \arrow[r, tail, "\theta_i"] \arrow[d, two heads, "\exists!"']
      & Y(\alpha(i)) \arrow[d, two heads, "Y(\uniqmor)"] \\
      {}^{\exists!} Z_\beta \arrow[r, tail, "\exists!"]
        \arrow[d, two heads, "\exists!"']
      & Y(\beta(i)) \arrow[d, two heads, "Y(\uniqmor)"] \\
      {}^{\exists!} Z_{\alpha'} \arrow[r, tail, "\exists!"']
      & Y(\alpha'(i))
    \end{tikzcd}
  \]

  Since $(\alpha,\theta)$ is a morphism in $\Gamma$,
  $X(i) \to Y(\alpha'(i))$ belongs to $C_{+}$.
  From the uniqueness of $(C_{-},C_{+})$-factorization
  follows that $X(i) \twoheadrightarrow Z_{\alpha'}$ is an identity.
  The inverseness of $C_{-}$ implies that
  $X(i) \twoheadrightarrow Z_\beta$ is an identity; therefore
  we obtain that $X(i) \to Y(\beta(i))$ is in $C_{+}$.
\end{proof}

Thanks to \cref{prop:groth-cat-up-set-isom}, we see that upward unbounded intervals
of the hom-posets of $\Gamma$ has a good structure:

\begin{corollary}\label{cor:groth-cat-up-set-lattice}
  Let
  $(\alpha, \theta)\colon \left(\deltaBra{m}, X\right)
    \to \left(\deltaBra{n}, Y\right)$
  be a morphism in $\Gamma$.
  Then the upward unbounded interval
  $\upset[{\HomOf[\Gamma]{(\deltaBra{m}, X)}{(\deltaBra{n}, Y)}}]
  {(\alpha, \theta)}$
  enjoys the structure of a bounded distributive lattice.
\end{corollary}
\begin{proof}
  Let $\alpha'$ be as specified in \cref{prop:groth-cat-up-set-isom}.
  For
  $\cintv[{\HomOf[\mtcSimpx]{\deltaBra{m}}{\deltaBra{n}}}]{\alpha}{\alpha'}$
  is a bounded distributive lattice, \cref{prop:groth-cat-up-set-isom}
  directly verifies the corollary.
\end{proof}

Now we are ready to simplify the definition of
the equivalence relation on the hom-sets of $\Gamma$ in a way that is
easier to investigate.

\begin{proposition} \label{prop:groth-cat-eqvce-eqvcond}
  Let
  $(\alpha, \theta), (\alpha', \theta')
  \colon (\deltaBra{m}, X) \to (\deltaBra{n}, Y)$
  be a parallel pair of morphisms in $\Gamma$.
  Then the following conditions are equivalent:
  \begin{enumerate}
  \item $(\alpha, \theta) \sim (\alpha', \theta')$.
    \label{item:prop-groth-cat-eqvce-eqvcond:eqvce}
  \item There exists a common upper bound of $(\alpha, \theta)$ and
    $(\alpha',\theta')$.
    \label{item:prop-groth-cat-eqvce-eqvcond:upper}
  \item Let
    \begin{align*}
      (\tilde{\alpha},\tilde{\theta})
      & \in \upset[{%
        \HomOf[\Gamma]{(\deltaBra{m}, X)}{%
        (\deltaBra{n}, Y)}}]{(\alpha, \theta)}\text{, and}\\
      (\tilde{\alpha}', \tilde{\theta}')
      & \in \upset[{\HomOf[\Gamma]{(\deltaBra{m}, X)}{(\deltaBra{n}, Y)}}]
        {(\alpha', \theta')}
    \end{align*}
    be the largest elements,
    shown to exist in \cref{lem:groth-cat-up-max}. Then they
    coincide:
    $(\tilde{\alpha},\tilde{\theta})
    = (\tilde{\alpha}', \tilde{\theta}')$.
    \label{item:prop-groth-cat-eqvce-eqvcond:max}
  \end{enumerate}
\end{proposition}
\begin{proof}
  We verify each part of
  $\eqref{item:prop-groth-cat-eqvce-eqvcond:eqvce}
  \implies \eqref{item:prop-groth-cat-eqvce-eqvcond:upper}
  \implies \eqref{item:prop-groth-cat-eqvce-eqvcond:max}
  \implies \eqref{item:prop-groth-cat-eqvce-eqvcond:eqvce}$
  separately.

  $\eqref{item:prop-groth-cat-eqvce-eqvcond:eqvce}
  \implies \eqref{item:prop-groth-cat-eqvce-eqvcond:upper}$:
  By the definition of the equivalence, it suffices to show that
  the condition \eqref{item:prop-groth-cat-eqvce-eqvcond:upper}
  is a symmetric transitive relation
  when seen as a relation between $(\alpha, \theta)$ and
  $(\alpha', \theta')$. Symmetry is trivial. For transitivity, let
  \[
    (\alpha_0, \theta_0),(\alpha_1, \theta_1), (\alpha_2, \theta_2)
    \colon (\deltaBra{m}, X) \to (\deltaBra{n}, Y)
  \]
  be a parallel triple of arrows in $\Gamma$. Assume that
  \[ (\beta_1,\phi_1), (\beta_2,\phi_2)
    \colon (\deltaBra{m}, X) \to (\deltaBra{n}, Y) \]
  satisfy $(\alpha_i,\theta_i) \le (\beta_j, \phi_j)$ for $j = 1,2$ and
  $i = j-1, j$. Since $(\alpha_1, \theta_1)$ bounds both
  $(\beta_1, \phi_1)$ and $(\beta_2,\phi_2)$ from below,
  \cref{cor:groth-cat-up-set-lattice} allows us to take their join,
  i.e., their least upper bound, which we call $(\gamma, \psi)$.
  The transitivity of the order gives
  $(\alpha_i,\theta_i) \le (\gamma,\psi)$ for $i=0,2$, which verifies
  the transitivity of the relation
  \eqref{item:prop-groth-cat-eqvce-eqvcond:upper}, as desired.

  $\eqref{item:prop-groth-cat-eqvce-eqvcond:upper}\implies
  \eqref{item:prop-groth-cat-eqvce-eqvcond:max}$:
  If $(\beta, \phi)$ is a common upper bound of $(\alpha, \theta)$
  and $(\alpha', \theta')$, then the upward unbounded interval
  $\upset{(\beta,\phi)}$ is a common upper-closed subset of
  $\upset{(\alpha,\theta)}$ and
  $\upset{(\alpha',\theta')}$.
  Since $\upset{(\beta,\phi)}$ is inhabited,
  $\upset{(\beta,\phi)}$ shares any largest element of
  $\upset{(\alpha,\theta)}$ or
  $\upset{(\alpha',\theta')}$;
  this immediately implies
  \eqref{item:prop-groth-cat-eqvce-eqvcond:max}.

  $\eqref{item:prop-groth-cat-eqvce-eqvcond:max}\implies
  \eqref{item:prop-groth-cat-eqvce-eqvcond:eqvce}$:
  We see $(\alpha,\theta) \sim (\tilde{\alpha},\tilde{\theta})
  = (\tilde{\alpha}',\tilde{\theta}') \sim (\alpha',\theta')$.
\end{proof}

As a direct consequence of the previous proposition, we get:

\begin{corollary}
  \label{cor:groth-cat-eqvce-class-max}
  If $(\deltaBra{m}, X), (\deltaBra{n}, Y) \in \Ob\Gamma$ are objects,
  every equivalence class
  in $\HomOf[\Gamma]{(\deltaBra{m}, X)}{(\deltaBra{n}, Y)}$
  possesses a maximum with respect to the equipped order.
\end{corollary}

\section{Reedy structure, directness and finiteness}
\label{sec:groth-cat-direct}

In this section, we will construct a Reedy structure of
$\int \catNerve^{{-},{+}}(C)$,
and prove the directness of
$\int \catNerve^{{--},{+}}_{+}(C)$ and $\Down(C)$.

Set $\Gamma\coloneqq\int\catNerve^{{-},{+}}(C)$
until \cref{prop:groth-cat-reedy}, where
we would like to establish a Reedy structure of $\Gamma$.

\begin{definition}[$\int\catNerve^{{-},\set{\idmor}}_{-}(C)$,
  $\int\catNerve^{{-},{+}}_{+}(C)$]
  \label{def:groth-cat-reedy-def}
  Let wide subcategories
  \begin{align*}
    \Gamma_{-}&=\textstyle\int\catNerve^{{-},\set{\idmor}}_{-}(C)\text{, and}\\
    \Gamma_{+}&=\textstyle\int\catNerve^{{-},{+}}_{+}(C)
  \end{align*}
  of $\Gamma=\int\catNerve^{{-},{+}}(C)$ be defined as follows:
  \begin{align*}
    \Mor\Gamma_{-}
    &\coloneqq
      \set{(\sigma,\idmor[X\compos\sigma])
      \colon(\deltaBra{m},X\compos\sigma)\to(\deltaBra{n},X)}[%
      \begin{gathered}
        (\deltaBra{n},X)\in\Ob\Gamma,\\
        \sigma\colon\deltaBra{m}\to\deltaBra{n} \text{ in } \mtcSimpx_{-}
      \end{gathered}]\text{, and}\\
    \Mor\Gamma_{+}
    &\coloneqq
      \set{(\delta,\theta)
      \colon(\deltaBra{m},X)\to(\deltaBra{n},Y)\text{ in }\Gamma}[%
      \delta\colon\deltaBra{m}\to\deltaBra{n} \text{ in } \mtcSimpx_{+}]\text{.}
  \end{align*}
\end{definition}

\begin{lemma}\label{lem:groth-cat-reedy-factor}
  $\Gamma=\int\catNerve^{{-},{+}}(C)$ admits
  unique $(\Gamma_{-}, \Gamma_{+})$-factorization.
\end{lemma}
\begin{proof}
  A $(\Gamma_{-},\Gamma_{+})$-factorization of $(\alpha,\theta)$
  is a commutative diagram of the form \eqref{eq:lem-groth-cat-reedy-factor:comm} having a
  morphism from $\Gamma_{-}$ and then one from $\Gamma_{+}$ in the bottom row:
  \begin{equation}\label{eq:lem-groth-cat-reedy-factor:comm}
    \begin{tikzcd}[column sep=huge]
      (\deltaBra{m}, X) \arrow[rr, "{(\alpha,\theta)}"]
      \arrow[d,equal]
      && (\deltaBra{n}, Y) \arrow[d,equal] \\
      {(\deltaBra{m}, X'\compos\sigma)}
      \arrow[r, "{(\sigma, \idmor[X'\compos\sigma])}"']
      & {(\deltaBra{l}, X')} \arrow[r, "{(\delta,\theta')}"']
      & (\deltaBra{n}, Y)
    \end{tikzcd}
  \end{equation}

  We break down the condition that the diagram
  \eqref{eq:lem-groth-cat-reedy-factor:comm} is a $(\Gamma_{-},\Gamma_{+})$-factorization.
  The membership conditions $(\deltaBra{l}, X') \in \Ob\Gamma$,
  $(\sigma, \idmor[X'\compos\sigma]) \in \Mor\Gamma_{-}$, and
  $(\delta, \theta') \in \Mor\Gamma_{+}$ are equivalent to the following set of
  conditions:
  \begin{align}
    \deltaBra{l}
    &\in \Ob\mtcSimpx,
      \label{eq:lem-groth-cat-reedy-factor:l}\\
    \sigma
    &\in \HomOf[\mtcSimpx_{-}]{\deltaBra{m}}{\deltaBra{l}},
      \label{eq:lem-groth-cat-reedy-factor:delta}\\
    \delta
    &\in \HomOf[\mtcSimpx_{+}]{\deltaBra{l}}{\deltaBra{n}},
      \label{eq:lem-groth-cat-reedy-factor:sigma}\\
    X'
    &\colon \deltaBra{l} \to C_{-} \text{ a functor,}
      \label{eq:lem-groth-cat-reedy-factor:xpr}\\
    \theta'
    &\colon X' \Rightarrow Y\compos\delta\colon\deltaBra{l} \to C
      \text{ a natural transformation,}
      \label{eq:lem-groth-cat-reedy-factor:thetapr}\\
    \theta'_i
    &\in \HomOf[C_{+}]{X'(i)}{Y(\delta(i))},\, \forall i \in \deltaBra{l}.
      \label{eq:lem-groth-cat-reedy-factor:plus}
  \end{align}
  For the commutativity of the diagram \eqref{eq:lem-groth-cat-reedy-factor:comm},
  it is necessary and sufficient to have all of
  \labelcref{eq:lem-groth-cat-reedy-factor:idx,eq:lem-groth-cat-reedy-factor:obj,%
    eq:lem-groth-cat-reedy-factor:tr}:
  \begin{alignat}{2}
    \alpha &= \delta \compos \sigma
    &&\colon \deltaBra{m} \to \deltaBra{n},
       \label{eq:lem-groth-cat-reedy-factor:idx}\\
    X &= X' \compos \sigma
    &&\colon \deltaBra{m} \to C_{-},
       \label{eq:lem-groth-cat-reedy-factor:obj}\\
    \theta &= \theta'\whiskr\sigma
    &&\colon X=X' \compos \sigma \Rightarrow Y\compos\alpha =
       Y\compos\delta\compos\sigma.
       \label{eq:lem-groth-cat-reedy-factor:tr}
  \end{alignat}

  Since $(\mtcSimpx, \mtcSimpx_{-}, \mtcSimpx_{+})$ is
  a Reedy category, there exists a unique triple $(\deltaBra{l}, \sigma, \delta)$
  satisfying
  \labelcref{eq:lem-groth-cat-reedy-factor:l,eq:lem-groth-cat-reedy-factor:delta,%
    eq:lem-groth-cat-reedy-factor:sigma,eq:lem-groth-cat-reedy-factor:idx}.
  We fix $(\deltaBra{l}, \sigma, \delta)$ as
  such.

  Now it suffices to establish the unique existence of a pair
  $(X',\theta')$ enjoying the conditions
  \labelcref{eq:lem-groth-cat-reedy-factor:xpr,eq:lem-groth-cat-reedy-factor:thetapr,%
    eq:lem-groth-cat-reedy-factor:plus,eq:lem-groth-cat-reedy-factor:obj,%
    eq:lem-groth-cat-reedy-factor:tr}.
  Remembering that every morphism in $\mtcSimpx_{-}$ is a split epimorphism in
  $\mtcSimpx$ and that the $\mtcPoset$-enriched category $\mtcSimpx$
  is in fact enriched over finite bounded lattices,
  take the largest section $\epsilon\colon \deltaBra{l} \to \deltaBra{m}$ of $\sigma$.
  By precomposing $\epsilon$ on the both sides of
  \cref{eq:lem-groth-cat-reedy-factor:obj,eq:lem-groth-cat-reedy-factor:tr}, we get:
  \begin{alignat}{2}
    X\compos\epsilon &= X'
    &&\colon \deltaBra{l} \to C_{-},
       \label{eq:lem-groth-cat-reedy-factor:objeps}\\
    \theta\whiskr\epsilon &= \theta'
    &&\colon X'\Rightarrow Y\compos\delta.
       \label{eq:lem-groth-cat-reedy-factor:treps}
  \end{alignat}
  From these equations, the uniqueness of $(X', \theta')$ is immediate.

  It just remains to show that the pair $(X', \theta')$ defined by
  \cref{eq:lem-groth-cat-reedy-factor:objeps,eq:lem-groth-cat-reedy-factor:treps}
  satisfies the conditions
  \labelcref{eq:lem-groth-cat-reedy-factor:xpr,eq:lem-groth-cat-reedy-factor:thetapr,%
    eq:lem-groth-cat-reedy-factor:plus,eq:lem-groth-cat-reedy-factor:obj,%
    eq:lem-groth-cat-reedy-factor:tr}, of which
  \labelcref{eq:lem-groth-cat-reedy-factor:xpr,eq:lem-groth-cat-reedy-factor:thetapr,%
    eq:lem-groth-cat-reedy-factor:plus} are obvious.
  We wish to show \cref{eq:lem-groth-cat-reedy-factor:obj,%
    eq:lem-groth-cat-reedy-factor:tr}.
  The maximality of $\epsilon$ yields the inequality
  $\idmor[\deltaBra{m}]\le\epsilon\compos\sigma\colon \deltaBra{m} \to \deltaBra{m}$;
  we have a natural transformation
  $\uniqmor = \uniqmor_{\idmor, \epsilon\compos\sigma}\colon
  \idmor[\deltaBra{m}]\Rightarrow\epsilon\compos\sigma$.
  Notice:
  \begin{equation}
    \label{eq:lem-groth-cat-reedy-factor:alphaeps}
    \alpha\whiskl\uniqmor_{\idmor, \epsilon\compos\sigma}
    = \uniqmor = \idmor\colon
    \alpha\Rightarrow\alpha\compos\epsilon\compos\sigma
    = \delta \compos\sigma\compos\epsilon\compos\sigma = \alpha.
  \end{equation}
  The interchange law of
  $\theta\colon X \Rightarrow Y\compos\alpha$ and
  $\uniqmor\colon \idmor[\deltaBra{m}] \Rightarrow \epsilon\compos\sigma$
  may be written, combining
  \cref{eq:lem-groth-cat-reedy-factor:objeps,eq:lem-groth-cat-reedy-factor:treps,%
    eq:lem-groth-cat-reedy-factor:alphaeps}, as the following
  commutative square of natural transformations:
  \begin{equation}\label{eq:lem-groth-cat-reedy-factor:interchange1}
    \begin{tikzcd}[column sep=huge]
      X \arrow[r, Rightarrow, "\theta"]
      \arrow[d, Rightarrow, "X\whiskl\uniqmor"']
      & Y\compos\alpha
      \arrow[d, equal] \\
      {X' \compos \sigma}
      \arrow[r, Rightarrow, "{\theta'\whiskr\sigma}"']
      & Y \compos\alpha
    \end{tikzcd}
  \end{equation}
  In the diagram \eqref{eq:lem-groth-cat-reedy-factor:interchange1},
  every constituent morphism of
  $X\whiskl\uniqmor\colon X \Rightarrow X'\compos\sigma$ belongs
  to $C_{-}$, and $C_{+}$ has
  all the component morphisms of natural transformations drawn horizontally in
  the diagram.
  Therefore the uniqueness of the $(C_{-},C_{+})$-factorization applies and
  we get \cref{eq:lem-groth-cat-reedy-factor:obj,eq:lem-groth-cat-reedy-factor:tr},
  which concludes the proof.
\end{proof}

\begin{proposition}\label{prop:groth-cat-reedy}
  The pair of wide subcategories $(\Gamma_{-}, \Gamma_{+})$ in
  \cref{def:groth-cat-reedy-def}
  is a Reedy structure on $\Gamma=\int\catNerve^{{-},{+}}(C)$.
\end{proposition}
\begin{proof}
  In light of \cref{lem:groth-cat-reedy-factor}, it just remains to
  show the conditions \labelcref{item:def-reedy-cat:id,item:def-reedy-cat:wf}
  from \cref{def:reedy-cat}.
  We first verify that the property of morphisms being an identity is decidable.
  For a morphism
  $(\alpha, \theta)\colon (\deltaBra{m},X)\to(\deltaBra{n},Y)$ to be identity,
  it is necessary and sufficient to have identities as the components
  $\alpha, \theta_0, \theta_1, \dotsc, \theta_m$. Since we have dichotomy of
  each of these morphisms into identities and non-identies, we obtain the
  desired decidablity.

  Now it suffices to establish the well-foundedness of $\Ob\Gamma$.
  Since $C$ is a Reedy category, $C_{+}$ is a direct category. That is,
  the relation $<_{+}$ defined by the following is a well-founded
  relation on $\Ob C = \Ob C_{+}$:
  \[
    \text{for $x,y\in\Ob C$, }\; x <_{+} y \coloniff
    \exists f\colon x \to y \text{ in }C_{+}, f \neq \idmor[x].
  \]
  Consider the following set:
  \[
    (\Ob C)^{<\omega} = \coprod_{n \in \metanats}(\Ob C)^n,
  \]
  which consists of finitely long sequences
  $(x_0,x_1,\dotsc,x_{n-1})$ of objects of $C$.
  For each pair of elements
  $\boldsymbol{x}=(x_0,x_1,\dotsc,x_{n-1}),
  \boldsymbol{y}=(y_0,y_1,\dotsc,y_{m-1})\in (\Ob C)^{<\omega}$,
  let us say $\boldsymbol{x} < \boldsymbol{y}$ if and only if we
  have either $m < n$ or all of the followings hold:
  \begin{itemize}
  \item $m = n$,
  \item for some $0 \le i < m$, $x_i <_{+} y_i$, and
  \item for all $0 \le i < m$, $x_i = y_i$ or $x_i <_{+} y_i$.
  \end{itemize}
  The well-foundedness of this relation $<$ on $(\Ob C)^{<\omega}$ follows
  from those of $\metanats$ and $\Ob C_{+}$.
  Now, consider the following map of sets:
  \[
    F\colon \Ob \Gamma \to (\Ob C)^{<\omega};\;
    (\deltaBra{n}, \theta) \mapsto (\theta_0,\theta_1,\dotsc,\theta_n).
  \]
  Now we can readily confirm that any
  non-identity $\boldsymbol{\alpha} \to \boldsymbol{\beta}$ in $\Gamma_{+}$
  satisfies $F(\boldsymbol{\alpha}) < F(\boldsymbol{\beta})$, and that
  any non-identity $\boldsymbol{\alpha} \to \boldsymbol{\beta}$ in $\Gamma_{-}$
  satisfies $F(\boldsymbol{\alpha}) > F(\boldsymbol{\beta})$. This implies
  the well-foundedness $\Ob\Gamma$ with respect to the relation specified
  in \eqref{item:def-reedy-cat:wf} from \cref{def:reedy-cat}.
\end{proof}

\begin{remark}
  Let $\Gamma=\int\catNerve^{{-},{\mathrm{all}}}(C)$ be the full subcategory
  of $\int\catNerve(C)$ spanned by those objects $(\deltaBra{n}, X)$ for which
  the constituent functor $X\colon \deltaBra{n} \to C$ factors through $C_{-}$.
  This category $\Gamma$ admits a Reedy structure $(\Gamma_{-}, \Gamma_{+})$,
  and the Reedy structure on $\int\catNerve^{{-},{+}}(C)$ in \cref{prop:groth-cat-reedy}
  is the restriction of this Reedy structure. Here, letting $\pm$ denote either
  $-$ or $+$, a morphism
  \[
  (\alpha, \theta)\colon (\deltaBra{m},X)\to(\deltaBra{n},Y)
  \]
  belongs to $\Gamma_{\pm} = \int\catNerve^{{-},{\pm}}_{\pm}(C)$ if and only if
  $\alpha\in\mtcSimpx_{\pm}$ and $\theta_i\in C_{\pm}$ for every $i\in\deltaBra{m}$.
  The proof of this statement is similar to the proof in
  \cref{lem:groth-cat-reedy-factor}, but it calls for a slightly more
  careful construction of the Reedy factorization.
\end{remark}

From the Reedy structure of $\int \catNerve^{{-},{+}}(C)$,
the directness of $\int \catNerve^{{--},{+}}_{+}(C)$ may be deduced.

\begin{corollary} \label{cor:groth-cat-direct}
  The category $\int \catNerve^{{--},{+}}_{+}(C)$ is direct.
\end{corollary}
\begin{proof}
  By \cref{prop:groth-cat-reedy}, the cagegory
  $\int \catNerve^{{-},{+}}_{+}(C)$ is the direct part
  of the Reedy structure
  of $\int \catNerve^{{-},{+}}(C)$, and is therefore direct;
  the category  $\int \catNerve^{{--},{+}}_{+}(C)$ in question
  is a (full) subcategory of $\int \catNerve^{{-},{+}}_{+}(C)$.
\end{proof}

We now show the directness of $\Down(C)$.

\begin{lemma}
  \label{lem:groth-cat-strict-ob-cond}
  Let $(\deltaBra{n}, X)\in\Ob\int\catNerve^{{-},{+}}(C)$.
  Then the following conditions are equivalent:
  \begin{enumerate}
  \item $(\deltaBra{n}, X)\in\Ob\int\catNerve^{{--},{+}}_{+}(C)$.
    \label{item:lem-groth-cat-strict-ob-cond:ob}
  \item The element
    $\idmor[(\deltaBra{n},X)]\in\HomOf[\int\catNerve^{{-},{+}}(C)]
    {(\deltaBra{n},X)}{(\deltaBra{n},X)}$ is maximal, i.e.,
    for any $\boldsymbol{\alpha} \ge \idmor[(\deltaBra{n},X)]
    \colon (\deltaBra{n}, X) \to (\deltaBra{n}, X)$, we have
    $\boldsymbol{\alpha} = \idmor[(\deltaBra{n}, X)]$.
    \label{item:lem-groth-cat-strict-ob-cond:mx}
  \item The element
    $\idmor[(\deltaBra{n},X)]\in\HomOf[\int\catNerve^{{-},{+}}(C)]
    {(\deltaBra{n},X)}{(\deltaBra{n},X)}$ is minimal, i.e.,
    for any $\boldsymbol{\alpha} \le \idmor[(\deltaBra{n},X)]
    \colon (\deltaBra{n}, X) \to (\deltaBra{n}, X)$, we have
    $\boldsymbol{\alpha} = \idmor[(\deltaBra{n}, X)]$.
    \label{item:lem-groth-cat-strict-ob-cond:mn}
  \item The element $\idmor[(\deltaBra{n},X)] \in
    \HomOf[\int\catNerve^{{-},{+}}(C)]{(\deltaBra{n},X)}{(\deltaBra{n},X)}$
    is the largest in its own equivalence class.
    \label{item:lem-groth-cat-strict-ob-cond:eqmx}
  \item The equivalence class of $\idmor[(\deltaBra{n},X)] \in
    \HomOf[\int\catNerve^{{-},{+}}(C)]{(\deltaBra{n},X)}{(\deltaBra{n},X)}$ is
    a singleton.
    \label{item:lem-groth-cat-strict-ob-cond:eqsing}
  \end{enumerate}
\end{lemma}
\begin{proof}
  By the definition of the equivalence of parallel morphisms,
  \eqref{item:lem-groth-cat-strict-ob-cond:eqsing}
  is equivalent to the conjuction of
  \labelcref{item:lem-groth-cat-strict-ob-cond:mx,%
    item:lem-groth-cat-strict-ob-cond:mn}.
  The equivalence of
  \eqref{item:lem-groth-cat-strict-ob-cond:mx}
  and \eqref{item:lem-groth-cat-strict-ob-cond:eqmx}
  follows from
  \cref{prop:groth-cat-eqvce-eqvcond}.
  We show
  $\eqref{item:lem-groth-cat-strict-ob-cond:ob}
  \implies \eqref{item:lem-groth-cat-strict-ob-cond:mx}
  \implies \eqref{item:lem-groth-cat-strict-ob-cond:mn}
  \implies \eqref{item:lem-groth-cat-strict-ob-cond:ob}$
  part by part.

  $\eqref{item:lem-groth-cat-strict-ob-cond:ob}
  \implies \eqref{item:lem-groth-cat-strict-ob-cond:mx}$:
  Let $(\deltaBra{n}, X)\in\Ob\int\catNerve^{{--},{+}}_{+}(C)$, and let:
  \[
    \boldsymbol{\alpha} = (\alpha,\theta) \ge \idmor[(\deltaBra{n},X)]
    \colon (\deltaBra{n}, X) \to (\deltaBra{n}, X).
  \]
  Let $i\in\deltaBra{n}$. We have the following commutative diagram:
  \[
    \begin{tikzcd}
      X(i)\arrow[dr, tail, "\theta_i"]
      \arrow[r, equal]
      & X(i) \arrow[d, twoheadrightarrow, "X(\uniqmor)"] \\
      & X(\alpha(i))
    \end{tikzcd}
  \]
  By the uniqueness of $(C_{-}, C_{+})$-factorization, we get
  $X(i) = X(\alpha(i))$, $\theta_i = \idmor[X(i)]$, and
  $X(\uniqmor_{\alpha(i)i})=\idmor[X(i)]\colon X(\alpha(i)) \to X(i)$.
  From the last equation of the three and the assumption
  \eqref{item:lem-groth-cat-strict-ob-cond:ob}, we obtain
  $\alpha(i) = i$. Wrapping up, we get $\alpha = \idmor[\deltaBra{n}]$,
  $\theta = \idmor[X]$, and hence
  $\boldsymbol{\alpha} = \idmor[(\deltaBra{n}, X)]$, as required.

  $\eqref{item:lem-groth-cat-strict-ob-cond:mx}
  \implies \eqref{item:lem-groth-cat-strict-ob-cond:mn}$:
  Assume \eqref{item:lem-groth-cat-strict-ob-cond:mx}, and let
  $\boldsymbol{\alpha} = (\alpha,\theta) \le \idmor[(\deltaBra{n},X)]
  \colon (\deltaBra{n}, X) \to (\deltaBra{n}, X)$.
  Define a morphism $\beta\colon \deltaBra{n} \to \deltaBra{n}$
  in $\mtcSimpx$ by
  $\beta(i) \coloneqq \max \alpha^{-1}\set{\alpha(i)}$ for
  $i \in \deltaBra{n}$. By definition, we see that
  $\beta \ge \idmor[\deltaBra{n}]$. From this and \cref{lem:groth-cat-up-lift}
  follows that $\beta$ corresponds to a unique morphism
  $\boldsymbol{\beta}=(\beta,\phi)\colon
  (\deltaBra{n}, X) \to (\deltaBra{n}, X)$ in $\int\catNerve^{{-},{+}}(C)$
  satisfying $\boldsymbol{\beta} \ge \idmor[(\deltaBra{n}, X)]$, which,
  by our assumption \eqref{item:lem-groth-cat-strict-ob-cond:mx},
  implies $\boldsymbol{\beta} = \idmor[(\deltaBra{n}, X)]$. Therefore
  $\alpha$ is injective, and hence $\alpha = \idmor[\deltaBra{n}]$.
  Combining this with $\boldsymbol{\alpha} \le \idmor[(\deltaBra{n}, X)]$,
  we obtain the desired equation
  $\boldsymbol{\alpha} = \idmor[(\deltaBra{n}, X)]$.

  $\eqref{item:lem-groth-cat-strict-ob-cond:mn}
  \implies \eqref{item:lem-groth-cat-strict-ob-cond:ob}$:
  Suppose \eqref{item:lem-groth-cat-strict-ob-cond:mn}. Assume that
  $0 \le k \le l \le n$ satisfy:
  \[
    X(\uniqmor_{kl})=\idmor[X(k)]\colon X(k)\to X(k)=X(l).
  \]
  It suffices to show $k=l$. By the inverseness of $C_{-}$,
  for each pair $k \le i \le j \le l$, it holds that:
  \begin{equation}
    \label{eq:lem-groth-cat-strict-ob-cond:idmor}
    X(\uniqmor_{ij})=\idmor[X(k)]\colon X(i)=X(k)\to X(k)=X(j).
  \end{equation}
  Define $\alpha\colon \deltaBra{n} \to \deltaBra{n}$ in $\mtcSimpx$ by
  \[
    \alpha(i)\coloneqq
    \begin{cases}
      k & \text{if $k \le i \le l$,} \\
      i & \text{if $0 \le i < k$ or $l < i \le n$}.
    \end{cases}
  \]
  By \cref{eq:lem-groth-cat-strict-ob-cond:idmor},
  we see $X = X\compos\alpha$; therefore
  $(\alpha, \idmor[X])\colon (\deltaBra{n},X)\to(\deltaBra{n},X)$
  is a valid morphism in $\int\catNerve^{{-},{+}}(C)$.
  By \cref{eq:lem-groth-cat-strict-ob-cond:idmor} again, we see that
  $(\alpha,\idmor[X]) \le \idmor[(\deltaBra{n},X)]$, and hence that
  $(\alpha,\idmor[X]) = \idmor[(\deltaBra{n},X)]$
  by the assumption \eqref{item:lem-groth-cat-strict-ob-cond:mn}.
  From this we obtain $l = \idmor[\deltaBra{n}](l) = \alpha(l) = k$,
  which concludes the proof.
\end{proof}

\begin{proposition} \label{prop:down-c-direct}
  The category $\Down(C)$ is direct.
\end{proposition}
\begin{proof}
  By \cref{lem:groth-cat-strict-ob-cond},
  the equivalence class of any idenitity in
  $\int \catNerve^{{--},{+}}_{+}(C)$ is a singleton.
  Therefore the result immediately follows from \cref{cor:groth-cat-direct}.
\end{proof}

We conclude this section by showing the finiteness of $\Down(C)$.

\begin{lemma} \label{lem:groth-cat-finite}
  If $C$ is finite, then $\int\catNerve^{{--},{+}}_{+}(C)$ is finite.
\end{lemma}
\begin{proof}
  Since $C$ is finite, $\Mor C$ has decidable equality.
  Therefore, for any pair of objects $(\deltaBra{m}, X),(\deltaBra{n}, Y)\in\Ob\int\catNerve(C)$,
  the hom-set $\HomOf[\int\catNerve(C)]{(\deltaBra{m}, X)}{(\deltaBra{n}, Y)}$
  is a decidable subset of the finite product 
  $\HomOf[\mtcSimpx]{\deltaBra{m}}{\deltaBra{n}} \times (\Mor C)^{m+1}$ of finite sets, and hence
  is finite. Assume that $(\deltaBra{m}, X),(\deltaBra{n}, Y)\in\Ob\int\catNerve^{{--},{+}}_{+}(C)$.
  The decidability of the subset $\Mor C_{+} \subseteq \Mor C$, which we have assumed as part of the definition of
  Reedy structures, implies the decidability of the following subset, and hence its finiteness:
  \[ 
      \HomOf[\int\catNerve^{{--},{+}}_{+}(C)]{(\deltaBra{m}, X)}{(\deltaBra{n}, Y)}
      \subseteq \HomOf[\int\catNerve(C)]{(\deltaBra{m}, X)}{(\deltaBra{n}, Y)}.
  \]

  Thus it suffices to show that the set of objects of $\int\catNerve^{{--},{+}}_{+}(C)$ is finite.
  If the number of objects of $C$ is $N$, then any object $(\deltaBra{n}, X)$ of
  $\int\catNerve^{{--},{+}}_{+}(C)$ satisfies $n < N$. Since the set of functors $\deltaBra{n} \to C_{-}$,
  which $X$ is a member of, is finite, the object set in question may be seen as the disjoint union of a finite family
  of finite sets, and is therefore finite.
\end{proof}

\begin{corollary} \label{lem:down-c-finite}
  If $C$ is finite, then $\Down(C)$ is finite.
\end{corollary}
\begin{proof}
  The category $\Down(C)$ is obtained by taking quotients of the hom-sets of $\int\catNerve^{{--},{+}}_{+}(C)$, which
  is finite by \cref{lem:groth-cat-finite}. Therefore it suffices to check the equivalence relations on the hom-sets
  producing the quotient is decidable. For this, see \cref{prop:groth-cat-eqvce-eqvcond} and remember that the category
  $C$ has decidable equality on morphisms.
\end{proof}

\section{Functors between the constructed categories}
\label{sec:groth-cat-comparison}

In this section, we examine canonical functors between the categories
defined in \cref{sec:def-cat-down}. Here we list such functors:

\begin{remark}\label{rem:groth-cat-functors}
  By the construction of the five categories constructed in
  \cref{sec:def-cat-down}, we have the following commutative diagram of functors:
  \[
    \begin{tikzcd}
      \int\catNerve^{{--},{+}}_{+}(C) \arrow[r, hook, "\text{full}"] \arrow[d, two heads]
      & \int\catNerve^{{-},{+}}(C) \arrow[d, two heads] \arrow[r, hook]
      & \int\catNerve(C) \\
      \Down(C) \arrow[r, hook, "\text{full}"]
      & \Downstar(C) &
    \end{tikzcd}
  \]
  The two horizontal arrows in the upper row are inclusions of subcategories.
  The vertical arrows are the canonical projections: to be pedantic on categorical formalism,
  it is induced by the canonical natural transformation
  \[
    \mathrm{forget} \Rightarrow \pi_0\colon \mtcPoset \to \mtcSet.
  \]
  The lower horizontal arrow $\Down(C) \hookrightarrow \Downstar(C)$ is induced by the inclusion
  $\int\catNerve^{{--},{+}}_{+}(C) \hookrightarrow \int\catNerve^{{-},{+}}(C)$, which,
  by \cref{def:groth-cat-poset-enrich,lem:groth-cat-full-sub},
  is a $\mtcPoset$-fully faithful $\mtcPoset$-enriched functor,
  meaning that the functor is order isomorphisms in hom-posets.
  Therefore, the $\Down(C) \hookrightarrow \Downstar(C)$ may be considered as the inclusion of a full subcategory.
\end{remark}

The main contents of this section are the following two results:
\begin{itemize}
  \item The inclusion $\Down(C) \hookrightarrow \Downstar(C)$ of a full subcategory is a categorical equivalence
  (\cref{lem:eqvce-down-c-downstar}).
  \item The canonical surjective functor $\int\catNerve^{{-},{+}}(C) \twoheadrightarrow \Downstar(C)$ is a strict
  1-localization functor (\cref{prop:downstar-strict-localization}).
\end{itemize}

\begin{lemma}\label{lem:eqvce-down-c-downstar}
  The inclusion $\Down(C) \hookrightarrow \Downstar(C)$
  is a categorical equivalence.
\end{lemma}
\begin{proof}
  Since the inclusion is fully faithful, it suffices to
  explicitly construct,
  for each $(\deltaBra{m}, X) \in \Ob(\Downstar(C))$, a pair of
  an object $(\deltaBra{n}, Y) \in \Ob(\Down(C))$ and an
  isomorphism $(\deltaBra{n}, Y) \overset{\sim}{\longrightarrow} (\deltaBra{m}, X)$
  in $\Downstar(C)$. Let $(\deltaBra{m}, X) \in \Ob(\Downstar(C))$ and
  consider its identity morphism
  in $\int\catNerve^{{-},{+}}(C)$:
  \[
    \idmor[(\deltaBra{m}, X)] = (\idmor[\deltaBra{m}], \idmor[X])\colon
    (\deltaBra{m}, X) \to (\deltaBra{m}, X)\text{.}
  \]
  Let $\boldsymbol{\alpha}\colon(\deltaBra{m}, X) \to (\deltaBra{m}, X)$
  be the largest
  element of the equivalence class of $\idmor[(\deltaBra{m}, X)]$, which is
  shown to exist in \cref{cor:groth-cat-eqvce-class-max}.
  Since
  $\idmor[(\deltaBra{m}, X)] \le \boldsymbol{\alpha}$, we have
  $\boldsymbol{\alpha} \le \boldsymbol{\alpha}\compos\boldsymbol{\alpha}$.
  The maximality of $\boldsymbol{\alpha}$ and
  $\boldsymbol{\alpha} \sim \boldsymbol{\alpha}\compos\boldsymbol{\alpha}$
  together give
  $\boldsymbol{\alpha} \ge \boldsymbol{\alpha}\compos\boldsymbol{\alpha}$;
  we therefore get
  $\boldsymbol{\alpha} = \boldsymbol{\alpha}\compos\boldsymbol{\alpha}$.
  By \cref{lem:reedy-idem-split} and \cref{prop:groth-cat-reedy}, the
  idempotent $\boldsymbol{\alpha}$ has a unique
  splitting of the following form:
  \[
    \begin{tikzcd}
      (\deltaBra{m}, X)
      \arrow[rr, "\boldsymbol{\alpha}"]
      \arrow[rd, twoheadrightarrow, "\boldsymbol{\sigma}"']
      && (\deltaBra{m}, X)
      \arrow[rd, twoheadrightarrow, "\boldsymbol{\sigma}"]& \\
      & (\deltaBra{n}, Y)
      \arrow[rr, equal]
      \arrow[ru, tail, "\boldsymbol{\delta}"]
      && (\deltaBra{n},Y)
    \end{tikzcd}
  \]
  Since $\boldsymbol{\alpha}\sim\idmor$, we have
  constructed an isomorphism $(\deltaBra{m}, X)\cong(\deltaBra{n},Y)$
  in $\Downstar(C)$.

  It only remains to show that
  \[
    (\deltaBra{n}, Y) \in
    \Ob(\Down(C))=\Ob\textstyle\int \catNerve^{{--},{+}}_{+}(C).
  \]
  Let $\boldsymbol{\alpha'}\colon(\deltaBra{n}, Y)\to(\deltaBra{n}, Y)$
  be the largest element of the equivalence class of the identity
  $\idmor[(\deltaBra{n}, Y)] \in
  \HomOf[\int\catNerve^{{-},{+}}(C)]{(\deltaBra{n}, Y)}{(\deltaBra{n}, Y)}$.
  By \cref{lem:groth-cat-strict-ob-cond},
  it suffices to prove $\boldsymbol{\alpha'} = \idmor[(\deltaBra{n}, Y)]$.
  By repeating the argument above, we may take $\boldsymbol{\sigma'}$ and
  $\boldsymbol{\delta'}$ with the following commutative diagram:
  \[
    \begin{tikzcd}
      (\deltaBra{n}, Y)
      \arrow[rr, "\boldsymbol{\alpha'}"]
      \arrow[rd, twoheadrightarrow, "\boldsymbol{\sigma'}"']
      && (\deltaBra{n}, Y)
      \arrow[rd, twoheadrightarrow, "\boldsymbol{\sigma'}"]& \\
      & (\deltaBra{l}, Z)
      \arrow[rr, equal]
      \arrow[ru, tail, "\boldsymbol{\delta'}"]
      && (\deltaBra{l},Z)
    \end{tikzcd}
  \]
  Since $\idmor[(\deltaBra{n}, Y)]\le\boldsymbol{\alpha'}$,
  we have:
  \[
    \boldsymbol{\alpha}
    = \boldsymbol{\delta} \compos  \idmor[(\deltaBra{n}, Y)] \compos \boldsymbol{\sigma}
    \le \boldsymbol{\delta} \compos \boldsymbol{\alpha'}  \compos \boldsymbol{\sigma}
    = \boldsymbol{\delta} \compos  \boldsymbol{\delta'}
    \compos \boldsymbol{\sigma'} \compos \boldsymbol{\sigma}.
  \]
  From this and the maximality of $\boldsymbol{\alpha}$, we can derive:
  \[
    \boldsymbol{\delta} \compos \boldsymbol{\sigma} = \boldsymbol{\alpha}
    = \boldsymbol{\delta} \compos  \boldsymbol{\delta'}
    \compos \boldsymbol{\sigma'} \compos \boldsymbol{\sigma}.
  \]
  \cref{prop:groth-cat-reedy} implies
  $(\deltaBra{l},Z)=(\deltaBra{n},Y)$,
  $\boldsymbol{\delta}=\boldsymbol{\delta}\compos\boldsymbol{\delta'}$ and
  $\boldsymbol{\sigma}=\boldsymbol{\sigma'}\compos\boldsymbol{\sigma}$, which require
  $\boldsymbol{\delta'}=\boldsymbol{\sigma'}=\idmor[(\deltaBra{n},Y)]$.
  Hence $\boldsymbol{\alpha'} = \idmor[(\deltaBra{n},Y)]$, as desired.
\end{proof}

Next, we consider the quotient functor $\int\catNerve^{{-},{+}}(C)\to\Downstar(C)$:

\begin{proposition} \label{prop:downstar-strict-localization}
  The quotient functor $q\colon\int\catNerve^{{-},{+}}(C)\to\Downstar(C)$
  exhibits the category $\Downstar(C)$ as the strict 1-localization of $\int\catNerve^{{-},{+}}(C)$
  at $\int\catNerve^{{-},\set{\idmor}}_{-}(C)$ in \cref{def:groth-cat-reedy-def}.
\end{proposition}
\begin{proof}
  We first show that the functor $q$ sends every morphism in $\int\catNerve^{{-},\set{\idmor}}_{-}(C)$
  to an isomorphism in $\Downstar(C)$.
  Let $(\sigma,\idmor[X\compos\sigma])\colon
  (\deltaBra{m},X\compos\sigma)\to(\deltaBra{n},X)$ be a morphism in
  $\int\catNerve^{{-},\set{\idmor}}_{-}(C)$.
  Using $\sigma\in\Mor\mtcSimpx_{-}$, which follows from the definition
  of the wide subcategory $\int\catNerve^{{-},\set{\idmor}}_{-}(C)$,
  take the largest section $\delta\colon\deltaBra{n}\to\deltaBra{m}$ of
  $\sigma\colon\deltaBra{m}\to\deltaBra{n}$.
  We easily calculate to see that
  $(\delta,\idmor[X])\colon(\deltaBra{n},X)=(\deltaBra{n},X\compos\sigma\compos\delta)
  \to(\deltaBra{m},X\compos\sigma)$
  is a legitimate morphism in $\int\catNerve^{{-},{+}}(C)$ and is a section
  of $(\sigma,\idmor[X\compos\sigma])$. In addition, we have:
  \[
    \idmor[(\deltaBra{m},X\compos\sigma)]
    =(\idmor[\deltaBra{m}],\idmor[X\compos\sigma])
    \le(\delta\compos\sigma,\idmor[X\compos\sigma]) =
    (\delta,\idmor[X])\compos(\sigma,\idmor[X\compos\sigma]).
  \]
  Thus, $q(\sigma,\idmor[X\compos\sigma])$ and $q(\delta,\idmor[X])$ are mutual
  inverses in $\Downstar(C)$, as desired.

  Next, we address the conditional unique factorization through $q$.
  Assume that a functor $F\colon\int\catNerve^{{-},\set{\idmor}}_{-}(C)\to D$ sends
  every morphism in $\int\catNerve^{{-},\set{\idmor}}_{-}(C)$ to an isomorphism in $D$.
  Since $q$ is a quotient functor,
  it suffices to prove that $F$ respects the equivalence
  relations on hom-sets: explicitly, $F(\alpha,\theta)=F(\alpha',\theta')$ for any
  parallel pair of morphisms
  $(\alpha,\theta)\ge(\alpha',\theta')\colon (\deltaBra{m},X)\to(\deltaBra{n},Y)$.
  Take any such parallel pair. To facilitate the proof, we construct a sequence
  of morphisms:
  \begin{multline*}
    (\alpha,\theta)=(\alpha^0,\theta^0)\ge(\alpha^1,\theta^1)\ge
    \dotsb\ge(\alpha^m,\theta^m)\ge(\alpha^{m+1},\theta^{m+1}) = (\alpha',\theta')\\
    \colon (\deltaBra{m},X)\to(\deltaBra{n},Y).
  \end{multline*}
  Specifically, for each $0\le k \le m+1$, the $k$-th entry of the sequence is
  given by:
  \begin{align*}
    \alpha^k(i)
    &\coloneqq
      \begin{cases}
        \alpha'(i) & \text{if } 0 \le i < k\text{, and} \\
        \alpha(i) & \text{if } k \le i \le m;
      \end{cases}\\
    \left(\theta^k\right)_i
    &\coloneqq
      \begin{cases}
        \theta'_i &\text{if } 0 \le i < k\text{, and} \\
        \theta_i &\text{if } k \le i \le m.
      \end{cases}
  \end{align*}

  Now, we need only demonstrate $F(\alpha^k,\theta^k)=F(\alpha^{k+1},\theta^{k+1})$
  for every $0\le k \le m$.
  Let the morphisms $\sigma^m_k\colon \deltaBra{m+1} \to \deltaBra{m}$ and
  $\delta^{m+1}_k,\delta^{m+1}_{k+1}\colon \deltaBra{m}\to\deltaBra{m+1}$
  in $\mtcSimpx$ be those from the usual convention.
  By assumption, the functor $F$ sends the morphism
  \[
    (\sigma^m_k,\idmor[X\compos\sigma^m_k])
    \colon (\deltaBra{m+1}, X\compos\sigma^m_k) \to (\deltaBra{m}, X)
  \]
  to an isomorphism. Since the two morphisms
  \[
    (\delta^{m+1}_k,\idmor[X]),(\delta^{m+1}_{k+1},\idmor[X])
    \colon (\deltaBra{m}, X) \to (\deltaBra{m+1}, X\compos\sigma^m_k)
  \]
  are both
  sections of $(\sigma^m_k,\idmor[X\compos\sigma^m_k])$,
  their images by $F$ coincide.
  Define the morphism
  $(\tilde{\alpha}^k,\tilde{\theta}^k)\colon(\deltaBra{m+1}, X\compos\sigma^m_k)
  \to(\deltaBra{n}, Y)$
  by:
  \begin{align*}
    \tilde{\alpha}^k(i)
    &\coloneqq
      \begin{cases}
        \alpha'(i) & \text{if } 0 \le i \le k\text{, and} \\
        \alpha(i-1) & \text{if } k < i \le m+1;
      \end{cases}\\
    {(\tilde{\theta}^k)}_i
    &\coloneqq
      \begin{cases}
        \theta'_i &\text{if } 0 \le i \le k\text{, and} \\
        \theta_{i-1} &\text{if } k < i \le m+1.
      \end{cases}
  \end{align*}
  We see $(\tilde{\alpha}^k,\tilde{\theta}^k)\compos(\delta^{m+1}_l,\idmor[X])
  = (\alpha^l,\theta^l)$ for $l=k, k+1$; therefore we get the following, as required:
  \begin{align*}
    &F(\alpha^k,\theta^k)\\
    &=F(\tilde{\alpha}^k,\tilde{\theta}^k)\compos F(\delta^{m+1}_k,\idmor[X])\\
    &=F(\tilde{\alpha}^k,\tilde{\theta}^k)\compos F(\delta^{m+1}_{k+1},\idmor[X])\\
    &=F(\alpha^{k+1},\theta^{k+1}).
  \end{align*}
\end{proof}

Now, we add an obvious corollary to the above proposition:
we can consider both $\Down(C)$ and $\Downstar(C)$
as the localization of $\int\catNerve^{{-},{+}}(C)$ at
$\int\catNerve^{{-},\set{\idmor}}_{-}(C)$:

\begin{corollary}\label{cor:down-weak-localization}
  Write $q\colon \int\catNerve^{{-},{+}}(C)\to\Downstar(C)$ for the canonical
  projection, and let $r\colon \Downstar(C)\to\Down(C)$ stand for the quasi-inverse of
  the inclusion $\Down(C)\hookrightarrow\Downstar(C)$, constructed in
  \cref{lem:eqvce-down-c-downstar}. Then the pairs
  $(\Downstar(C), q)$ and $(\Down(C), r\compos q)$ are both weak 1-localizations
  of $\int\catNerve^{{-},{+}}(C)$ at
  $\int\catNerve^{{-},\set{\idmor}}_{-}(C)$.
\end{corollary}
\begin{proof}
  Follows from \cref{prop:downstar-strict-localization,rmk:cat-1-localization-weak-strict-equiv}.
\end{proof}

\section{The proof of 1-localization}
\label{sec:down-last-localiz}

Now in this section, we compare $\Down(C)$ and $C$.
We first construct the functor that compares
the two categories:

\begin{definition}[The last component functor]
  \label{def:down-fctr-last}
  The following describes three functors $\int\catNerve(C)\to C$,
  $\int\catNerve^{{-},{+}}(C)\to C$, and $\int\catNerve^{{--},{+}}_{+}(C)\to C$,
  which will all be denoted commonly by $\last$:
  \begin{alignat*}{2}
    \last(\deltaBra{n},X)
    &\coloneqq X(n)
    &\quad&\text{for }
            (\deltaBra{n},X)\text{ an object of the domain};\\
    \last(\alpha,\theta)
    &\coloneqq Y(\uniqmor_{\alpha(m),n})\compos\theta_m
    &&\text{for }(\alpha,\theta)\colon (\deltaBra{m},X)\to(\deltaBra{n},Y).
  \end{alignat*}
  These functors respect the equivalence relations on the hom-sets of the domain categories,
  so that they induce the functors $\Down(C) \to C$ and $\Downstar(C) \to C$. These two
  will also be denoted by $\last$, abusing the notation. When we need to distinguish
  the five functors above, we write $\last[\Gamma]\colon \Gamma \to C$.
\end{definition}

Our purpose is to prove that $\last$ presents $C$ as a localization of $\Down(C)$.
Therefore we will also need the class of weak equivalences:

\begin{definition}[$\last$-weak equivalence]
  \label{def:down-weq-last}
  Let $\Gamma$ denote one of the five categories $\int\catNerve(C)$,
  $\int\catNerve^{{-},{+}}(C)$, $\int\catNerve^{{--},{+}}_{+}(C)$,
  $\Downstar(C)$, and $\Down(C)$. A morphism in $\Gamma$ is said to be a
  $\last$-weak equivalence (in $\Gamma$) if $\last\colon \Gamma \to C$ maps it to some
  identity in $C$. We will write:
  \[
    \weqlast = \weqlast[\Gamma] \coloneqq
    \set{\boldsymbol{\alpha}\in\Mor\Gamma}[{\last({\boldsymbol{\alpha})=\idmor}}]
    \subseteq \Mor \Gamma
  \]
  for the collection of $\last$-weak equivalences.
\end{definition}

The main theorem of this section is the following one:

\begin{theorem}
  \label{thm:down-last-localiz}
  Let $\Gamma$ be one of the four categories $\int\catNerve(C)$,
  $\int\catNerve^{{-},{+}}(C)$, $\Down(C)$, and $\Downstar(C)$.
  The category $C$ and the functor $\last[\Gamma]$ in \cref{def:down-fctr-last} constitute
  a weak 1-localization at $\last$-weak equivalences.
\end{theorem}

Notice the \emph{four} in the statement of the theorem; we have one counterexample:

\begin{remark}
  The functor $\last\colon\int\catNerve^{{--},{+}}_{+}(C)\to C$ is \emph{not}
  in general a 1-lo\-cal\-i\-za\-tion. For a simple counterexample, define $C$ to be
  the poset $\set{0 \le 2 \le 1}$ equipped with the Reedy structure
  $(C_{-},C_{+})$ determined by $g\coloneqq\uniqmor_{21} \in C_{-}$,
  $f\coloneqq\uniqmor_{02}\in C_{+}$, and $g\compos f=\uniqmor_{01} \in C_{+}$.
  Let $C'$ be the category obtained by freely adjoining $h\colon 0 \to 1$ to $C$.
  In diagram:
  \begin{align*}
    C&\coloneqq\left\{
      \begin{tikzcd}[ampersand replacement=\&]
        \&2 \arrow[d, "g"{name=MID}, two heads]\\
        \&1 \\
        0 \arrow[uur, "f", bend left, tail]
        \arrow[phantom, to=MID, "\circlearrowleft"{pos=0.55}, bend left=20]
        \arrow[ur, "g\compos f"', tail] \&
      \end{tikzcd}
    \right\}; &
    C'&\coloneqq\left\{
      \begin{tikzcd}[ampersand replacement=\&]
        \&2 \arrow[d,"g"{name=MID}]\\
        \&1\\
        0\arrow[uur, "f", bend left] \arrow[ur] \arrow[ur, "h"', bend right=50]
        \arrow[phantom, to=MID, "\circlearrowleft"{pos=0.55}, bend left=20]
        \arrow[ur, phantom, "\circlearrowleft", "\bigtimes", bend right=25] \&
      \end{tikzcd}
    \right\}.
  \end{align*}
  Let $\mathrm{incl}\colon C\hookrightarrow C'$ be the inclusion.
  Let us define a functor $F\colon\int\catNerve^{{--},{+}}_{+}(C)\to C'$.
  On objects, we set: 
  \[ F(\deltaBra{n},X)\coloneqq X(n) = \operatorname{incl}(\last(\deltaBra{n},X)). \]
  Let $(\alpha,\theta)\colon(\deltaBra{m},X)\to(\deltaBra{n},Y)$ be a morphism in
  $\int\catNerve^{{--},{+}}_{+}(C)$. Then we set:
  \[ F(\alpha,\theta) \coloneqq \begin{cases}
      h & \text{if } m=0,\,X(0) = 0,\, Y(\alpha(0))=1\text{, and } \theta_0 = g\compos f; \\
      \operatorname{incl}(\last(\alpha,\theta)) & \text{otherwise}.
    \end{cases} \]
  Then $F$ is a well-defined functor and sends all the $\last$-weak equivalences to isomorphisms.
  However it does not factor through $\last$ even up to natural isomorphism;
  indeed, $F$ maps the parallel pair of morphisms
  \[ (\delta^1_1, f), (\delta^1_0, g\compos f)\colon(\deltaBra{0},0)\to
  (\deltaBra{1},2\overset{g}{\twoheadrightarrow}1) \]
  to an unequal parallel pair $g\compos f$ and $h$,
  while $\last$ sends them to an equal morphism $g\compos f$. This establishes a counterexample:
  $\last\colon\int\catNerve^{{--},{+}}_{+}(C)\to C$ is not a 1-localization.
\end{remark}

The rest of this section is devoted to the proof of \cref{thm:down-last-localiz}.
The demonstration is broken down to several lemmas, which will be wrapped up in
\cpageref{prf:thm-down-last-localiz}.
We begin with the most trivial two of these auxiliary propositions:

\begin{lemma}\label{lem:down-groth-last-respects-weq}
  The five functors denoted by $\last$ in \cref{def:down-fctr-last} send
  $\last$-weak equivalences to isomorphisms. In fact, $\last$-weak equivalences
  are exactly the morphisms mapped to isomorphisms by $\last$.
\end{lemma}
\begin{proof}
  The first statement follows from the definition of $\last$-weak equivalences. The second
  follows from the assumption that $C$ is a Reedy category.
\end{proof}

\begin{lemma}\label{lem:down-groth-last-compat}
  The functors $\last$ and the families $\weqlast$ are compatible with the
  following strictly commutative diagram of functors. Explicitly,
  for any arrow $F\colon\Gamma\to\Gamma'$ depicted in the diagram, we have
  ${\last[\Gamma']}\compos F = \last[\Gamma]$
  and $F^{-1}(\weqlast[\Gamma']) = \weqlast[\Gamma]$.
  \begin{equation}
    \label{eq:diagram-down-groth-cat}
    \begin{tikzcd}
      \int \catNerve^{{--},{+}}_{-}(C)
      \arrow[r, hook] \arrow[d, "\mathrm{quotient}"',two heads]
      & \int \catNerve^{{-},{+}}(C)
      \arrow[r, hook] \arrow[d, "\mathrm{quotient}",two heads]
      & \int \catNerve(C) \\
      \Down(C)
      \arrow[r, hook, "\sim"]
      &\Downstar(C)
      &
    \end{tikzcd}
  \end{equation}
\end{lemma}
\begin{proof}
  By definition.
\end{proof}

To facilitate a clearer presentation, we commence with the relatively simple case of $\Gamma=\int\catNerve(C)$ in \cref{thm:down-last-localiz}:

\begin{lemma}\label{lem:groth-tot-last-localiz}
  The functor $\last\colon \int\catNerve(C)\to C$ exhibits $C$ as a weak 1-localization
  of $\int\catNerve(C)$ at $\last$-weak equivalences.
\end{lemma}
\begin{proof}
  The functor in question admits the following right adjoint:
  \begin{align*}
    i\colon C &\to \textstyle\int\catNerve(C);\\
    \Ob C=\Ob \functCat{\deltaBra{0}}{C}\ni x &\mapsto (\deltaBra{0},x);\\
    (f\colon x\to y) &\mapsto (\idmor[\deltaBra{0}],f).
  \end{align*}
  Since this right adjoint is fully faithful, we may consider $C$ as a reflective
  subcategory of $\int\catNerve(C)$, and $\last$ is its reflector.
  Now the claim follows from \cref{lem:down-groth-last-respects-weq}.
\end{proof}

The other cases are a little more intricate; therefore we prove them part by part.
The next lemma shows the weak uniqueness property of the factorization through
$\last[\Gamma]$, which is a part of the definition of weak 1-localization in 
\cref{def:cat-1-localization}.
It holds even for $\Gamma=\int\catNerve^{{--},{+}}_{+}(C)$:

\begin{lemma}\label{lem:last-localiz-whisker-bij}
  Let $\last=\last[\Gamma]\colon\Gamma\to C$ be any of the five functors defined
  in \cref{def:down-fctr-last}.
  Let $F, G\colon C\to D$ be functors. If:
  \[
    \epsilon\colon F\compos{\last} \Rightarrow G\compos{\last}
    \colon \Gamma \to D
  \]
  is a natural transformation, there is a unique natural transformation
  $\tilde{\epsilon}\colon F \Rightarrow G$ satisfying
  $\tilde{\epsilon}\whiskr\last = \epsilon$.
\end{lemma}
\begin{proof}
  The equation $\tilde{\epsilon}\whiskr\last = \epsilon$ implies:
  \begin{equation}\label{eq:lem-last-localiz-whisker-bij:comp}
    \tilde{\epsilon}_x = \tilde{\epsilon}_{\last(\deltaBra{0},x)}
    = \epsilon_{(\deltaBra{0},x)}\colon F(x) \to G(x)
  \end{equation}
  for any $x \in \Ob C = \Ob C^{\deltaBra{0}}$. This shows uniqueness.
  It suffices to prove that
  \eqref{eq:lem-last-localiz-whisker-bij:comp} defines a natural transformation
  $\tilde{\epsilon}\colon F\Rightarrow G$ and to check
  $\tilde{\epsilon}\whiskr\last = \epsilon$.

  We shall first show:
  \begin{equation}\label{eq:lem-last-localiz-whisker-bij:whisk}
    \tilde{\epsilon}_{\last(\deltaBra{n},X)}
    =\epsilon_{(\deltaBra{n},X)}\colon
    F(\last(\deltaBra{n},X)) \to G(\last(\deltaBra{n},X))
  \end{equation}
  for any $(\deltaBra{n},X)\in\Ob\Gamma$. Once we show the naturality of $\tilde\epsilon$,
  \eqref{eq:lem-last-localiz-whisker-bij:whisk} should automatically prove
  $\tilde{\epsilon}\whiskr\last = \epsilon$. Let $(\deltaBra{n},X)\in\Ob\Gamma$ be
  arbitrary. Remember our notation; the morphism
  $\iota_n\colon \deltaBra{0}\rightarrowtail\deltaBra{n}$ in $\mtcSimpx$
  is the one with $\iota_n(0)=n$.
  We apply the naturality of $\epsilon$ to
  $(\iota_n,\idmor[X(n)])\colon(\deltaBra{0},X(n))\to(\deltaBra{n},X)$,
  which belongs to any choice of $\Gamma$, to get the
  following commutative diagram:
  \[
    \begin{tikzcd}[column sep=huge]
      F(\last(\deltaBra{0},X(n)))
      \arrow[r,"\epsilon_{(\deltaBra{0},X(n))}"]
      \arrow[d,"{F(\last(\iota_n,\idmor[X(n)]))}"']
      & G(\last(\deltaBra{0},X(n)))
      \arrow[d,"{G(\last(\iota_n,\idmor[X(n)]))}"]
      \\
      F(\last(\deltaBra{n},X))
      \arrow[r,"\epsilon_{(\deltaBra{n},X)}"']
      & G(\last(\deltaBra{n},X))
    \end{tikzcd}
  \]
  Since the vertical arrows reduce to $\idmor[F(X(n))]$ and $\idmor[G(X(n))]$,
  the horizontal arrows must be equal. Hence the desired equality
  \eqref{eq:lem-last-localiz-whisker-bij:whisk} follows from the defining equation
  \eqref{eq:lem-last-localiz-whisker-bij:comp}.

  It just remains to confirm the naturality. Since $C$ is a Reedy category,
  it suffices to show
  the commutativity of the naturality diagram:
  \begin{equation}\label{eq:lem-last-localiz-whisker-bij:nat}
    \begin{tikzcd}
      F(x) \arrow[r, "\tilde{\epsilon}_x"]
      \arrow[d, "F(f)"']
      & G(x) \arrow[d, "G(f)"]
      \\
      F(y) \arrow[r, "\tilde{\epsilon}_y"']
      & G(y)
    \end{tikzcd}
  \end{equation}
  in the cases $f \in \Mor C_{+}$ and $f \in \Mor C_{-}$.

  If $f \in \Mor C_{+}$,
  the diagram \eqref{eq:lem-last-localiz-whisker-bij:nat} coincides with
  the following commutative diagram, which is
  the naturality of $\epsilon$ for
  $(\idmor[\deltaBra{0}],f)\colon (\deltaBra{0},x)\to(\deltaBra{0},y)$:
  \[
    \begin{tikzcd}[column sep=huge]
      F(\last(\deltaBra{0},x))
      \arrow[r,"\epsilon_{(\deltaBra{0},x)}"]
      \arrow[d,"{F(\last(\idmor[\deltaBra{0}],f))}"']
      & G(\last(\deltaBra{0},x))
      \arrow[d,"{G(\last(\idmor[\deltaBra{0}],f))}"]
      \\
      F(\last(\deltaBra{0},y))
      \arrow[r,"\epsilon_{(\deltaBra{0},y)}"']
      & G(\last(\deltaBra{0},y))
    \end{tikzcd}
  \]

  We discuss the remaining case: $f \in \Mor C_{-}$. Remember the property of
  $f$ being an identity is decidable. If $f$ is some identity, the desired naturality
  follows from the case $f\in\Mor C_{+}$. Otherwise, we use
  \eqref{eq:lem-last-localiz-whisker-bij:whisk} to see that the desired
  commutative diagram \eqref{eq:lem-last-localiz-whisker-bij:nat} is exactly
  the naturality of $\epsilon$ for
  $(\delta^1_1,\idmor[x])\colon(\deltaBra{0},x)\to
  (\deltaBra{1},x\overset{f}{\twoheadrightarrow}y)$, which is:
  \[
    \begin{tikzcd}[column sep=huge]
      F(\last(\deltaBra{0},x))
      \arrow[r,"\epsilon_{(\deltaBra{0},x)}"]
      \arrow[d,"{F(\last(\delta^1_1,\idmor[x]))}"']
      & G(\last(\deltaBra{0},x))
      \arrow[d,"{G(\last(\delta^1_1,\idmor[x]))}"]
      \\
      F(\last(\deltaBra{1},f))
      \arrow[r,"\epsilon_{(\deltaBra{1},f)}"']
      & G(\last(\deltaBra{1},f))
    \end{tikzcd}
  \]
  This now concludes the proof.
\end{proof}

Now we demonstrate the remaining condition of 1-localization:
the conditional existence of the factorization through $\last[\Gamma]$
for $\Gamma=\Downstar(C)$. The cases $\Gamma=\int\catNerve^{{-},{+}}(C)$ and
$\Gamma=\Down(C)$ are equivalent to this case, so their treatment is postponed.
We split out a small lemma from the proof to make it easier to read:

\begin{lemma}\label{lem:down-groth-last-decp}
  Let $\Gamma$ stand for any of $\int\catNerve(C)$, $\int\catNerve^{{-},{+}}(C)$
  and $\int\catNerve^{{--},{+}}_{+}(C)$.
  Then any morphism $(\alpha,\theta)\colon(\deltaBra{l},X)\to(\deltaBra{n},Z)$ 
  decomposes into the composition
  $(\alpha,\theta)=(\gamma,\psi)\compos(\beta,\phi)$
  of a (not necessarily unique) pair of morphisms
  \begin{align*}
    (\beta,\phi) &\colon (\deltaBra{l}, X) \to (\deltaBra{m},Y),\\
    (\gamma,\psi) &\colon (\deltaBra{m}, Y) \to (\deltaBra{n}, Z)
  \end{align*}
  satisfying $\beta(l) = m$, $Y(m) = Z(\gamma(m))$,
  and $\psi_m = \idmor[Y(m)] = \idmor[Z(\gamma(m))]$.
\end{lemma}
\begin{proof}
  Set $m\coloneqq \alpha(l)$. There is a factorization that
  fits into the following commutative diagram, given by $\alpha'(i) = \alpha(i)$:
  \[
    \begin{tikzcd}
      \deltaBra{l}\arrow[r,"{\exists\alpha'}"]
      \arrow[rd, "\alpha"']
      &\deltaBra{m}\arrow[d,tail,"\iota_{\le m}"]\\
      &\deltaBra{n}
    \end{tikzcd}
  \]
  Here, $\iota_{\le m} = \iota_{\set{0,1,\dotsc,m}}\colon\deltaBra{m}
  \rightarrowtail \deltaBra{n}$ is specified by $\iota_{\le m}(i) = i$. Now,
  \begin{align*}
    (\beta, \phi) \coloneqq (\alpha', \theta)
    &\colon (\deltaBra{l}, X)\to(\deltaBra{m}, Y\compos\iota_{\le m}),\\
    (\gamma,\psi) \coloneqq
    (\iota_{\le m}, \idmor[Y\compos\iota_{\le m}])
    &\colon (\deltaBra{m}, Y\compos\iota_{\le m})\to(\deltaBra{n}, Y)
  \end{align*}
  meets the requirements.
\end{proof}

Now we proceed to the actual proof of the desired existence.

\begin{proposition}\label{prop:downstar-last-localiz-factor-exist}
  Let $D$ be a category. Suppose that a functor
  $F\colon \Downstar(C)\to D$ maps every $\last$-weak equivalence to an isomorphism.
  Then there exists a functor $\tilde{F}\colon C \to D$ and a natural isomorphism
  $\theta\colon \tilde{F}\compos\last\overset\sim\Rightarrow F$.
\end{proposition}
\begin{proof}
  For the sake of readability of the construction of the functor $\tilde{F}$, we begin by
  specifying its two restrictions $\tilde{F}_{+}\colon C_{+}\to D$ and
  $\tilde{F}_{-}\colon C_{-}\to D$.
  For any object $x \in \Ob C = \Ob\functCat{\deltaBra{0}}{C}$, we set
  $\tilde{F}(x) = \tilde{F}_{+}(x) = \tilde{F}_{-}(x) \coloneqq
  F(\deltaBra{0},x)$.
  If $d\colon x \rightarrowtail y$ be a morphism in $C_{+}$, then we define
  $\tilde{F}_{+}(d)\coloneqq F(\idmor[\deltaBra{0}], d)$,
  through the identification $C = \functCat{\deltaBra{0}}{C}$.
  This obviously makes $\tilde{F}_{+}$ a functor.

  Now consider a morphism $s\colon x \twoheadrightarrow y$ in $C_{-}$.
  The object $(x\overset{s}{\twoheadrightarrow}y)\in
  \Ob\functCat{\deltaBra{1}}{(C_{-})}$ constitutes an object
  $(\deltaBra{1},s)\in\Ob\Downstar(C)$. We define $\tilde{F}_{-}(s)$ to be the
  unique morphism that commutates the following diagram:
  \begin{equation}\label{eq:lem-downstar-last-localiz-factor-exist:img-degen-spec}
    \begin{tikzcd}[column sep=huge]
      \tilde{F}_{-}(x) \arrow[d, equal] \arrow[rr, "\tilde{F}_{-}(s)"]
      && \tilde{F}_{-}(y) \arrow[d, equal] \\
      F(\deltaBra{0},x) \arrow[r, "{F(\delta^0_1,\idmor[x])}"']
      & F(\deltaBra{1},s)
      & F(\deltaBra{0},y)\arrow[l, "{F(\delta^0_0,\idmor[y])}", "\sim"']
    \end{tikzcd}
  \end{equation}
  We need to show the functoriality of $\tilde{F}_{-}$.
  Since putting $x=y$, $s=\idmor[x]$,
  and $\tilde{F}_{-}(\idmor[x])=\idmor[\tilde{F}_{-}(x)]$ commutes the diagram
  \eqref{eq:lem-downstar-last-localiz-factor-exist:img-degen-spec}, we see that
  $\tilde{F}_{-}(\idmor[x])=\idmor[\tilde{F}_{-}(x)]$.
  We prove that $\tilde{F}_{-}$ preserves composition.
  Let
  $X=(x\overset{s}{\twoheadrightarrow}y\overset{t}{\twoheadrightarrow}z)\in
  \Ob\functCat{\deltaBra{2}}{(C_{-})}$ be an arbitrary composable pair of morphisms.
  We have the following commutative diagram in
  $\Downstar(C)$ (in fact in $\int \catNerve^{{-},{+}}(C)$):
  \begin{equation}\label{eq:lem-downstar-last-localiz-factor-exist:degen-functorial-1}
    \begin{tikzcd}[column sep=huge]
      (\deltaBra{0},x)
      \arrow[r,"{(\delta^1_1,\idmor[x])}"]
      \arrow[rd,"{(\iota_0,\idmor[x])}"']
      \arrow[d,"{(\delta^1_1,\idmor[x])}"']
      &(\deltaBra{1},t\compos s)
      \arrow[d, bend right, "{(\delta^2_1,\idmor[t\compos s])}"]
      &(\deltaBra{0},z)
      \arrow[l, "{(\delta^1_0,\idmor[z])}"']
      \arrow[dl, "{(\iota_2,\idmor[z])}"]
      \arrow[d, "{(\delta^1_0,\idmor[z])}"]\\
      (\deltaBra{1},s)
      \arrow[r, "{(\delta^2_2,\idmor[s])}"']
      &(\deltaBra{2},X)
      &(\deltaBra{1},t)
      \arrow[l, "{(\delta^2_0,\idmor[t])}"]\\
      &(\deltaBra{0},y)
      \arrow[ul, bend left, "{(\delta^1_0,\idmor[y])}"]
      \arrow[u, "{(\iota_1,\idmor[y])}"]
      \arrow[ur, bend right, "{(\delta^1_1,\idmor[y])}"']
      &
    \end{tikzcd}
  \end{equation}
  The functor $F$ sends this to the commutative diagram below, which demonstrates
  $\tilde{F}_{-}(t)\compos \tilde{F}_{-}(s) = \tilde{F}_{-}(t\compos s)$, as required:
  \begin{equation}\label{eq:lem-downstar-last-localiz-factor-exist:degen-functorial-2}
    \begin{tikzcd}[column sep=small, row sep=small]
      \tilde{F}_{-}(x)
      \arrow[dr, equal]
      \arrow[rrrr, "\tilde{F}_{-}(t\compos s)"]
      \arrow[rrdddd, bend right, "\tilde{F}_{-}(s)"']
      &&&&\tilde{F}_{-}(z)
      \arrow[dl, equal]\\
      &F(\deltaBra{0},x)
      \arrow[r]
      \arrow[rd]
      \arrow[d]
      &F(\deltaBra{1},t\compos s)
      \arrow[d, "\sim" sloped]
      &F(\deltaBra{0},z)
      \arrow[l,"\sim"']
      \arrow[dl,"\sim" sloped]
      \arrow[d,"\sim" sloped]&\\
      &F(\deltaBra{1},s)
      \arrow[r]
      &F(\deltaBra{2},X)
      &F(\deltaBra{1},t)
      \arrow[l, "\sim" sloped]&\\
      &&F(\deltaBra{0},y)
      \arrow[ul, "\sim" sloped]
      \arrow[u]
      \arrow[ur]
      &&\\
      &&\tilde{F}_{-}(y)
      \arrow[rruuuu, bend right, "\tilde{F}_{-}(t)"']
      \arrow[u, equal]
      &&
    \end{tikzcd}
  \end{equation}
  Thus we obtain that $\tilde{F}_{-}$ is a functor.

  Now we set
  $\tilde{F}(d\compos s) \coloneqq \tilde{F}_{+}(d)\compos\tilde{F}_{-}(s)$
  for any composable pair
  $x\overset{s}{\twoheadrightarrow}y\overset{d}{\rightarrowtail} z$ in $C$.
  The unique factorization of the Reedy structure of $C$ ensures that this
  gives a well-defined family of functions $\HomOf[C]{x}{z} \to \HomOf[D]{x}{z}$ for
  $x,z\in\Ob C$.

  We wish to verify the functoriality of $\tilde{F}$. 
  Consider any commutative square of the following form:
  \begin{equation}\label{eq:lem-downstar-last-localiz-factor-exist:funct-welldef-1}
    \begin{tikzcd}[ampersand replacement=\&]
      w \arrow[r, tail, "{d'}"] \arrow[d, two heads, "s"']
      \& x \arrow[d, two heads, "{s'}"] \\
      y \arrow[r, tail, "d"']
      \& z
    \end{tikzcd}
  \end{equation}
  The diagram \eqref{eq:lem-downstar-last-localiz-factor-exist:funct-welldef-1}
  represents a natural transformation
  $\phi_{d'd}\colon s \Rightarrow s'\colon \deltaBra{1}\to C$,
  which comprises a
  morphism $(\idmor[\deltaBra{1}],\phi_{d'd})\colon(\deltaBra{1},s)\to(\deltaBra{1},s')$
  in $\int\catNerve^{{-},{+}}(C)$.

  Now, notice that the desired functoriality follows from the outer square of the diagram
  \eqref{eq:lem-downstar-last-localiz-factor-exist:funct-welldef-3} below.
  However, in the following
  \eqref{eq:lem-downstar-last-localiz-factor-exist:funct-welldef-3},
  the commutativity of each inner small trapezoid is definitionally guaranteed,
  and that of each small triangle is obtained by simple calculation:
  \begin{equation}\label{eq:lem-downstar-last-localiz-factor-exist:funct-welldef-3}
    \begin{tikzcd}
      \tilde{F}(w) \arrow[dr, equal]
      \arrow[rrrrr, "{\tilde{F}_{+}(d')}"]
      \arrow[dddd, "{\tilde{F}_{-}(s)}"']
      &[-0.9em] &&&&[-0.9em] \tilde{F}(x) \arrow[dl, equal]
      \arrow[dddd, "{\tilde{F}_{-}(s')}"]
      \\[-0.9em]
      & F(\deltaBra{0},w)
      \arrow[rrr, "{F(\idmor[\deltaBra{0}], d')}"]
      \arrow[drrr, "{F(\delta^1_1, d')}"]
      \arrow[d, "{F(\delta^1_1, \idmor[w])}"']
      &&& F(\deltaBra{0},x)
      \arrow[d, "{F(\delta^1_1,\idmor[x])}"]
      &\\
      & F(\deltaBra{1}, s) 
      \arrow[rrr, "{F(\idmor[\deltaBra{1}], \phi_{d'd})}"{pos=0.35}]
      &&& {F(\deltaBra{1},s')}
      &\\
      & F(\deltaBra{0}, y)
      \arrow[u, "\sim"' sloped, "{F(\delta^1_0, \idmor[y])}"]
      \arrow[urrr, "{F(\delta^1_0, d)}"]
      \arrow[rrr, "{F(\idmor[\deltaBra{0}],d)}"']
      &&& F(\deltaBra{0}, z)
      \arrow[u, "\sim" sloped, "{F(\delta^1_0, \idmor[z])}"']
      &\\[-0.9em]
      \tilde{F}(y) \arrow[ur, equal] \arrow[rrrrr, "\tilde{F}_{+}(d)"']
      &&&&& \tilde{F}(z) \arrow[ul, equal]
    \end{tikzcd}
  \end{equation}
  Hence the outer square in
  \eqref{eq:lem-downstar-last-localiz-factor-exist:funct-welldef-3}
  is commutative, which completes the proof of the functoriality of $\tilde{F}$.

  Now we wish to define $\theta\colon \tilde{F}\compos\last\overset\sim\Rightarrow F$ by
  \[
    \theta_{(\deltaBra{n}, X)} \coloneqq F(\iota_n, \idmor[X\compos\iota_n])
    \colon \tilde{F}(\last(\deltaBra{n}, X)) = F(\deltaBra{0}, X(n))
    \overset\sim\longrightarrow F(\deltaBra{n},X).
  \]
  Since $(\iota_n, \idmor[X\compos\iota_n])$ is a $\last$-weak equivalence, the morphism
  above is an isomorphism. It just remains to demonstrate the naturality of $\theta$:
  for any $\boldsymbol{\alpha}\colon (\deltaBra{m},X)\to(\deltaBra{n},Y)$, we shall check
  the commutativity of the following diagram:
  \begin{equation}\label{eq:lem-downstar-last-localiz-factor-exist:nat}
    \begin{tikzcd}[column sep=large]
      \tilde{F}(\last(\deltaBra{m}, X))
      \arrow[r, "\theta_{(\deltaBra{m},X)}"]
      \arrow[d, "{\tilde{F}(\last(\boldsymbol{\alpha}))}"']
      & F(\deltaBra{m}, X)
      \arrow[d, "{F(\boldsymbol{\alpha})}"]
      \\
      \tilde{F}(\last(\deltaBra{n}, Y))
      \arrow[r, "{\theta_{(\deltaBra{n},Y)}}"']
      & F(\deltaBra{n},Y)
    \end{tikzcd}
  \end{equation}
  Let $(\alpha,\phi)\colon(\deltaBra{m},X)\to(\deltaBra{n},Y)$ in
  $\int\catNerve^{{-},{+}}(C)$ be any representative of the equivalence class
  $\boldsymbol{\alpha}\in\Mor\Downstar(C)$.
  By \cref{lem:down-groth-last-decp}, we may confine our consideration into two
  cases: the one where $\alpha(m) = n$ and the one where
  $X(m)=Y(\alpha(m))$ and $\phi_m = \idmor[X(m)]$.

  If $\alpha(m) = n$, then $\last(\boldsymbol{\alpha})=\phi_m\in\Mor C_{+}$;
  hence by definition we have:
  \[
    \tilde{F}(\last(\boldsymbol{\alpha}))=\tilde{F}_{+}(\phi_m)
    =F(\idmor[\deltaBra{0}], \phi_m)\colon F(\deltaBra{0}, X(m)) \to F(\deltaBra{0}, Y(n)).
  \]
  Now send the following commutative diagram under $F$ to obtain the commutativity of
  \eqref{eq:lem-downstar-last-localiz-factor-exist:nat}:
  \begin{equation}\label{eq:lem-downstar-last-localiz-factor-exist:nat-face}
    \begin{tikzcd}[column sep=huge]
      (\deltaBra{0}, X(m)) \arrow[d, "{(\idmor[\deltaBra{0}], \phi_m)}"']
      \arrow[r, "{(\iota_m, \idmor[X(m)])}"]
      & (\deltaBra{m}, X) \arrow[d, "{(\alpha,\phi)}"]
      \\
      (\deltaBra{0}, Y(n)) \arrow[r, "{(\iota_n,\idmor[Y(n)])}"']
      & (\deltaBra{n}, Y)
    \end{tikzcd}
  \end{equation}

  Finally, we consider the remaining case: $X(m)=Y(\alpha(m))$ and $\phi_m = \idmor[X(m)]$.
  Let us write
  $s\coloneqq Y(\uniqmor_{\alpha(m),n})\colon Y(\alpha(m))\twoheadrightarrow Y(n)$.
  The premise of the case implies that $\last(\boldsymbol{\alpha})=s$ and that
  $\tilde{F}(\last(\boldsymbol{\alpha}))=\tilde{F}_{-}(s)$.
  Let $\beta\colon \deltaBra{1} \to \deltaBra{n}$ in $\mtcSimpx$ be given by
  $\beta(0) = \alpha(m)$ and $\beta(1) = n$. By the defining diagram
  \eqref{eq:lem-downstar-last-localiz-factor-exist:img-degen-spec},
  the commutativity of
  \eqref{eq:lem-downstar-last-localiz-factor-exist:nat}
  follows from the image by $F$ of the following commutative diagram:
  \begin{equation}\label{eq:lem-downstar-last-localiz-factor-exist:nat-degen}
    \begin{tikzcd}[column sep=huge]
      (\deltaBra{0}, X(m)) \arrow[d, "{(\delta^0_1, \idmor[X(m)])}"']
      \arrow[r, "{(\iota_m, \idmor[X(m)])}"]
      & (\deltaBra{m}, X) \arrow[dd, "{(\alpha,\phi)}"]
      \\
      (\deltaBra{1}, s)\arrow[dr, "{(\beta,\idmor[s])}"]
      & \\
      (\deltaBra{0}, Y(n)) \arrow[r, "{(\iota_n,\idmor[Y(n)])}"']
      \arrow[u, "{(\delta^0_0,\idmor[Y(n)])}"]
      & (\deltaBra{n}, Y)
    \end{tikzcd}
  \end{equation}

  Now we have constructed the both that have been required.
\end{proof}

Now we wrap up the lemmas above to achieve our goal of this section:

\begin{proof}[Proof of \cref{thm:down-last-localiz}]\label{prf:thm-down-last-localiz}
  The case $\Gamma=\int \catNerve(C)$ is addressed in
  \cref{lem:groth-tot-last-localiz}.
  The case $\Gamma=\Downstar(C)$ follows from
  \cref{lem:down-groth-last-respects-weq,%
    lem:last-localiz-whisker-bij,prop:downstar-last-localiz-factor-exist}.
  By \cref{lem:down-groth-last-compat,lem:eqvce-down-c-downstar},
  the cases $\Gamma=\Downstar(C)$ and $\Gamma=\Down(C)$ are equivalent.

  The remaining case is $\Gamma=\int\catNerve^{{-},{+}}(C)$.
  In \cref{prop:downstar-strict-localization}, note that:
  \[
    \Mor \textstyle\int\catNerve^{{-},\set{\idmor}}_{-}(C) \subseteq
    \weqlast[\int\catNerve^{{-},{+}}(C)].
  \]
  Therefore, by \cref{lem:down-groth-last-compat}, the case in question
  is again equivalent to the case $\Gamma=\Downstar(C)$. This concludes the proof.
\end{proof}

\section{Shapes of \texorpdfstring{$(\infty,1)$}{(∞,1)}-diagrams used in the proof of
\texorpdfstring{$(\infty,1)$}{(∞,1)}-localization}
\label{sec:down-last-infty-loc-shape}

In the preceding discussion, we worked on finite or constructive foundation.
In this and the next section, we will work on a sufficiently strong classical
set theoretic foundation, like ZFC.

In \cref{thm:down-last-localiz} from the previous section, 
we proved that the functor $\last[\Gamma]$ is a $1$-localization functor for
$\Gamma=\int\catNerve(C),\,\int\catNerve^{{-},{+}}(C),\,\Downstar(C),\,\Down(C)$.
This section presents the preparatory technical arguments needed for
the next \cref{sec:down-last-infty-loc}, where we prove that $\last[\Gamma]$ is
an $(\infty,1)$-localization map for the same four instances of $\Gamma$.
To be more precise, the generalities discussed here will affect the proof for three of those four:
$\Gamma=\int\catNerve^{{-},{+}}(C),\,\Downstar(C),\,\Down(C)$.
The remaining case, $\Gamma=\int\catNerve(C)$, simply follows from the argument
of a reflexive subcategory.

More concretely speaking, in this section, we shall construct the following six endofunctors
on the category $\mtcSSet$ of simplicial sets, and study certain natural transformations between them:
\begin{itemize}
  \item $\mtsDcp$ and $\mtsDcpI$ (\cref{def:dcp-dcpI});
  \item $\mtsESd$ (\cref{def:esd}) and $\mtsESdI$ (\cref{def:esdi});
  \item $\mtsESdp$ and $\mtsESdpI$ (\cref{def:esdp-esdpi}).
\end{itemize}
The transformations fit into the following commutative diagram (see \cref{lem:dcp-esd-esdp-i-comm}):
\[
  \begin{tikzcd}[column sep=small]
    \mtsDcp X\times\set{0} \arrow[r, two heads] \arrow[d, hook]
    & \mtsESd X\times\set{0} \arrow[r, hook] \arrow[d, hook]
    & \mtsESdp X\times\set{0} \arrow[d, hook]
    & X\times\set{0} \arrow[l, hook'] \arrow[d, equal] \\
    \mtsDcpI X \arrow[r, two heads]
    & \mtsESdI X \arrow[r, hook]
    & \mtsESdpI X
    & X \times \set{0} \arrow[d, hook]\\
    X\times\set{1} \arrow[u, hook] \arrow[r, equal]
    & X\times\set{1} \arrow[u, hook] \arrow[r, equal]
    & X\times \set{1} \arrow[r, hook]
    & X \times \mtyndSpx\deltaBra{1} \arrow[lu, hook']
  \end{tikzcd}
\]
As we will see in \cref{subsec:down-last-infty-loc-shape-uloc,subsec:down-last-infty-loc-shape-inner-anod,%
subsec:down-last-infty-loc-shape-inner-anod2,subsec:down-last-infty-loc-shape-main-lem}, the maps
$\mtsDcp X \to \mtsESdp X$ and $\mtsDcpI X \to \mtsESdpI X$ are localization maps for any $X$.

These functors are used to describe the shapes of the $(\infty,1)$-diagrams
that appear in the proof that $\last[\Gamma]$ is an $(\infty,1)$-localization.
Specifically, in \cref{lem:down-and-orig-dcp-esdp-maps} from the next section, we will construct
the following simplicial maps:
\begin{itemize}
  \item $\mtsDcp\nerve(C) \to \nerve(\Gamma)$;
  \item $\mtsDcpI\nerve(\Gamma) \to \nerve(\Gamma)$;
  \item $\mtsESdp\nerve(C) \to \nerve(C)$;
  \item $\mtsESdpI\nerve(\Gamma) \to \nerve(C)$.
\end{itemize}
Suppose that a quasi-category-valued simplicial map $\nerve(\Gamma)\to Q$ sends
$\last$-weak equivalences to equivalences. Then the general properties of the transformations
investigated here will allow us to factor $\mtsDcp\nerve(C) \to \nerve(\Gamma) \to Q$ through
$\mtsESdp\nerve(C)$ and $\mtsDcpI\nerve(\Gamma) \to \nerve(\Gamma) \to Q$ through
$\mtsESdpI\nerve(\Gamma)$. This will give $\nerve(C) \to Q$ and $\nerve(\Gamma) \times \mtyndSpx\deltaBra{1} \to Q$,
which shows the factorization property of a localization map.

This section is organized as follows: in \cref{subsec:down-last-infty-loc-shape-functors},
we will define the six endofunctors on $\mtcSSet$ mentioned above. In \cref{subsec:down-last-infty-loc-shape-transf},
we will construct the natural transformations between them, as shown in the diagram above.
\Cref{subsec:down-last-infty-loc-shape-uloc,subsec:down-last-infty-loc-shape-inner-anod,subsec:down-last-infty-loc-shape-inner-anod2}
will be devoted to the investigation of specific properties of these transformations:
\begin{itemize}
  \item in \cref{subsec:down-last-infty-loc-shape-uloc}, we will show that $\mtsDcp X \to \mtsESd X$ and
  $\mtsDcpI X \to \mtsESdI X$ are universal localization maps;
  \item in \cref{subsec:down-last-infty-loc-shape-inner-anod},
  we will prove that $\mtsESd X \to \mtsESdp X$ is an inner anodyne map;
  \item in \cref{subsec:down-last-infty-loc-shape-inner-anod2}, we will demonstrate that
  $\mtsESdI X \cup_{\mtsESd X} \mtsESdp X \to \mtsESdpI X$ is an inner anodyne map.
\end{itemize}
\Cref{subsec:down-last-infty-loc-shape-main-lem}
will conclude this \cref{sec:down-last-infty-loc-shape} by introducing new notation and restating the results.
It will facilitate the reference from the next section.

\subsection{Some endofunctors on \texorpdfstring{$\mtcSSet$}{Set\_Δ}}
\label{subsec:down-last-infty-loc-shape-functors}

We shall begin by defining two simplicial sets that elaborate to higher dimensions
the commutative diagrams used in the proof of \cref{prop:downstar-last-localiz-factor-exist}.

\begin{definition}\label{def:dcp-dcpI}
  For any object $\deltaBra{n} \in \mtcSimpx$, we define the category $\mtcDcp\deltaBra{n}$
  as the full subcategory of the product category
  $\deltaBra{n} \times (\overcat{\mtcSimpx}{\deltaBra{n}})$
  spanned by the objects $(x,\alpha\colon\deltaBra{k} \to \deltaBra{n})$ such that $x \le \alpha(0)$.
  Also, we define the category $\mtcDcpI\deltaBra{n}$ by the following formulae:
  \begin{align*}
    \Ob(\mtcDcpI\deltaBra{n}) \coloneqq \Ob(\mtcDcp\deltaBra{n}) \amalg \deltaBra{n} 
    &= (\set{0} \times \Ob\mtcDcp\deltaBra{n}) \cup (\set{1} \times \deltaBra{n});\\
    \HomOf[\mtcDcpI\deltaBra{n}]{(0,p)}{(0,q)} &\coloneqq \HomOf[\mtcDcp\deltaBra{n}]{p}{q};\\
    \HomOf[\mtcDcpI\deltaBra{n}]{(0,(x,\alpha\colon\deltaBra{k}\to\deltaBra{n}))}{(1,y)}
    &\coloneqq\HomOf[\deltaBra{n}]{\alpha(k)}{y};\\
    \HomOf[\mtcDcpI\deltaBra{n}]{(1,x)}{(0,q)} &\coloneqq \emptyset;\\
    \HomOf[\mtcDcpI\deltaBra{n}]{(1,x)}{(1,y)} &\coloneqq \HomOf[\deltaBra{n}]{x}{y}.
  \end{align*}
  Here, the composition of morphisms is clear, for every composition has its value
  in a singleton set or is of the form $(0,p)\to(0,q)\to(0,r)$.
  In an evident way, 
  the categories $\mtcDcp\deltaBra{n}$ and $\mtcDcpI\deltaBra{n}$ assemble into
  cosimplicial categories 
  \[ 
    {\mtcDcp},{\mtcDcpI}\colon\mtcSimpx\to\mtcCat.
  \]
  We shall denote by ${\mtsDcp},{\mtsDcpI}\colon\mtcSSet\to\mtcSSet$, respectively,
  the left Kan extensions of ${\nerve}\compos{\mtcDcp},{\nerve}\compos{\mtcDcpI}
  \colon\mtcSimpx\to\mtcSSet$ along
  the Yoneda embedding $\mtyndSpx\colon\mtcSimpx\to\mtcSSet$.
\end{definition}

\begin{remark}
  Explicitly, for any simplicial set $X$, we have
  \begin{alignat*}{2}
    \mtsDcp X
    &= \int^{\deltaBra{n}\in\mtcSimpx}N(\mtcDcp\deltaBra{n})\times X_n
    &&= \mtColim_{\mtyndSpx\deltaBra{n}\to X} N(\mtcDcp\deltaBra{n});\\
    \mtsDcpI X
    &= \int^{\deltaBra{n}\in\mtcSimpx}N(\mtcDcpI\deltaBra{n})\times X_n
    &&= \mtColim_{\mtyndSpx\deltaBra{n}\to X} N(\mtcDcpI\deltaBra{n}).
  \end{alignat*}
\end{remark}

In \cref{prop:downstar-last-localiz-factor-exist}, we have constructed
a functor $\tilde{F}\colon C\to D$ and a natural isomorphism
$\theta\colon\tilde{F}\compos\last\overset\sim\Rightarrow F$ from a functor
$F\colon\Downstar(C)\to D$ that sends every $\last$-weak equivalence to an
isomorphism. The simplicial set $\mtsDcp(\nerve(C))$ captures the shape
of the diagrams used for the construction of $\tilde{F}$, namely
\eqref{eq:lem-downstar-last-localiz-factor-exist:degen-functorial-2} and
\eqref{eq:lem-downstar-last-localiz-factor-exist:funct-welldef-3}.
The construction of $\theta$, which used the commutative diagrams
\eqref{eq:lem-downstar-last-localiz-factor-exist:nat-face} and
\eqref{eq:lem-downstar-last-localiz-factor-exist:nat-degen}, is covered by
the simplicial set $\mtsDcpI(\nerve(\Downstar(C)))$.

For a good control of compositions and inverses in quasi-categories,
we need a few more endofunctors of $\mtcSSet$. We go on to define them.

\begin{definition}\label{def:esd}
  For the purpose of this definition, let $F\colon \mtcSimpx\to\mtcSimpx$ denote
  the functor $F\deltaBra{n} = \deltaBra{n}\star\deltaBra{n}=\deltaBra{2n+1}$.
  We shall define the endofunctor
  ${\mtsESd}\colon\mtcSSet\to\mtcSSet$ as the precomposition of 
  $\dualCat{F}\colon\dualCat{\mtcSimpx}\to\dualCat{\mtcSimpx}$:
  for any simplicial set $X$, we have $\mtsESd(X)_n = X_{2n+1}$.
\end{definition}

\begin{remark}
  The endofunctor ${\mtsESd}\colon\mtcSSet\to\mtcSSet$ is not new: it is the
  composition of the two well-known functors. Let $\mtcSSSet$ denote
  the category of bisimplicial sets: set-valued presheaves on the
  product category $\mtcSimpx\times\mtcSimpx$.
  The \emph{total d\'ecalage functor} $\mathrm{Dec}\colon\mtcSSet\to\mtcSSSet$ and
  the \emph{diagonal functor} $\mathrm{diag}\colon\mtcSSSet\to\mtcSSet$ are given by
  the precomposition of the \emph{join} ${\star}\colon \mtcSimpx\times\mtcSimpx\to\mtcSimpx$
  and the \emph{diagonal inclusion} $\mtcSimpx\to\mtcSimpx\times\mtcSimpx$.
  It is easy to see that these two functors compose to produce our ${\mtsESd}$.
\end{remark}

Since the left and the right Kan extensions of $\dualCat{F}$ gives the left
and the right adjoints of $\mtsESd$, note that $\mtsESd$ preserves limits
and colimits, as should also be clear from the formula
$\mtsESd(X)_n = X_{2n+1}$.

\begin{definition}\label{def:esdi}
  We define the cosimplicial simplicial set
  ${\mtsESdIc}\colon\mtcSimpx\to\mtcSSet$ by
  \[
    \mtsESdIc\deltaBra{n}_k \coloneqq
    \coprod_{\substack{I\sqcup J = \deltaBra{k}\\I < J}}
    \HomOf[\mtcPoset]{I\star I\star J}{\deltaBra{n}}.
  \]
  Here, the disjoint union $\amalg$ is taken over all the partitions
  $\deltaBra{k} = I \sqcup J$ satisfying
  $i < j$ for any $i \in I$ and $j \in J$.
  We define the endofunctor ${\mtsESdI}\colon\mtcSSet\to\mtcSSet$
  as the left Kan extension of $\mtsESdIc$ along the Yoneda embedding
  $\mtyndSpx\colon\mtcSimpx\to\mtcSSet$.
\end{definition}

With this definition, the set of $n$-simplices of $\mtsESdI X$ can be expressed as
the following disjoint union:
\[
  (\mtsESdI X)_n \cong \coprod_{i=0}^{n+1} X_{n+i}.
\]

\begin{definition}\label{def:esdp-esdpi}
  Let us define the cosimplicial posets
  ${\mtpESdp},{\mtpESdpI}\colon\mtcSimpx\to\mtcPoset$ by:
  \begin{align*}
    \mtpESdp\deltaBra{n} &\coloneqq \functCat{\deltaBra{1}}{\deltaBra{n}};\\
    \mtpESdpI\deltaBra{n} &\coloneqq
    (\set{0}\times\functCat{\deltaBra{1}}{\deltaBra{n}}) \cup
    (\set{1}\times\deltaBra{n}).
  \end{align*}
  Here, the latter set is regarded as a poset by pulling back the order through the following
  injective map, whose codomain is endowed with the product order:
  \begin{align*}
    \mtpESdpI\deltaBra{n} &\to \deltaBra{1}\times\functCat{\deltaBra{1}}{\deltaBra{n}};\\
    (0,f) &\mapsto (0,f);\\
    (1,x) &\mapsto (1,\mathrm{const}_x).
  \end{align*}
  We define the endofunctors ${\mtsESdp},{\mtsESdpI}\colon\mtcSSet\to\mtcSSet$
  as the left Kan extensions along the Yoneda embedding
  $\mtyndSpx\colon\mtcSimpx\to\mtcSSet$ 
  of the post-compositions of ${\mtpESdp}$ and ${\mtpESdpI}$
  with the embedding $\mtcPoset\hookrightarrow\mtcCat\hookrightarrow\mtcSSet$.
\end{definition}

Let us explain the shapes and the intended meanings of the simplicial sets we have defined,
emphasizing lower dimensions. We begin with $\mtsDcp$. Since we would like to show that
$\nerve({\last})\colon\nerve(\Down(C))\to \nerve(C)$ is a localization map, we need to
construct a map $\nerve(C)\to Q$, given a quasi-category-valued simplicial map
$\nerve(\Down(C))\to Q$ that sends $\last$-weak equivalences to equivalences.
This endofunctor $\mtsDcp$ is designed to make $\mtsDcp(\nerve(C))$
capture how simplices in the nerve $\nerve(C)$ of our fixed Reedy category $C$ should
be mapped to $\nerve(\Down(C))$, considering $\last$-weak equivalences as invertibles.
Let us say we have a $2$-simplex in $\nerve(C)$: a commutative triangle $X_0 \to X_1 \to X_2$
in $C$. Assume that the following is the Reedy factorization of the triangle:
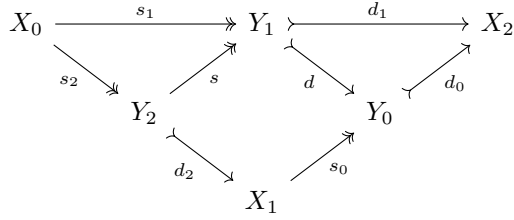
\[
  \begin{tikzcd}
    X_0 \arrow[rr, two heads, "s_1"] \arrow[dr, two heads, "s_2"']
    && Y_1 \arrow[rr, tail, "d_1"] \arrow[dr, tail, "d"'] && X_2 \\
    & Y_2 \arrow[dr, tail, "d_2"'] \arrow[ur, two heads, "s"'] && Y_0 \arrow[ur, tail, "d_0"'] & \\
    && X_1 \arrow[ur, two heads, "s_0"'] &&
  \end{tikzcd}
\]
Let us write:
\[
  S \coloneqq \left(X_0 \overset{s_2}{\twoheadrightarrow} Y_2 
  \overset{s}{\twoheadrightarrow} Y_1\right)
\]
for the length-$2$ composable chain in $C$, representing an object in $\Downstar(C)$.
The factorized triangle $X_0 \to X_1 \to X_2$ should be mapped to the following in $\Downstar(C)$:
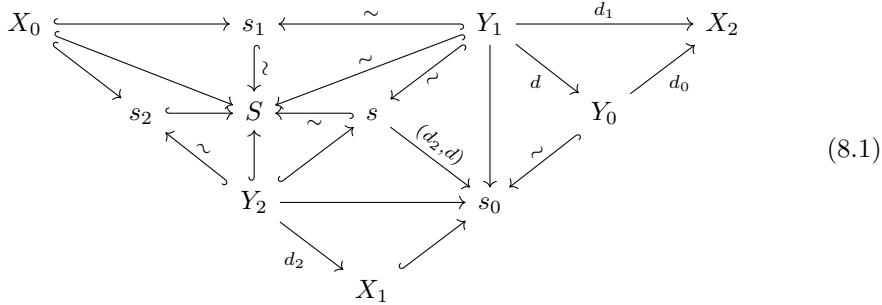
\begin{equation}\label{eq:diagram-dcp-2-simplex-concrete}
  \begin{tikzcd}
    X_0 \arrow[rr, hook] \arrow[dr, hook] \arrow[drr, hook]
    && s_1 \arrow[d, hook, "\sim" sloped]
    && Y_1 \arrow[ll, hook, "\sim"'] \arrow[dl, hook, "\sim"' sloped] \arrow[dll, hook, "\sim" sloped] \arrow[dr, "d"'] \arrow[rr, "d_1"] \arrow[dd]
    && X_2 \\
    & s_2 \arrow[r, hook] & S & s \arrow[l,hook, "\sim"] \arrow[dr, "{(d_2,d)}" sloped] && Y_0 \arrow[ur, "d_0"'] \arrow[dl, hook, "\sim" sloped] &\\
    && Y_2 \arrow[ul, hook, "\sim" sloped] \arrow[ur, hook] \arrow[dr, "d_2"'] \arrow[u, hook] \arrow[rr] && s_0 && \\
    &&& X_1 \arrow[ur, hook] &&&
  \end{tikzcd}
\end{equation}
Here, the labels for objects and morphisms are abbreviated for simplicity.
Also note that this diagram includes the diagrams 
\labelcref{eq:lem-downstar-last-localiz-factor-exist:funct-welldef-3,%
eq:lem-downstar-last-localiz-factor-exist:degen-functorial-2}
from the proof of \cref{prop:downstar-last-localiz-factor-exist}.

The diagram \eqref{eq:diagram-dcp-2-simplex-concrete} is, itself, not a single $2$-simplex in
$\nerve(\Downstar(C))$. However, since all arrows labeled with $\sim$ are $\last$-weak equivalences,
if we consider these arrows as invertible, we can compute
the ``composition'' of this diagram to yield a valid $2$-simplex.

In order to describe the shape of the concrete diagram \eqref{eq:diagram-dcp-2-simplex-concrete}, 
we consider the following $2$-dimensional simplicial set, which is an incomplete prototype of $\mtsDcp\mtyndSpx\deltaBra{2}$:
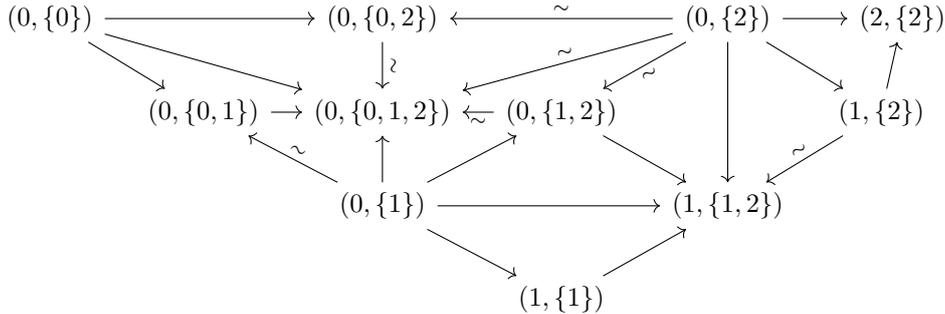
\[
  \begin{tikzcd}[column sep=small]
    (0,\set{0}) \arrow[rr] \arrow[dr] \arrow[drr]
    && (0,\set{0,2}) \arrow[d, "\sim" sloped]
    && (0,\set{2}) \arrow[ll, "\sim"'] \arrow[dl, "\sim"' sloped] \arrow[dll, "\sim" sloped] \arrow[dr] \arrow[rr] \arrow[dd]
    &&[-4.5em] (2,\set{2}) \\
    & (0,\set{0,1}) \arrow[r] 
    & (0,\set{0,1,2}) 
    & (0,\set{1,2}) \arrow[l, "\sim"] \arrow[dr] && (1,\set{2}) \arrow[ur] \arrow[dl, "\sim" sloped] &\\
    && (0,\set{1}) \arrow[ul, "\sim" sloped] \arrow[ur] \arrow[dr] \arrow[u] \arrow[rr] 
    && (1,\set{1,2}) && \\
    &&& (1,\set{1}) \arrow[ur] &&&
  \end{tikzcd}
\]
Here, $\sim$ on an edge indicates that the corresponding morphism in \eqref{eq:diagram-dcp-2-simplex-concrete}
is a $\last$-weak equivalence, meaning it preserves the left coordinate and the maximum
of the right coordinate. If we were to formulate the exact shape of the diagram as $\mtsDcp\mtyndSpx\deltaBra{2}$,
we would have the following:
\begin{align*}
  \text{``}\mtpDcp\deltaBra{n}\text{''} &\overset{?}{\coloneqq} \set{(x,\alpha)}[x \in \deltaBra{n},
     \emptyset\neq\alpha\subseteq\deltaBra{n}, x\le\min\alpha]\text{ poset};\\
  \text{``}\mtsDcp\mtyndSpx\deltaBra{n}\text{''} &\overset{?}{\coloneqq} \nerve(\mtpDcp\deltaBra{n}).
\end{align*}
However, for a better interaction with degeneracies, we would like to consider as $\alpha$ general
$\deltaBra{n}$-valued morphisms in $\mtcSimpx$, not just the $\deltaBra{n}$-valued injectives; hence the construction
in \cref{def:dcp-dcpI}. This makes the simplicial set $\mtsDcp\mtyndSpx\deltaBra{2}$ higher dimensional than $2$,
and too visually complex to be depicted here, but we can more easily manipulate it in our proofs.

In each vertex $(x,\alpha)\in\Ob(\mtcDcp\deltaBra{n})$ of $\mtsDcp\mtyndSpx\deltaBra{n}$,
its right coordinate $\alpha$ par\-a\-metrizes the length of a chain in $C_{-}$, which is an object in
$\Downstar(C)$, and the left coordinate $x$ controls the Reedy factorization from which the chain is obtained.
More rigourously put, if $X_0 \to X_1 \to \cdots \to X_n$ is an $n$-simplex in $\nerve(C)$, then
the vertex $(x,\alpha\colon\deltaBra{k}\to\deltaBra{n})$ in $\mtsDcp(\mtyndSpx\deltaBra{n})$
corresponds to the chain $Z_0 \twoheadrightarrow \cdots \twoheadrightarrow Z_k$ below in $C_{-}$,
representing an object in $\Downstar(C)$:
\[
\begin{tikzcd}
  X_x \arrow[r] \arrow[dr, two heads]
  & X_{\alpha(0)} \arrow[r]
  & X_{\alpha(1)} \arrow[r]
  & \cdots \arrow[r]
  & X_{\alpha(k)} \\
  & Z_0 \arrow[r, two heads] \arrow[u, tail]
  & Z_1 \arrow[r, two heads] \arrow[u, tail]
  & \cdots \arrow[r, two heads]
  & Z_k \arrow[u, tail]
\end{tikzcd}
\]
This glues together to form a simplicial map
$\mtsDcp(\nerve(C))\to\nerve(\Downstar(C))$, as shall be demonstrated in \cref{lem:down-and-orig-dcp-esdp-maps}.

In order to compose simplices forming diagrams like \eqref{eq:diagram-dcp-2-simplex-concrete},
we need to properly treat the reverse-direction simplices. For that purpose,
we simply collapse the arrows with $\sim$ and simplices containing them into degeneracies.
If we do so in $\mtsDcp(\mtyndSpx\deltaBra{2})$, then we obtain $\mtsESd(\mtyndSpx\deltaBra{2})$,
which has the following shape: 
\begin{equation}\label{eq:esd-2-simplex-example}
  \begin{tikzcd}
    (0, 0) \arrow[rr] \arrow[dr]
    && (0, 2) \arrow[rr]\arrow[dr]
    && (2, 2) \\
    & (0, 1)\arrow[ur] \arrow[rr]\arrow[dr]
    && (1, 2) \arrow[ur]
    &\\
    && (1, 1) \arrow[ur]
    &&
  \end{tikzcd}
\end{equation}
Here, a vertex $(x,\alpha)$ in $\mtsDcp(\mtyndSpx\deltaBra{2})$ is collapsed into the vertex
$(x,\max\alpha)$. This collapse works as intended. This generalizes to higher dimensions and glues together
to form a natural simplicial map $\mtsDcp X \to \mtsESd X$ for any simplicial set $X$. This natural map is a
localization map, as we will show in \cref{prop:dcp-esd-connecting-surj-univ-loc}.

We are left only with forward-direction simplices in $\mtsESd$, but the composition of simplices in
$\mtsESd(\mtyndSpx\deltaBra{n})$ is not specified in the simplicial set: for example,
note that the non-degenerate simplices
in $\mtsESd(\mtyndSpx\deltaBra{2})$ are the vertices, the arrows, and the four small triangle explicitly
depicted in the diagram \eqref{eq:esd-2-simplex-example}.
Therefore, even if we have a map with domain $\mtsESd(\nerve(C))$, the values of simplices
in $\nerve(C)$ is not determined. We need an extension
that contains the ``composition'' of these simplices. This is the job $\mtsESdp$ does. The simplicial set
$\mtsESdp(\mtyndSpx\deltaBra{n}) = \nerve(\functCat{\deltaBra{1}}{\deltaBra{n}})$ is obtained by considering
a poset that is obtained by all lacking compositions in $\mtsESd(\mtyndSpx\deltaBra{n})$. This makes
simplicial set $\mtsESdp(\mtyndSpx\deltaBra{n})$ a quasi-category for each $n$. Most desirably,
for each $X$, we have natural inclusions $X \hookrightarrow \mtsESdp X$ and $\mtsESd X \hookrightarrow \mtsESdp X$,
the latter of which is, as shown in \cref{prop:esd-esdp-conn-inner-anod}, inner anodyne.

The three endofunctors $\mtsDcp$, $\mtsESd$, and $\mtsESdp$ are used to describe the shapes of the
diagram used to construct a factorizing functor $\nerve(C) \to Q$ from a functor $\nerve(\Down(C)) \to Q$.
The construction of the natural transformation that ensures the functor is a factorization,
requires the other three endofunctors: $\mtsDcpI$, $\mtsESdI$, and $\mtsESdpI$.
These endofunctors with ``I'' evaluated at a simplicial set $X$ serve as a cylinder that connects
the value at $X$ of the version without ``I'' and the original simplicial set $X$.
For example, $\mtsESdI(\mtyndSpx\deltaBra{1})$ has the following shape:
\[
  \begin{tikzcd}
    (0, (0,0)) \arrow[r] \arrow[d] \arrow[rrd]
    & (0, (0,1)) \arrow[rd] \arrow[r]
    & (0, (1,1)) \arrow[d] \\
    (1, 0) \arrow[rr]
    && (1, 1)
  \end{tikzcd}
\]

To conclude this subsection, we shall check an important property of the functors we have defined:

\begin{lemma}\label{lem:dcp-esd-esdp-preserves-mono}
  The functors $\mtsDcp$, $\mtsDcpI$, $\mtsESd$, $\mtsESdI$, $\mtsESdp$,
  and $\mtsESdpI$ preserve monomorphisms.
\end{lemma}
\begin{proof}
  By \cref{cor:unaug-cospx-kan-ext-inj}, the claim follows from the construction
  of the functors in question.
\end{proof}

\subsection{Some natural transformations between the endofunctors}
\label{subsec:down-last-infty-loc-shape-transf}

We shall now construct some natural transformations between the endofunctors
we have defined. They should fit into the following commutative diagram
(see \cref{lem:dcp-esd-esdp-i-comm}):
\begin{equation} \label{eq:dcp-esd-esdp-i-comm}
  \begin{tikzcd}
    \mtsDcp X\times\set{0} 
    \arrow[r, two heads, "\labelcref{def:dcp-esd-connecting-surj}"]
    \arrow[d, hook, "\labelcref{def:dcp-esd-esdp-incl-to-i}"']
    & \mtsESd X\times\set{0} \arrow[r, hook, "\labelcref{def:esd-esdp-connecting-inj}"]
    \arrow[d, hook, "\labelcref{def:dcp-esd-esdp-incl-to-i}"']
    & \mtsESdp X\times\set{0} \arrow[d, hook, "\labelcref{def:dcp-esd-esdp-incl-to-i}"']
    & X\times\set{0} \arrow[l, hook', "\labelcref{def:id-esdp-connecting-inj}"'] \arrow[d, equal] \\
    \mtsDcpI X \arrow[r, two heads, "\labelcref{def:dcp-esd-connecting-surj}"]
    & \mtsESdI X \arrow[r, hook, "\labelcref{def:esd-esdp-connecting-inj}"]
    & \mtsESdpI X
    & X \times \set{0} \arrow[d, hook]\\
    X\times\set{1} \arrow[u, hook, "\labelcref{def:id-times-1-incl-to-dcpi-esdi}"] \arrow[r, equal]
    & X\times\set{1} \arrow[u, hook, "\labelcref{def:id-times-1-incl-to-dcpi-esdi}"] \arrow[r, equal]
    & X\times \set{1} \arrow[r, hook]
    & X \times \mtyndSpx\deltaBra{1} \arrow[lu, hook', "\labelcref{def:id-esdp-connecting-inj}"']
  \end{tikzcd}
\end{equation}
Here, the numbers stand for the definition where the natural transformation is constructed.

\begin{definition}\label{def:dcp-esd-connecting-surj}
  Let $X\in \mtcSSet$ be a simplicial set. Below we shall construct surjective
  simplicial maps $\mtsDcp X \twoheadrightarrow \mtsESd X$ and
  $\mtsDcpI X \twoheadrightarrow \mtsESdI X$, which are natural in $X$.
  We shall consider these natural maps as the canonical connecting maps between
  those simplicial sets.
\end{definition}
\begin{proof}[Construction]
  Since the functors in question are all colimit-preserving and
  colimits preserve surjections, we may confine
  our attention to the case where $X$ is a representable simplicial set
  $\mtyndSpx\deltaBra{n}$. We first construct
  $\mtsDcp\mtyndSpx\deltaBra{n}\twoheadrightarrow\mtsESd\mtyndSpx\deltaBra{n}$.
  Let $\varphi\colon\deltaBra{n}\to\mtcDcp\deltaBra{n}$ be a $k$-simplex in
  $\mtsDcp\mtyndSpx\deltaBra{n}=N(\mtcDcp\deltaBra{n})$. Write:
  \begin{alignat*}{2}
    (x_i,\alpha_i\colon\deltaBra{m_i}\to\deltaBra{n}) &\coloneqq \varphi(i)
    &\quad&\text{for }0\le i\le k;\\
    (\uniqmor_{x_i,x_j},\beta_{i,j}\colon\alpha_i\to\alpha_j) &\coloneqq \varphi(\uniqmor_{i,j})
    &\quad&\text{for }0\le i \le j \le k.
  \end{alignat*}
  Since $\beta_{i,j}$ above is a morphism in $\overcat{\mtcSimpx}{\deltaBra{n}}$, we deduce:
  \begin{gather*}
    x_k \le \alpha_k(0) \le \alpha_k(\beta_{0,k}(0)) = \alpha_0(0) \le \alpha_0(m_0); \\
    \alpha_{i}(m_{i}) = \alpha_{j}(\beta_{i,j}(m_{i})) \le \alpha_j(m_j) \quad\text{for }1\le i\le j\le k.
  \end{gather*}
  Therefore, the following chain of inequalities holds:
  \[
    x_0 \le x_1 \le \cdots \le x_k
    \le \alpha_0(m_0) \le \alpha_1(m_1) \le \cdots \le \alpha_k(m_k).
  \]
  The above $(2k+2)$-chain in $\deltaBra{n}$ represents a $k$-simplex in
  $\mtsESd\mtyndSpx\deltaBra{n}$. This defines a
  simplicial map $\mtsDcp\mtyndSpx\deltaBra{n}\to\mtsESd\mtyndSpx\deltaBra{n}$,
  natural in $\deltaBra{n}$, which is surjective by construction.

  Next, we construct
  $\mtsDcpI\mtyndSpx\deltaBra{n}\twoheadrightarrow\mtsESdI\mtyndSpx\deltaBra{n}$.
  Let $\varphi\colon \deltaBra{k}\to\mtcDcpI\deltaBra{n}$ be a $k$-simplex in
  $\mtsDcpI\mtyndSpx\deltaBra{n} = N(\mtcDcpI\deltaBra{n})$.
  We wish to assign to $\varphi$ an element of
  \[
    \mtsESdIc\deltaBra{n}_k = \coprod_{\substack{I\sqcup J = \deltaBra{k}\\I < J}}
    \HomOf[\mtcPoset]{I\star I\star J}{\deltaBra{n}}.
  \]
  Let the subposet $I=\set{0,1,\ldots,l-1}$ and $J=\set{l,l+1,\ldots,k}$
  of $\deltaBra{k}$, respectively, be the inverse images of $\set{0}\times\Ob\mtcDcp\deltaBra{n}$
  and $\set{1}\times\deltaBra{n}$ under $\varphi$. Write:
  \begin{alignat*}{2}
    (0,(x_i,\alpha_i\colon\deltaBra{m_i}\to\deltaBra{n})) &\coloneqq \varphi(i)
    &\quad&\text{for }i \in I;\\
    (1,y_i) &\coloneqq \varphi(i) &\quad&\text{for }i \in J.
  \end{alignat*}
  From the same reasoning as above, we deduce that the following chain of inequalities holds,
  yielding an order-preserving map $I\star I\star J\to\deltaBra{n}$:
  \begin{multline*}
    x_0 \le x_1 \le \cdots \le x_{l-1} \\
    \le \max\alpha_0 \le \max\alpha_1 \le \cdots \le \max\alpha_{l-1} \\
    \le y_l \le y_{l+1} \le \cdots \le y_k.
  \end{multline*}
  This constitutes an element of $\mtsESdI\mtyndSpx\deltaBra{n}_k = \mtsESdIc\deltaBra{n}_k$
  above, and tying these corresponding elements together gives a simplicial map
  $\mtsDcpI\mtyndSpx\deltaBra{n}\to\mtsESdI\mtyndSpx\deltaBra{n}$, naturally in $\deltaBra{n}$,
  which is surjective by construction.
\end{proof}

\begin{definition}\label{def:esd-esdp-connecting-inj}
  Let $X\in \mtcSSet$ be a simplicial set. In the subsequent 
  Construction, we shall give injective
  simplicial maps $\mtsESd X \hookrightarrow \mtsESdp X$ and
  $\mtsESdI X \hookrightarrow \mtsESdpI X$, which are natural in $X$.
  We shall think of these natural transformations as the canonical
  ones between the involved simplicial sets.
\end{definition}
\begin{proof}[Construction]
  Again, we may confine our attention to the case where $X$ is a representable
  simplicial set $\mtyndSpx\deltaBra{n}$. We first construct
  $\mtsESd\mtyndSpx\deltaBra{n}\hookrightarrow\mtsESdp\mtyndSpx\deltaBra{n}$.
  For any $\deltaBra{k} \in \mtcSimpx$, consider the following order-preserving bijection
  (which is not an isomorphism of posets):
  \[
    u_{\deltaBra{k}}\colon \deltaBra{1}
    \times \deltaBra{k} \to \deltaBra{k} \star \deltaBra{k}=\deltaBra{2k+1};
    \,\, (i, x) \mapsto i*(k+1) + x.
  \]
  This is natural in $\deltaBra{k}$. Let an order-preserving map
  $f\colon \deltaBra{k}\star \deltaBra{k}\to\deltaBra{n}$ represent a $k$-simplex
  in $\mtsESd\mtyndSpx\deltaBra{n}$. The composite
  $f\compos u_{\deltaBra{k}}\colon \deltaBra{1}\times\deltaBra{k}\to\deltaBra{n}$
  has the adjunct map $\deltaBra{k}\to\functCat{\deltaBra{1}}{\deltaBra{n}}$,
  which is an element of
  $\mtsESdp\mtyndSpx\deltaBra{n}_k=\nerve(\mtpESdp\deltaBra{n})_k$.
  By corresponding this adjunct to the $k$-simplex $f$, we obtain a simplicial map
  $\mtsESd\mtyndSpx\deltaBra{n}\hookrightarrow\mtsESdp\mtyndSpx\deltaBra{n}$,
  naturally in $\deltaBra{n}$. This family of maps induces a natural transformation
  ${\mtsESd}\Rightarrow\mtsESdp$.
  By \cref{cor:unaug-cospx-kan-ext-nat-inj},
  the injectivity of the induced transformation follows from that of each
  $\mtsESd\mtyndSpx\deltaBra{n}\hookrightarrow\mtsESdp\mtyndSpx\deltaBra{n}$,
  which is clear.

  Next, we construct $\mtsESdI\mtyndSpx\deltaBra{n}\hookrightarrow\mtsESdpI\mtyndSpx\deltaBra{n}$.
  In order to arbitrarily take a $k$-simplex in $\mtsESdI\mtyndSpx\deltaBra{n}$,
  let $\deltaBra{k}=I\sqcup J$ be a partition with $I < J$, and let
  $f\colon I\star I\star J\to\deltaBra{n}$ be a map. By applying the above argument to
  $\left. f\right|_{I\star I}\colon I \star I \to \deltaBra{n}$, we may obtain a map
  $I \to \functCat{\deltaBra{1}}{\deltaBra{n}}$. Bundling this with the map
  $\left. f\right|_{J}\colon J \to \deltaBra{n}$ gives a map
  \[
    \deltaBra{k} = I \sqcup J = I \star J
    \to \set{0} \times \functCat{\deltaBra{1}}{\deltaBra{n}} \cup \set{1} \times \deltaBra{n}
    = \mtpESdpI\deltaBra{n}.
  \]
  This gives a $k$-simplex in $\mtsESdpI\mtyndSpx\deltaBra{n}=N(\mtpESdpI\deltaBra{n})$.
  We have thus constructed a simplicial map
  $\mtsESdI\mtyndSpx\deltaBra{n}\hookrightarrow\mtsESdpI\mtyndSpx\deltaBra{n}$,
  naturally in $\deltaBra{n}$.
  
  The injectivity of $\mtsESdI X \to \mtsESdpI X$ for any simplicial set $X$ may be shown
  by \cref{cor:unaug-cospx-kan-ext-nat-inj}.
\end{proof}

\begin{definition}\label{def:id-esdp-connecting-inj}
  For a simplicial set $X \in \mtcSSet$, we will demonstrate the construction of
  canonical injective simplicial maps $X \hookrightarrow \mtsESdp X$ and
  $X\times\mtyndSpx\deltaBra{1} \hookrightarrow \mtsESdpI X$,
  which are natural in $X$. These constructions will be referred to as the canonical
  connecting maps between these simplicial sets.
\end{definition}
\begin{proof}[Construction]
  For the third time, we may consider only the case where $X$ is a representable
  simplicial set $\mtyndSpx\deltaBra{n}$. The map
  $\mtyndSpx\deltaBra{n}\hookrightarrow\mtsESdp\mtyndSpx\deltaBra{n}$
  may be obtained by sending the constant inclusion
  \[ 
    \mathrm{const}\colon\deltaBra{n} \hookrightarrow \deltaBra{n}^{\deltaBra{1}};\,\,
    x \mapsto \mathrm{const}_{x} = (x,x),
  \]
  which is natural in $\deltaBra{n}$,
  with the nerve functor $\nerve\colon \mtcPoset\to\mtcCat\to\mtcSSet$.

  In order to construct the map
  $\mtyndSpx\deltaBra{n}\times\mtyndSpx\deltaBra{1}\hookrightarrow\mtsESdpI\mtyndSpx\deltaBra{n}$,
  condider the following inclusion of posets:
  \begin{align*}
    \deltaBra{n} \times \deltaBra{1} &\hookrightarrow \mtpESdpI\deltaBra{n}
    = \set{0} \times \functCat{\deltaBra{1}}{\deltaBra{n}} \cup \set{1} \times \deltaBra{n};\\
    (x,0) &\mapsto (0,\mathrm{const}_x),\\
    (x,1) &\mapsto (1,x).
  \end{align*}
  This is natural in $\deltaBra{n}$.
  We may then apply the nerve functor to this inclusion to obtain the desired map.
  
  The injectivity of these maps follows from \cref{cor:unaug-cospx-kan-ext-nat-inj}, as before.
\end{proof}

\begin{definition}\label{def:dcp-esd-esdp-incl-to-i}
  Given a simplicial set $X\in \mtcSSet$, we aim to construct canonical injective
  simplicial maps $\mtsDcp X \hookrightarrow \mtsDcpI X$,
  $\mtsESd X \hookrightarrow \mtsESdI X$, and
  $\mtsESdp X \hookrightarrow \mtsESdpI X$.
  These constructions, detailed below, are natural in $X$ and
  will serve as the canonical natural transformation between the involved functors.
\end{definition}
\begin{proof}[Construction]
  For the fourth time, we may confine our attention to the case where $X$ is a representable
  simplicial set $\mtyndSpx\deltaBra{n}$.
  $\mtcDcpI\deltaBra{n}$ and $\mtpESdpI\deltaBra{n}$ are defined in such a way that
  there is an obvious inclusions $\mtcDcp\deltaBra{n}\hookrightarrow\mtcDcpI\deltaBra{n}$
  and $\mtpESdp\deltaBra{n}\hookrightarrow\mtpESdpI\deltaBra{n}$, which induce the
  desired simplicial maps $\mtsDcp\mtyndSpx\deltaBra{n}\hookrightarrow\mtsDcpI\mtyndSpx\deltaBra{n}$
  and $\mtsESdp\mtyndSpx\deltaBra{n}\hookrightarrow\mtsESdpI\mtyndSpx\deltaBra{n}$.
  This leaves us with the construction of
  $\mtsESd\mtyndSpx\deltaBra{n}\hookrightarrow\mtsESdI\mtyndSpx\deltaBra{n}$.
  Unwinding the definition:
  \begin{align*}
    \mtsESd\mtyndSpx\deltaBra{n}_k
    &\cong
    \HomOf[\mtcPoset]{\deltaBra{k}\star \deltaBra{k}}{\deltaBra{n}};\\
    \mtsESdI\mtyndSpx\deltaBra{n}_k
    &= \coprod_{\substack{I\sqcup J = \deltaBra{k}\\I < J}}
    \HomOf[\mtcPoset]{I\star I\star J}{\deltaBra{n}}.
  \end{align*}
  Therefore, the set $\mtsESd\mtyndSpx\deltaBra{n}_k$ embeds isomorphically to
  the component of $\mtsESdI\mtyndSpx\deltaBra{n}_k$ indexed by the partition
  $I=\deltaBra{k}$ and $J=\emptyset$. This gives the desired simplicial map.

  The naturality of these inclusions is clear from the definitions.
  The injectivity of these maps may be shown by using \cref{cor:unaug-cospx-kan-ext-nat-inj}.
\end{proof}

\begin{definition}\label{def:id-times-1-incl-to-dcpi-esdi}
  Let $X$ be a simplicial set. The injective simplicial maps
  $X \hookrightarrow \mtsDcpI X$ and $X \hookrightarrow \mtsESdI X$, natural
  in $X$ and constructed below, will be considered as the canonical ones between the
  involved simplicial sets.
\end{definition}
\begin{proof}[Construction]
  For the fifth time, we may confine our attention to the case where $X$ is a representable
  simplicial set $\mtyndSpx\deltaBra{n}$. The first of the desired maps:
  \[
    \nerve\deltaBra{n} = \mtyndSpx\deltaBra{n}
    \hookrightarrow\mtsDcpI\mtyndSpx\deltaBra{n} = N(\mtcDcpI\deltaBra{n})
  \]
  may be obtained by applying the nerve functor to the inclusion of the full subcategory:
  \[
    \deltaBra{n} = \set{1} \times \deltaBra{n} \hookrightarrow 
    \mtcDcpI\deltaBra{n};\,\,
    x \mapsto (1,x).
  \]
  Note that this is natural in $\deltaBra{n}$.

  For the construction of the other map $\mtyndSpx\deltaBra{n}\hookrightarrow\mtsESdI\mtyndSpx\deltaBra{n}$,
  remember:
  \begin{align*}
    \mtyndSpx\deltaBra{n}_k
    &= \HomOf[\mtcPoset]{\deltaBra{k}}{\deltaBra{n}};\\
    \mtsESdI\mtyndSpx\deltaBra{n}_k = \mtsESdIc\deltaBra{n}_k &= \coprod_{\substack{I\sqcup J = \deltaBra{k}\\I < J}}
    \HomOf[\mtcPoset]{I\star I\star J}{\deltaBra{n}}.
  \end{align*}
  Therefore,
  \[
    \HomOf[\mtcPoset]{\deltaBra{k}}{\deltaBra{n}}
    \cong \HomOf[\mtcPoset]{\emptyset\star\emptyset\star\deltaBra{k}}{\deltaBra{n}}
    \hookrightarrow \coprod_{\substack{I\sqcup J = \deltaBra{k}\\I < J}} \HomOf[\mtcPoset]{I\star I\star J}{\deltaBra{n}}
  \]
  gives a map between the sets of $k$-simplices.
  Since this is natural in $\deltaBra{k} \in \dualCat{\mtcSimpx}$ and $\deltaBra{n} \in \mtcSimpx$,
  we obtain what we want.

  To show the injectivity of these maps, it suffices to remember \cref{cor:unaug-cospx-kan-ext-nat-inj}.
\end{proof}

Finally, we remember the diagram that we presented at the beginning of this subsection:

\begin{lemma}\label{lem:dcp-esd-esdp-i-comm}
  The diagram \eqref{eq:dcp-esd-esdp-i-comm} at the beginning of this subsection,
  which we shall reproduce here for the reader's convenience, commutes:
  \[
    \begin{tikzcd}
      \mtsDcp X\times\set{0} 
      \arrow[r, two heads, "\labelcref{def:dcp-esd-connecting-surj}"]
      \arrow[d, hook, "\labelcref{def:dcp-esd-esdp-incl-to-i}"']
      & \mtsESd X\times\set{0} \arrow[r, hook, "\labelcref{def:esd-esdp-connecting-inj}"]
      \arrow[d, hook, "\labelcref{def:dcp-esd-esdp-incl-to-i}"']
      & \mtsESdp X\times\set{0} \arrow[d, hook, "\labelcref{def:dcp-esd-esdp-incl-to-i}"']
      & X\times\set{0} \arrow[l, hook', "\labelcref{def:id-esdp-connecting-inj}"'] \arrow[d, equal] \\
      \mtsDcpI X \arrow[r, two heads, "\labelcref{def:dcp-esd-connecting-surj}"]
      & \mtsESdI X \arrow[r, hook, "\labelcref{def:esd-esdp-connecting-inj}"]
      & \mtsESdpI X
      & X \times \set{0} \arrow[d, hook]\\
      X\times\set{1} \arrow[u, hook, "\labelcref{def:id-times-1-incl-to-dcpi-esdi}"] \arrow[r, equal]
      & X\times\set{1} \arrow[u, hook, "\labelcref{def:id-times-1-incl-to-dcpi-esdi}"] \arrow[r, equal]
      & X\times \set{1} \arrow[r, hook]
      & X \times \mtyndSpx\deltaBra{1} \arrow[lu, hook', "\labelcref{def:id-esdp-connecting-inj}"']
    \end{tikzcd}
  \]
  Remember that the numbers stand for the definition where the natural transformation is constructed.
\end{lemma}
\begin{proof}
  By construction.
\end{proof}

\subsection{Properties of the transformations 1: two universal localizations}
\label{subsec:down-last-infty-loc-shape-uloc}

In \cref{def:dcp-esd-connecting-surj}, we have constructed the canonical natural transformations
${\mtsDcp} \Rightarrow {\mtsESd}$ and ${\mtsDcpI} \Rightarrow {\mtsESdI}$.
The purpose of this section is to show that the following three classes of natural
maps are universal localizations in the sense of
\cite[\href{https://kerodon.net/tag/02M0}{Definition 02M0}]{kerodon}:

\begin{itemize}
  \item the canonical map $\mtsDcp X \twoheadrightarrow \mtsESd X$ for each simplicial set $X$
  (\cref{prop:dcp-esd-connecting-surj-univ-loc});
  \item the canonical map $\mtsDcpI X \twoheadrightarrow \mtsESdI X$ for each simplicial set $X$
  (\cref{prop:dcpi-esdpi-connecting-surj-univ-loc});
  \item the map $\mtsDcpI X \cup_{\mtsDcp X} \mtsDcp Y \twoheadrightarrow \mtsESdI X \cup_{\mtsESd X} \mtsESd Y$
  for any simplicial map $X\to Y$, induced by the canonical transformations 
  (\cref{cor:dcp-esd-conn-surj-pushout-univ-loc}).
\end{itemize}

Using \href{https://kerodon.net/tag/02M9}{Proposition 02M9} and \href{https://kerodon.net/tag/02MA}{Proposition 02MA} 
from \cite{kerodon},  we can show that these three types of maps are universal localizations 
by decomposing both the domain and codomain into smaller parts, verifying universal localization on each part, 
and then assembling the results via colimits.  
\cref{lem:last-functor-overcat-n-univ-loc,lem:dcpi-esdpi-connecting-surj-delta-univ-loc-piece}
presented below provide the fundamental building blocks of these proofs.

\begin{notation}\label{not:last-functor-overcat-n}
  Let $\deltaBra{n}\in\Ob\mtcSimpxAug$.
  For the purpose of this subsection, we use the following symbol for a functor:
  \[
    \lambda_{\deltaBra{n}}\colon \overcat{\mtcSimpx}{\deltaBra{n}} \to \deltaBra{n};\;
    (\alpha\colon\deltaBra{m}\to\deltaBra{n})\mapsto\alpha(m).
  \]
  Here, note that $\overcat{\mtcSimpx}{\deltaBra{n}}$ may be extended to the case $n=-1$
  by regarding it as a full subcategory of $\overcat{\mtcSimpxAug}{\deltaBra{n}}$, i.e.,
  $\overcat{\mtcSimpx}{\deltaBra{-1}} = \emptyset$.
  Note that this is natural in $\deltaBra{n}\in\mtcSimpx$. We shall also write
  \[
    \lambda_{\deltaBra{n}}\coloneqq\nerve(\lambda_{\deltaBra{n}})\colon
    \nerve(\overcat{\mtcSimpx}{\deltaBra{n}})\to\mtyndSpx\deltaBra{n}.
  \]
  Remember our notation: $\mtyndSpx\deltaBra{-1} = \nerve\deltaBra{-1} = \emptyset$.
\end{notation}

\begin{lemma}\label{lem:last-functor-overcat-n-univ-loc}
  For any $\deltaBra{n}\in\mtcSimpxAug$, the functor
  $\lambda_{\deltaBra{n}}\colon\nerve(\overcat{\mtcSimpx}{\deltaBra{n}})\to\mtyndSpx\deltaBra{n}$
  in \cref{not:last-functor-overcat-n} is a universal localization in the sense of
  \cite[\href{https://kerodon.net/tag/02M0}{Definition 02M0}]{kerodon}.
\end{lemma}
\begin{proof}
  According to \cite[\href{https://kerodon.net/tag/04JT}{Proposition 04JT}]{kerodon}, to show our claim,
  it suffices to construct a section $u\colon\mtyndSpx\deltaBra{n}\to\nerve(\overcat{\mtcSimpx}{\deltaBra{n}})$
  to $\lambda_{\deltaBra{n}}$, and to check that the composition
  $u\compos\lambda_{\deltaBra{n}}$ and the identity $\idmor[\mtyndSpx\deltaBra{n}]$
  belong to the same connected component of the mapping space
  $\operatorname{Fun}_{/{\mtyndSpx\deltaBra{n}}}(\lambda_{\deltaBra{n}},\lambda_{\deltaBra{n}})$
  over $\mtyndSpx\deltaBra{n}$ (see \cite[\href{https://kerodon.net/tag/01AB}{Construction 01AB}]{kerodon} for the notation).
  In terms of 1-categories, it is enough to construct a functor $u\colon\deltaBra{n}\to\overcat{\mtcSimpx}{\deltaBra{n}}$
  strictly satisfying $\lambda_{\deltaBra{n}}\compos u = \idmor[\deltaBra{n}]$, and a natural transformation
  $\theta\colon \idmor[\overcat{\mtcSimpx}{\deltaBra{n}}]\Rightarrow u\compos\lambda_{\deltaBra{n}}$ satisfying
  $\lambda_{\deltaBra{n}}\whiskl\theta = \idmor[\lambda_{\deltaBra{n}}]$.

  We define $u$ as follows. For each object $x\in\deltaBra{n}$, we set
  $u(x) \coloneqq \iota^n_{\deltaBra{x}}$. 
  Remember, for any non-empty subset $S=\set{x_0<\dotsb<x_k}\subseteq\deltaBra{n}$,
  the injective order-preserving map $\iota^n_S\colon\deltaBra{k}\to\deltaBra{n}$ has been defined by
  $\iota^n_S(i) = x_i$, and the morphism $\iota^n_{\deltaBra{x}}\colon \deltaBra{x} \to \deltaBra{n}$
  in $\mtcSimpx$ gives an object of $\overcat{\mtcSimpx}{\deltaBra{n}}$. If $x \le y$ in $\deltaBra{n}$,
  the morphism 
  \[u(\uniqmor_{x,y})\colon u(x) = \iota^n_{\deltaBra{x}} \to \iota^n_{\deltaBra{y}} = u(y)\]
  is given by $u(\uniqmor_{x,y}) \coloneqq \iota^y_{\deltaBra{x}}\colon\deltaBra{x}\to\deltaBra{y}$,
  which evidently satisfies the condition for a morphism in $\mtcSimpx$ to be a morphism in $\overcat{\mtcSimpx}{\deltaBra{n}}$.
  This defines a functor $u\colon\deltaBra{n}\to\overcat{\mtcSimpx}{\deltaBra{n}}$, and we can easily check:
  \[ 
    (\lambda_{\deltaBra{n}}\compos u)(x) = \iota^n_{\deltaBra{x}}(x) = x = \idmor[\deltaBra{n}](x).
  \]

  We proceed to construct the natural transformation $\theta$. For each object in $\overcat{\mtcSimpx}{\deltaBra{n}}$,
  i.e, a morphism $\alpha\colon\deltaBra{m}\to\deltaBra{n}$ in $\mtcSimpx$, we need to define a morphism
  \[
    \theta_{\alpha}\colon\alpha=\idmor[\overcat{\mtcSimpx}{\deltaBra{n}}](\alpha)
    \to u(\lambda_{\deltaBra{n}}(\alpha)) = \iota^n_{\deltaBra{\alpha(m)}}
  \]
  in $\overcat{\mtcSimpx}{\deltaBra{n}}$. This is given by the morphism
  $\deltaBra{m}\to\deltaBra{\alpha(m)};\; x\mapsto\alpha(x)$ in $\mtcSimpx$, which is indeed a morphism
  $\alpha\to\iota^n_{\deltaBra{\alpha(m)}}$ in $\overcat{\mtcSimpx}{\deltaBra{n}}$. The naturality of
  the family $\theta_\alpha$ in $\alpha$ may be simply checked. Indeed, let the following commutative
  triangle in $\mtcSimpx$ present an arbitrary morphism $\beta$ in $\overcat{\mtcSimpx}{\deltaBra{n}}$:
  \[
  \begin{tikzcd}
    \deltaBra{m} \ar[rr, "\beta"] \ar[dr, "\alpha"']
    && \deltaBra{m'} \ar[dl, "\alpha'"]\\
    & \deltaBra{n} &
  \end{tikzcd}
  \]
  Then, for any $x\in\deltaBra{m}$, we compute:
  \[
    (\theta_{\alpha'}\compos\idmor[\deltaBra{m}](\beta))(x)
    = \alpha'(\beta(x))
    = \alpha(x)
    = \iota^{\alpha'(m')}_{\deltaBra{\alpha(m)}}(\alpha(x))
    = ((u\compos\lambda_{\deltaBra{n}})(\beta) \compos \theta_{\alpha})(x),
  \]
  which is the desired naturality. Since $\deltaBra{n}$ is a poset,
  the equation $\lambda_{\deltaBra{n}}\whiskl\theta = \idmor[\lambda_{\deltaBra{n}}]$ trivially holds.
  This completes the proof.
\end{proof}

\begin{corollary}\label{lem:dcpi-esdpi-connecting-surj-delta-univ-loc-piece}
  Let $\deltaBra{m},\deltaBra{n}\in\mtcSimpxAug$ be objects of
  the augmented simplex category.
  Using the notation from \cref{not:last-functor-overcat-n}, consider the following simplicial map:
  \[
    \idmor[\mtyndSpx\deltaBra{m}]\times\lambda_{\deltaBra{n}}
    \colon \mtyndSpx\deltaBra{m}\times\nerve(\overcat{\mtcSimpx}{\deltaBra{n}})
    \to \mtyndSpx\deltaBra{m}\times\mtyndSpx\deltaBra{n}.
  \]
  Then this map is a universal localization in the sense of
  \cite[\href{https://kerodon.net/tag/02M0}{Definition 02M0}]{kerodon}.
\end{corollary}
\begin{proof}
  Since the product preserves universal localizations (\cref{lem:univ-loc-closure-product}),
  this follows from \cref{lem:last-functor-overcat-n-univ-loc}.
\end{proof}

Building on the minimal components of the domain and codomain of our maps, treated in
\cref{lem:last-functor-overcat-n-univ-loc,lem:dcpi-esdpi-connecting-surj-delta-univ-loc-piece},
we now proceed to slightly larger parts:

\begin{lemma}\label{lem:dcp-esd-connecting-surj-delta-univ-loc}
  Let $\deltaBra{n}\in\mtcSimpxAug$ be an object of the augmented simplex category.
  Then the canonical map $\mtsDcp\mtyndSpx\deltaBra{n}\to\mtsESd\mtyndSpx\deltaBra{n}$
  from \cref{def:dcp-esd-connecting-surj} is a universal localization in the sense of
  \cite[\href{https://kerodon.net/tag/02M0}{Definition 02M0}]{kerodon}.
\end{lemma}
\begin{proof}
  The case $n=-1$ is trivial, so we may assume $n\ge 0$.

  For each $-1\le l\le n$, let $C^n_{\le l}$ and $C^n_l$ denote the full subcategories of
  $\mtcDcp\deltaBra{n}$ determined by:
  \begin{align*}
    \Ob(C^n_{\le l}) &\coloneqq
     \set{(x,\alpha)\in\Ob(\mtcDcp\deltaBra{n})}[x\le l];\\
    \Ob(C^n_{l}) &\coloneqq
     \set{(x,\alpha)\in\Ob(\mtcDcp\deltaBra{n})}[x\le l\le\alpha(0)].
  \end{align*}
  We set $F_{\le l}\coloneqq \nerve(C^n_{\le l}),\, F_l\coloneqq \nerve(C^n_l) 
  \subseteq \mtsDcp\mtyndSpx\deltaBra{n}$. Then we have the following filtration of simplicial sets:
  \[
    \emptyset = F_{\le -1} \subseteq F_{\le 0} \subseteq F_{\le 1}
    \subseteq \dotsb \subseteq F_{\le n} = \mtsDcp\mtyndSpx\deltaBra{n}.
  \]
  We shall also consider the following simplicial subsets of $\mtsESd\mtyndSpx\deltaBra{n}$:
  for each $-1\le l \le n$:
  \begin{align*}
    (G_{\le l})_k &\coloneqq \set{f\colon \deltaBra{2k+1} = 
      \deltaBra{k}\star\deltaBra{k}\to\deltaBra{n}}[f(k) \le l];\\
    (G_{l})_k &\coloneqq \set{f\colon \deltaBra{2k+1} =
      \deltaBra{k}\star\deltaBra{k}\to\deltaBra{n}}[f(k) \le l \le f(k+1)].
  \end{align*}
  Then we again obtain the following filtration of simplicial sets:
  \[
    \emptyset = G_{\le -1} \subseteq G_{\le 0} \subseteq G_{\le 1}
    \subseteq \dotsb \subseteq G_{\le n} = \mtsESd\mtyndSpx\deltaBra{n}.
  \]
  Notice that the canonical map $\mtsDcp\mtyndSpx\deltaBra{n}\to\mtsESd\mtyndSpx\deltaBra{n}$
  restricts to $F_{\le l}\to G_{\le l}$ and $F_{l}\to G_{l}$ for each $-1\le l \le m$.

  By induction, in the increasing order of $l$, we shall show that
  the map $F_{\le l}\to G_{\le l}$ is a universal localization, which will
  imply our desired lemma. The base case $l=-1$ is trivial. Let $0\le l\le m$ and
  suppose that $F_{\le {l-1}}\to G_{\le {l-1}}$ is a universal localization;
  we need to show that $F_{\le l}\to G_{\le l}$ is a universal localization.
  Consider the following commutative cube, consisting of inclusions of simplicial subsets
  ($\hookrightarrow$) and the restrictions of the canonical map $\mtsDcp\mtyndSpx\deltaBra{n}\to\mtsESd\mtyndSpx\deltaBra{n}$
  ($\rightarrow$):
  \[
    \begin{tikzcd}[row sep=small]
      & G_{l-1} \cap G_{l} \arrow[dd, hook] \arrow[rr, hook]
      && G_l \arrow[dd, hook] \\
      F_{l-1}\cap F_l \arrow[ru] \arrow[dd, hook] \arrow[rr, hook, crossing over]
      && F_l \arrow[ru]
      &\\
      & G_{\le l-1} \arrow[rr, hook]
      && G_{\le l} \\
      F_{\le l-1} \arrow[ru] \arrow[rr, hook]
      && F_{\le l} \arrow[from=uu, hook, crossing over] \arrow[ru]
      &
    \end{tikzcd}
  \]
  In this cube, the front and the back faces are pushout squares by construction,
  and the hooked arrows are injective. Therefore,
  by \cite[\href{https://kerodon.net/tag/02MA}{Proposition 02MA}]{kerodon},
  in order to show that $F_{\le l}\to G_{\le l}$ is a universal localization, it suffices
  to check that the three maps $F_{l-1}\cap F_l\to G_{l-1}\cap G_l$,
  $F_l \to G_l$, and $F_{\le l-1}\to G_{\le l-1}$ are universal localizations.
  Of those three,
  the third map $F_{\le l-1}\to G_{\le l-1}$ is a universal localization by the induction hypothesis.
  Now, using the notation from \cref{not:last-functor-overcat-n}, we observe that
  the first map $F_{l-1}\cap F_l\to G_{l-1}\cap G_l$ is isomorphic to: 
  \[
    \idmor[\mtyndSpx\deltaBra{l-1}] \times \lambda_{\deltaBra{n-l}}
    \colon\mtyndSpx\deltaBra{l-1}\times\nerve(\overcat{\mtcSimpx}{\deltaBra{n-l}})
    \to\mtyndSpx\deltaBra{l-1}\times\mtyndSpx\deltaBra{n-l},
  \]
  and that the second map $F_l\to G_l$ is isomorphic to the map:
  \[
    \idmor[\mtyndSpx\deltaBra{l}] \times \lambda_{\deltaBra{n-l}}
    \colon\mtyndSpx\deltaBra{l}\times\nerve(\overcat{\mtcSimpx}{\deltaBra{n-l}})
    \to\mtyndSpx\deltaBra{l}\times\mtyndSpx\deltaBra{n-l}.
  \]
  These two maps are universal localizations by
  \cref{lem:dcpi-esdpi-connecting-surj-delta-univ-loc-piece}.
  Thus we have proved the sufficient condition for the map
  $F_{\le l}\to G_{\le l}$ to be a universal localization.
  This completes the induction, and we have obtained the desired claim.
\end{proof}

\begin{corollary} \label{lem:dcpi-esdpi-connecting-surj-delta-univ-loc-piece2}
  Let $\deltaBra{m},\deltaBra{n}\in\mtcSimpx$ be objects. Then the following join of
  the canonical simplicial map from \cref{def:dcp-esd-connecting-surj} and the identity map
  is a universal localization in the sense of
  \cite[\href{https://kerodon.net/tag/02M0}{Definition 02M0}]{kerodon}:
  \[
    \mtsDcp\mtyndSpx\deltaBra{m}\star\mtyndSpx\deltaBra{n}
    \to\mtsESd\mtyndSpx\deltaBra{m}\star\mtyndSpx\deltaBra{n}.
  \]
\end{corollary}
\begin{proof}
  By using the fact that the join ${}\star X$ of a simplicial set $X$ is preserves universal localizations
  (\cref{cor:univ-loc-closure-join}), this follows from \cref{lem:dcp-esd-connecting-surj-delta-univ-loc}.
\end{proof}

Now, we are ready to prove one of the main results of this section:

\begin{proposition}\label{prop:dcp-esd-connecting-surj-univ-loc}
  Let $X\in\mtcSSet$ be a simplicial set. Then the canonical simplicial map
  $\mtsDcp X\to\mtsESd X$ in \cref{def:dcp-esd-connecting-surj}
  is a universal localization in the sense of
  \cite[\href{https://kerodon.net/tag/02M0}{Definition 02M0}]{kerodon}.
\end{proposition}
\begin{proof}
  As \cite[\href{https://kerodon.net/tag/02M9}{Proposition 02M9}]{kerodon}
  states, the class of universal localizations is closed under filtered colimits in
  $\functCat{\deltaBra{1}}{\mtcSSet}$. Therefore, we may assume that $X$ is a
  finite simplicial set. We shall show the claim by induction on the dimension
  and the number of the highest-dimensional non-degenerate simplices of $X$.
  If $X$ is empty, the claim is trivial. If $X$ is non-empty, let
  $\sigma\colon \mtyndSpx\deltaBra{n}\to X$ be one of the highest-dimensional
  non-degenerate simplices of $X$. There is the following pushout square,
  with $X'$ a subcomplex of $X$ having fewer dimension-$n$ non-degenerate simplices:
  \[
    \begin{tikzcd}
      \mtBdrySpx\deltaBra{n} \arrow[r, hook] \arrow[d] &
      \mtyndSpx\deltaBra{n} \arrow[d] \\
      X' \arrow[r, hook] & X
    \end{tikzcd}
  \]
  By sending this square under $\mtsDcp$ and $\mtsESd$, we obtain the following
  commutative cube:
  \[
    \begin{tikzcd}[row sep=small, column sep=small]
      & \mtsESd \mtBdrySpx\deltaBra{n} \arrow[dd] \arrow[rr, hook]
      && \mtsESd \mtyndSpx\deltaBra{n} \arrow[dd] \\
      \mtsDcp\mtBdrySpx\deltaBra{n} \arrow[ru] \arrow[dd] \arrow[rr, hook, crossing over]
      && \mtsDcp \mtyndSpx\deltaBra{n} \arrow[ru] \arrow[dd]
      &\\
      & \mtsESd X' \arrow[rr, hook]
      && \mtsESd X \\
      \mtsDcp X' \arrow[ru] \arrow[rr, hook]
      && \mtsDcp X \arrow[from=uu, crossing over] \arrow[ru]
      &
    \end{tikzcd}
  \]
  In this cube, the front and the back faces are pushout squares by construction,
  and the hooked arrows are injective. The maps
  $\mtsDcp \mtBdrySpx\deltaBra{n}\to\mtsESd \mtBdrySpx\deltaBra{n}$
  and $\mtsDcp X'\to\mtsESd X'$ are universal localizations by the induction hypothesis.
  The map $\mtsDcp \mtyndSpx\deltaBra{n}\to\mtsESd \mtyndSpx\deltaBra{n}$ is a universal
  localization by \cref{lem:dcp-esd-connecting-surj-delta-univ-loc}.
  Therefore, by \cite[\href{https://kerodon.net/tag/02MA}{Proposition 02MA}]{kerodon},
  the map $\mtsDcp X\to\mtsESd X$ is a universal localization.
  This completes the induction, and we have obtained the desired claim.
\end{proof}

The following lemma is the largest component of the remaining main universal localization result:

\begin{lemma}\label{lem:dcpi-esdpi-connecting-surj-delta-univ-loc}
  Let $\deltaBra{n}\in\mtcSimpx$ be a simplex. Then the canonical simplicial
  map $\mtsDcpI\mtyndSpx\deltaBra{n}\to\mtsESdI\mtyndSpx\deltaBra{n}$ is a
  universal localization in the sense of 
  \cite[\href{https://kerodon.net/tag/02M0}{Definition 02M0}]{kerodon}.
\end{lemma}
\begin{proof}
  For each $-1\le l\le n$, let $C^n_{\le l}$ and $C^n_l$ denote the full subcategories of
  $\mtcDcpI\deltaBra{n}$ determined by:
  \begin{align*}
    \Ob(C^n_{\le l}) &\coloneqq
    (\set{0}\times\set{(x,\alpha)\in\Ob(\mtcDcp\deltaBra{n})}[\max\alpha \le l])
    \cup (\set{1}\times\deltaBra{n});\\
    \Ob(C^n_{l}) &\coloneqq
    (\set{0}\times\set{(x,\alpha)\in\Ob(\mtcDcp\deltaBra{n})}[\max\alpha \le l])
    \cup \set{(1,x)}[l \le x \le n].
  \end{align*}
  We set $F_{\le l}\coloneqq \nerve(C^n_{\le l}),\, F_l\coloneqq \nerve(C^n_l)\subseteq \mtsDcpI\mtyndSpx\deltaBra{n}$.
  Then we have the following filtration of simplicial sets:
  \[
    \set{1}\times\mtyndSpx\deltaBra{n} \cong F_{\le -1} \subseteq F_{\le 0} \subseteq F_{\le 1}
    \subseteq \dotsb \subseteq F_{\le n} = \mtsDcpI\mtyndSpx\deltaBra{n}.
  \]

  We also need a filtration of the codomain. For notational ease, regard:
  \[
    \mtsESdI\mtyndSpx\deltaBra{n}_k = \mtsESdIc\deltaBra{n}_k
    = \coprod_{\substack{I\sqcup J = \deltaBra{k}\\I < J}}
    \HomOf[\mtcPoset]{I\star I\star J}{\deltaBra{n}},
  \]
  as the set of triplets of the form $(I,J,f)$, where $I\sqcup J = \deltaBra{k}$, $I < J$ is a partition,
  and $f\colon I\star I\star J\to\deltaBra{n}$ is an order-preserving function.
  For each $-1\le l \le n$, we define the following simplicial subsets of $\mtsESdI\mtyndSpx\deltaBra{n}$:
  \begin{align*}
    (G_{\le l})_k &\coloneqq \set{(I,J,f)}[\max \mleft.f\mright|_{I\star I} \le l];\\
    (G_{l})_k &\coloneqq \set{(I,J,f)}[\max \mleft.f\mright|_{I\star I} \le l \le \min \mleft.f\mright|_{J}].
  \end{align*}
  Here, we set $\max\mleft.f\mright|_{I\star I} = -1$ if $I=\emptyset$, and $\min\mleft.f\mright|_{J} = n$
  if $J=\emptyset$.
  Then we obtain the following filtration of simplicial sets:
  \[
    \mtyndSpx\deltaBra{n} \cong G_{\le -1} \subseteq G_{\le 0} \subseteq G_{\le 1}
    \subseteq \dotsb \subseteq G_{\le n} = \mtsESdI\mtyndSpx\deltaBra{n}.
  \]
  Notice that the canonical map $\mtsDcpI\mtyndSpx\deltaBra{n}\to\mtsESdI\mtyndSpx\deltaBra{n}$
  restricts to $F_{\le l}\to G_{\le l}$ and $F_{l}\to G_{l}$ for each $-1\le l \le n$.

  Now, we repeat an inductive argument similar to that in the proof of \cref{lem:dcp-esd-connecting-surj-delta-univ-loc}.
  We shall show that $F_{\le l}\to G_{\le l}$ is a universal localization for each $-1\le l \le n$,
  which is enough for the estabilishment of the lemma.
  For the base case $l=-1$, the map $F_{\le -1}\to G_{\le -1}$ is the identity map, which is a universal localization.

  For an inductive case, let $0\le l\le m$ and suppose that $F_{\le {l-1}}\to G_{\le {l-1}}$ is a universal localization;
  we need to show that $F_{\le l}\to G_{\le l}$ is a universal localization.
  We again wish to use \cite[\href{https://kerodon.net/tag/02MA}{Proposition 02MA}]{kerodon}. Consider the
  following commutative cube, where $\hookrightarrow$ denotes inclusions of simplicial subsets,
  and $\rightarrow$ denotes restrictions of the canonical map $\mtsDcpI\mtyndSpx\deltaBra{n}\to\mtsESdI\mtyndSpx\deltaBra{n}$:
  \[
    \begin{tikzcd}[row sep=small]
      & G_{l-1} \cap G_{l} \arrow[dd, hook] \arrow[rr, hook]
      && G_l \arrow[dd, hook] \\
      F_{l-1}\cap F_l \arrow[ru] \arrow[dd, hook] \arrow[rr, hook, crossing over]
      && F_l \arrow[ru]
      &\\
      & G_{\le l-1} \arrow[rr, hook]
      && G_{\le l} \\
      F_{\le l-1} \arrow[ru] \arrow[rr, hook]
      && F_{\le l} \arrow[from=uu, hook, crossing over] \arrow[ru]
      &
    \end{tikzcd}
  \]
  Since the front and the back faces are pushout squares by construction, and the hooked arrows are injective,
  \cite[\href{https://kerodon.net/tag/02MA}{Proposition 02MA}]{kerodon} applies: to show that $F_{\le l}\to G_{\le l}$
  is a universal localization, it suffices to check that the three maps $F_{l-1}\cap F_l\to G_{l-1}\cap G_l$,
  $F_l \to G_l$, and $F_{\le l-1}\to G_{\le l-1}$ are universal localizations.

  The first map $F_{l-1}\cap F_l\to G_{l-1}\cap G_l$ is isomorphic to:
  \[ \mtsDcp\mtyndSpx\deltaBra{l-1}\star\mtyndSpx\deltaBra{n-l}
  \to\mtsESd\mtyndSpx\deltaBra{l-1}\star\mtyndSpx\deltaBra{n-l}.
  \]
  The second map $F_l\to G_l$ is isomorphic to:
  \[ \mtsDcp\mtyndSpx\deltaBra{l}\star\mtyndSpx\deltaBra{n-l}
     \to\mtsESd\mtyndSpx\deltaBra{l}\star\mtyndSpx\deltaBra{n-l}.
  \]
  Therefore these two maps are universal localizations by \cref{lem:dcpi-esdpi-connecting-surj-delta-univ-loc-piece2}.
  The third map $F_{\le l-1}\to G_{\le l-1}$ is a universal localization by the induction hypothesis; thus
  we have shown that $F_{\le l}\to G_{\le l}$ is a universal localization. This completes the induction,
  and we have obtained the desired lemma.
\end{proof}

Now, it only remains to prove the main results of this subsection:

\begin{proposition}\label{prop:dcpi-esdpi-connecting-surj-univ-loc}
  Let $X\in\mtcSSet$ be a simplicial set. Then the canonical simplicial map
  $\mtsDcpI X\to\mtsESdI X$ in \cref{def:dcp-esd-connecting-surj}
  is a universal localization in the sense of 
  \cite[\href{https://kerodon.net/tag/02M0}{Definition 02M0}]{kerodon}.
\end{proposition}
\begin{proof}
  By virtue of \cref{lem:dcpi-esdpi-connecting-surj-delta-univ-loc},
  the proof goes exactly in the same way as that
  of \cref{prop:dcp-esd-connecting-surj-univ-loc}.
\end{proof}

\begin{corollary}\label{cor:dcp-esd-conn-surj-pushout-univ-loc}
  Let $f\colon X\to Y$ be a simplicial map. Consider the following
  morphism of spans of simplicial sets:
  \[
    \begin{tikzcd}
      \mtsDcp Y \arrow[d, two heads]
      & \mtsDcp X \arrow[d, two heads] \arrow[l, "f"'] \arrow[r, hook]
      & \mtsDcpI X \arrow[d, two heads] \\
      \mtsESd Y
      & \mtsESd X \arrow[l, "f"] \arrow[r, hook]
      & \mtsESdI X
    \end{tikzcd}
  \]
  Then the following induced map of pushouts is a universal localization
  in the sense of \cite[\href{https://kerodon.net/tag/02M0}{Definition 02M0}]{kerodon}:
  \[
    (\mtsDcp Y) \cup_{\mtsDcp X} (\mtsDcpI X) \to (\mtsESd Y) \cup_{\mtsESd X} (\mtsESdI X).
  \]
\end{corollary}
\begin{proof}
  By \cref{prop:dcp-esd-connecting-surj-univ-loc,prop:dcpi-esdpi-connecting-surj-univ-loc},
  the claim follows from \cite[\href{https://kerodon.net/tag/02MA}{Proposition 02MA}]{kerodon}.
\end{proof}

\subsection{Properties of the transformations 2: an inner anodyne map}
\label{subsec:down-last-infty-loc-shape-inner-anod}

The goal of this subsection is the following:

\begin{proposition}\label{prop:esd-esdp-conn-inner-anod}
  Let $X\in\mtcSSet$ be a simplicial set. Then the canonical simplicial map
  $\mtsESd X\to\mtsESdp X$ in \cref{def:esd-esdp-connecting-inj}
  is inner anodyne.
\end{proposition}

For the proof of this proposition, the following lemma is crucial:

\begin{lemma}\label{lem:esd-esdp-conn-inner-anod}
  Let $\deltaBra{n}\in\mtcSimpx$ be a simplex. Consider the following 
  commutative square:
  \[
    \begin{tikzcd}
      \mtsESd\mtBdrySpx\deltaBra{n}
      \arrow[r, hookrightarrow] \arrow[d, hookrightarrow] &
      \mtsESd\mtyndSpx\deltaBra{n} \arrow[d, hookrightarrow] \\
      \mtsESdp\mtBdrySpx\deltaBra{n} \arrow[r, hookrightarrow] &
      \mtsESdp\mtyndSpx\deltaBra{n}
    \end{tikzcd}
  \]
  Here, the vertical maps are from \cref{def:esd-esdp-connecting-inj},
  and the horizontal maps are induced by the canonical inclusions.
  Then the corresponding simplicial map
  \[
    (\mtsESd\mtyndSpx\deltaBra{n})
    \cup_{\mtsESd\mtBdrySpx\deltaBra{n}}
    (\mtsESdp\mtBdrySpx\deltaBra{n})
    \to \mtsESdp\mtyndSpx\deltaBra{n}
  \]
  out of the pushout is inner anodyne.
\end{lemma}

With this lemma, the proof of \cref{prop:esd-esdp-conn-inner-anod} is straightforward.

\begin{proof}[Proof of \cref{prop:esd-esdp-conn-inner-anod} from \cref{lem:esd-esdp-conn-inner-anod}]
  \Cref{lem:esd-esdp-conn-inner-anod} says that the canonical natural transformation
  ${\mtsESd}\compos{\mtyndSpx} \Rightarrow {\mtsESdp}\compos{\mtyndSpx}\colon \mtcSimpx \to \mtcSSet$
  is Reedy inner anodyne.
  Therefore from \cref{lem:cospx-kan-ext-reedy-nat} follows that
  the map in question, obtained by the left Kan extension of this natural transformation
  along the Yoneda embedding, is inner anodyne.
\end{proof}

The rest of this subsection is devoted to the proof of \cref{lem:esd-esdp-conn-inner-anod}.
The proof of the lemma is combinatorial and straightforward: since inner horn inclusion
$\mtsHorn{n}{k}\hookrightarrow\mtyndSpx\deltaBra{n}$ adds two non-degenerate simplices in pair,
we just need to add non-degenerate simplices two-by-two in an appropriate order.
However, making such pairs (called \emph{horn pairs}) gracefully takes a bit of care and a long argument.
To enhance readability, we shall break the proof down into several other lemmas.
The concluding proof of \cref{lem:esd-esdp-conn-inner-anod} is given in
\cpageref{proof:lem:esd-esdp-conn-inner-anod},
after \cref{lem:esd-esdp-conn-inner-anod-horn-pair-order}.
We shall first introduce some notation and terminology used for the rest of this subsection:

\begin{notation}\label{not:esd-esdp-conn-inner-anod}
  Here we list the terminology and notation used in the proof of \cref{lem:esd-esdp-conn-inner-anod}.
  The notation is used only throughout this subsection.
  \begin{itemize}
    \item We fix a simplex $\deltaBra{n}\in\mtcSimpx$.
    \item We shall write $f\colon X\hookrightarrow Y$ for the map in focus. Specifically,
      $f$ is the canonical injective map between:
      \begin{align*}
        X &\coloneqq (\mtsESd\mtyndSpx\deltaBra{n})
        \cup_{\mtsESd\mtBdrySpx\deltaBra{n}}
        (\mtsESdp\mtBdrySpx\deltaBra{n});\\
        Y &\coloneqq \mtsESdp\mtyndSpx\deltaBra{n}.
      \end{align*}
    \item We define an \emph{outsider} to be a simplex of $Y$ that is not in the image of $f$.
      If an outsider is a $k$-simplex, we shall say that it is a $k$-outsider.
    \item We shall denote the set of \emph{non-degenerate} outsiders by $\mathcal{O}$. The set of non-degenerate
      $k$-outsiders is denoted by $\mathcal{O}_k$.
    \item If $\sigma$ is a $k$-simplex of $Y$, it is represented by an order-preserving map
      $\deltaBra{k}\to\functCat{\deltaBra{1}}{\deltaBra{n}}$. We shall denote the adjunct of this map
      by $u_\sigma\colon\deltaBra{k}\times\deltaBra{1}\to\deltaBra{n}$.
    \item A \emph{horn pair} is an ordered pair $(\sigma, \tau)$ of non-degenerate outsiders 
      $\sigma$ and $\tau$ that satisfy the following conditions for some $0<i<k\coloneqq\dim\sigma$:
      \begin{enumerate}
        \item $u_\sigma(i-1,0) = u_\sigma(i,0)$;
        \item $u_\sigma(i-1,1) < u_\sigma(i,1) = u_\sigma(k,0)$;
        \item the simplex $\tau$ is the $i$-th facet of $\sigma$.
      \end{enumerate}
    \item The \emph{horn position} of a horn pair $(\sigma, \tau)$ is the integer $i$ in the definition above.
      Note that the integer $i$ is unique, since we have $u_\sigma(i-1,1) < u_\sigma(i,1) = u_\sigma(k,0)$
      and $u_\sigma$ is order-preserving.
    \item The \emph{anticipated horn position} of a non-degenerate $k$-outsider $\sigma$ is the number of
      $0\le j\le k$ such that $u_\sigma(j,1) < u_\sigma(k,0)$.
    \item We denote the set of horn pairs by $\mathcal{H}$.
    \item If $(\sigma, \tau)$ is a horn pair, we shall say that $\sigma$ is the \emph{(horn) core} of
      the pair, and $\tau$ is the \emph{(horn) periphery} of the pair. We shall also simply say that
      $\sigma$ is a \emph{horn core} to mean that, for some $\tau$, the pair $(\sigma, \tau)$ is a horn pair.
      A \emph{horn periphery} is defined similarly.
  \end{itemize}
\end{notation}

In order to prove our desired \cref{lem:esd-esdp-conn-inner-anod}, it suffices to show that non-degenerate
outliers are divided into mutually disjoint horn pairs, and well-order these horn pairs appropriately.
We shall first see the equivalent conditions for a simplex to be a non-degenerate outsider:

\begin{lemma}\label{lem:esd-esdp-conn-inner-anod-outsider-cond}
  Let $\sigma$ be any $k$-simplex in $Y$.
  Then we have the following:
  \begin{enumerate}
    \item The simplex $\sigma$ is non-degenerate if and only if, 
      for each $0\le i < k$, we have $(u_\sigma(i,0), u_\sigma(i,1)) < (u_\sigma(i+1,0), u_\sigma(i+1,1))$
      with respect to the product order.
      \label{item:lem-esd-esdp-conn-inner-anod-outsider-cond:nd}
    \item The simplex $\sigma$ does not belong to the injective image of $\mtsESdp\mtBdrySpx\deltaBra{n}$
      if and only if $u_\sigma$ is surjective.
    \item The simplex $\sigma$ does not belong to the injective image of $\mtsESd\mtyndSpx\deltaBra{n}$
      if and only if $u_\sigma(k,0) > u_\sigma(0,1)$.
    \item The simplex $\sigma$ is an outsider if and only if
      $u_\sigma(k,0) > u_\sigma(0,1)$ and $u_\sigma$ is surjective.
      \label{item:lem-esd-esdp-conn-inner-anod-outsider-cond:out}
  \end{enumerate}
\end{lemma}
\begin{proof}
  By direct calculation.
\end{proof}

A number of lemmas below are needed to show that non-degenerate outsiders are divided into horn pairs:

\begin{lemma}\label{lem:esd-esdp-conn-inner-anod-horn-pos}
  Let $\sigma\in\mathcal{O}_k$ be any non-degenerate $k$-outsider.
  Let $0 \le i \le k+1$ be the anticipated horn position of $\sigma$.
  Then we have the following:
  \begin{enumerate}
    \item It holds that $0 < i \le k$.
      \label{item:lem-esd-esdp-conn-inner-anod-horn-pos:range}
    \item If $\sigma$ belongs to a horn pair $p\in\mathcal{H}$, either as its core or its periphery,
      then the horn position of $p$ is $i$.
      \label{item:lem-esd-esdp-conn-inner-anod-horn-pos:horn-pos}
  \end{enumerate}
\end{lemma}
\begin{proof}
  \eqref{item:lem-esd-esdp-conn-inner-anod-horn-pos:range}:
  Since $u_\sigma$ is order-preserving, we have $i \le k$. The other inequality $i > 0$ follows from
  the assumption that $\sigma$ is an outsider;
  see \cref{lem:esd-esdp-conn-inner-anod-outsider-cond}~\eqref{item:lem-esd-esdp-conn-inner-anod-outsider-cond:out}.

  \eqref{item:lem-esd-esdp-conn-inner-anod-horn-pos:horn-pos}:
  If $\sigma$ is the core of $p$, the claim is immediate from the definition of a horn pair
  and the order preservation of $u_\sigma$.
  Assume that $p$ is of the form $(\tau, \sigma)$, and let us write $i$ for the horn position of $p$.
  Since $\sigma$ is the $i$-th face of $\tau$, we have, for each $0\le j< i$:
  \[ 
    u_\sigma(j,1) = u_\tau(j,1) \le u_\tau(i-1,1) < u_\tau(i,1) = u_\tau(k+1,0) = u_\sigma(k,0).
  \]
  Here we used $i<k+1$. We also have, for each $i \le j \le k$:
  \[
    u_\sigma(j,1) = u_\tau(j+1,1) \ge u_\tau(i,1) = u_\tau(k+1,0) = u_\sigma(k,0).
  \]
  The claim should be clear from these inequalities.
\end{proof}

\begin{lemma}\label{lem:esd-esdp-conn-inner-anod-core-not-periphery}
  A horn core is not a horn periphery.
\end{lemma}
\begin{proof}
  For contradition, suppose that there were to be two horn pairs $(\sigma_0, \sigma_1)$ and $(\sigma_1, \sigma_2)$.
  Let $0<i<k$ be the horn position of the second pair, where $k=\dim\sigma_1$. By applying
  \cref{lem:esd-esdp-conn-inner-anod-horn-pos}~\eqref{item:lem-esd-esdp-conn-inner-anod-horn-pos:horn-pos}
  to $\sigma_1$, we see that $i$ is also the horn position of
  the first pair $(\sigma_0, \sigma_1)$. We may now deduce:
  \begin{gather*}
    u_{\sigma_0}(i,0) = u_{\sigma_0}(i-1,0) = u_{\sigma_1}(i-1,0)= u_{\sigma_1}(i,0) = u_{\sigma_0}(i+1,0);\\
    u_{\sigma_0}(i,1) = u_{\sigma_0}(k+1,0) = u_{\sigma_1}(k,0) = u_{\sigma_1}(i,1) = u_{\sigma_0}(i+1,1).
  \end{gather*}
  These equations and \cref{lem:esd-esdp-conn-inner-anod-outsider-cond}~\eqref{item:lem-esd-esdp-conn-inner-anod-outsider-cond:nd}
  imply that $\sigma_0$ is a degenerate simplex, which contradicts the definition of a horn pair.
\end{proof}

\begin{lemma}\label{lem:esd-esdp-conn-inner-anod-horn-core-unique}
  Let $\sigma$ be
  a non-degenerate $k$-outsider, and let $0<i\le k$ be the anticipated horn position of $\sigma$;
  see \cref{lem:esd-esdp-conn-inner-anod-horn-pos}~\eqref{item:lem-esd-esdp-conn-inner-anod-horn-pos:range}.
  If we have $u_\sigma(i-1,0) = u_\sigma(i,0)$ and $u_\sigma(i,1) = u_\sigma(k,0)$,
  then $\sigma$ is the core of a unique horn pair.
\end{lemma}
\begin{proof}
  Consider the $i$-th face $\tau$ of $\sigma$. By the definition of a horn position and
  \cref{lem:esd-esdp-conn-inner-anod-horn-pos}~\eqref{item:lem-esd-esdp-conn-inner-anod-horn-pos:horn-pos},
  the uniqueness is clear: if $\sigma$ is the core of a horn pair, its periphery can only be $\tau$.

  It remains to demonstrate that $p=(\sigma, \tau)$ is a horn pair, for which it suffices
  to show that $i<k$, that $u_\sigma(i-1,1) < u_\sigma(i,1)$, and that $\tau\in\mathcal{O}$.
  Of these three the second is immediate from the definition of $i$. The first claim $i<k$ follows from the fact that
  $u_\sigma$ is order-preserving: if we were to have $i=k$, then the following contradictory
  inequalities would hold:
  \[
    u_\sigma(i-1,1) < u_\sigma(i,1) = u_\sigma(k,0) = u_\sigma(i,0) = u_\sigma(i-1,0) \leq u_\sigma(i-1,1).
  \]

  We shall now check that $\tau$ is a non-degenerate outsider. The non-degeneracy of $\tau$ is trivial,
  since $\tau$ is a face of the non-degenerate $\sigma$ and $Y$ is the nerve of a poset.
  To show that $\tau$ is an $(k-1)$-outsider, we remember
  \cref{lem:esd-esdp-conn-inner-anod-outsider-cond}~\eqref{item:lem-esd-esdp-conn-inner-anod-outsider-cond:out}.
  Since $0 < i < k$, we have the following inequality:
  \[
    u_\tau(k-1,0) = u_\sigma(k,0) > u_\sigma(0,1) = u_\tau(0,1).
  \]
  Also, $u_\tau$ is surjective, since we have:
  \[
    \operatorname{Im} u_\tau 
    = \operatorname{Im} \mleft(\mleft. u_\sigma
      \mright|_{(\deltaBra{k}\times\deltaBra{1})\setminus(\set{i}\times\deltaBra{1})}
      \mright)
    = \operatorname{Im} u_\sigma
    = \deltaBra{n}.
  \]
  Here, the second equality follows from $u_\sigma(i-1,0) = u_\sigma(i,0)$ and $u_\sigma(i,1) = u_\sigma(k,0)$.
  Therefore, $\tau$ is a non-degenerate $(k-1)$-outsider, and the claim has now been proven.
\end{proof}

\begin{lemma}\label{lem:esd-esdp-conn-inner-anod-horn-periph-unique}
  We apply \cref{not:esd-esdp-conn-inner-anod} here. Let $\tau$ be
  a non-degenerate $k$-outsider, and let $0<i\le k$ be the anticipated horn position of $\tau$.
  Assume that we have either $u_\tau(i-1,0) < u_\tau(i,0)$ or $u_\tau(i,1) > u_\tau(k,0)$.
  Then $\tau$ is the periphery of a unique horn pair.
\end{lemma}
\begin{proof}
  We begin with the uniqueness. Let $\sigma$ be any non-degenerate $(k+1)$-outsider such that
  $p=(\sigma, \tau)$ is a horn pair. The horn position of $(\sigma, \tau)$ is equal to $i$ by
  \cref{lem:esd-esdp-conn-inner-anod-horn-pos}~\eqref{item:lem-esd-esdp-conn-inner-anod-horn-pos:horn-pos}.
  Since $\tau$ is the $i$-th face of $\sigma$, we have:
  \begin{alignat}{2}
    u_\sigma(j,l) &= u_\tau(j,l) &\quad&\text{for } 0\le j < i, 0\le l\le 1;
    \label{eq:lem-esd-esdp-conn-inner-anod-horn-periph-unique:1}\\
    u_\sigma(j,l) &= u_\tau(j-1,l) &\quad&\text{for } i< j \le k+1, 0\le l\le 1.
    \label{eq:lem-esd-esdp-conn-inner-anod-horn-periph-unique:2}
  \end{alignat}
  By the definition of a horn pair, we also have:
  \begin{alignat}{2}
    u_\sigma(i,0) &= u_\sigma(i-1,0) &&= u_\tau(i-1,0);
    \label{eq:lem-esd-esdp-conn-inner-anod-horn-periph-unique:3}\\
    u_\sigma(i,1) &= u_\sigma(k+1,0) &&= u_\tau(k,0).
    \label{eq:lem-esd-esdp-conn-inner-anod-horn-periph-unique:4}
  \end{alignat}
  \Cref{eq:lem-esd-esdp-conn-inner-anod-horn-periph-unique:1,eq:lem-esd-esdp-conn-inner-anod-horn-periph-unique:2,%
  eq:lem-esd-esdp-conn-inner-anod-horn-periph-unique:3,eq:lem-esd-esdp-conn-inner-anod-horn-periph-unique:4}
  together uniquely determine the map $u_\sigma$, and hence the simplex $\sigma$. This is the desired uniqueness.

  We proceed to the existence of the horn pair.
  By the definition of the anticipated horn position, we have
  $u_\tau(i,1) \ge u_\tau(k,0)$. Therefore,
  \cref{eq:lem-esd-esdp-conn-inner-anod-horn-periph-unique:1,eq:lem-esd-esdp-conn-inner-anod-horn-periph-unique:2,%
  eq:lem-esd-esdp-conn-inner-anod-horn-periph-unique:3,eq:lem-esd-esdp-conn-inner-anod-horn-periph-unique:4}
  well-define an order-preserving map $u_\sigma\colon\deltaBra{k+1}\times\deltaBra{1}\to\deltaBra{n}$,
  and hence a $(k+1)$-simplex $\sigma$ of $Y$. It suffices to show that $p=(\sigma, \tau)$ is a horn pair.
  By the definition of $i$ and $\sigma$, we see that $u_\sigma(i-1,0) = u_\sigma(i,0)$;
  that $u_\sigma(i-1,1)<u_\sigma(i,1) = u_\sigma(k+1,0)$; and that $\tau$ is the $i$-th face of $\sigma$.

  The only thing left to check is that $\sigma$ is a non-degenerate outsider. We should be reminded of
  \labelcref{item:lem-esd-esdp-conn-inner-anod-outsider-cond:nd,item:lem-esd-esdp-conn-inner-anod-outsider-cond:out}
  from \cref{lem:esd-esdp-conn-inner-anod-outsider-cond}. The surjectivity of $u_\sigma$ is immediate
  from that of $u_\tau$. The inequality $u_\sigma(0,1) < u_\sigma(k+1,0)$ is also clear from the
  corresponding inequality $u_\tau(0,1) < u_\tau(k,0)$. Since $\tau$ is non-degenerate and we have
  $u_\sigma(i-1,1) < u_\sigma(i,1)$, the only non-trivial inequality needed for the non-degeneracy
  of $\sigma$ is $(u_\sigma(i,0), u_\sigma(i,1)) < (u_\sigma(i+1,0), u_\sigma(i+1,1))$ with respect
  to the product order.
  There are two cases to consider: $u_\tau(i-1,0) < u_\tau(i,0)$ and $u_\tau(i,1) > u_\tau(k,0)$.
  In the former case, we compute:
  \[
    u_\sigma(i,0) = u_\tau(i-1,0) < u_\tau(i,0) = u_\sigma(i+1,0).
  \]
  In the latter case, we deduce:
  \[
    u_\sigma(i,1) = u_\tau(k,0) < u_\tau(i,1) = u_\sigma(i+1,1).
  \]
  Either way, $\sigma$ is a non-degenerate outsider, and the claim has been proven.
\end{proof}

\begin{corollary}\label{cor:esd-esdp-conn-inner-anod-horn-pair-unique}
  Every non-degenerate outsider is contained in a unique horn pair,
  either as its core or its periphery.
\end{corollary}
\begin{proof}
  This is a direct consequence of
  \cref{lem:esd-esdp-conn-inner-anod-core-not-periphery,lem:esd-esdp-conn-inner-anod-horn-core-unique,%
  lem:esd-esdp-conn-inner-anod-horn-periph-unique}. If $\sigma$ is a non-degenerate outsider and $0<i\le k$
  is its anticipated horn position, we only need to perform a case analysis on whether or not we have both
  $u_\sigma(i-1,0) = u_\sigma(i,0)$ and $u_\sigma(i,1) = u_\sigma(k,0)$.
\end{proof}

The final lemma before the proof of \cref{lem:esd-esdp-conn-inner-anod} is the following;
we shall define a well-order on the set of horn pairs:

{
\newcommand{\property}{($\star$)}
\begin{lemma}\label{lem:esd-esdp-conn-inner-anod-horn-pair-order}
  There is an irreflexive well-order
  $\prec$ on the finite set $\mathcal{H}$ of horn pairs that satisfies the following property \property:
\begin{itemize}
	\item[\property]  Let $\sigma$ be a non-degenerate outsider, and $\tau$ be a proper 
  face of $\sigma$ of any dimension. Let $p$ be the horn pair
  containing $\sigma$ (see \cref{cor:esd-esdp-conn-inner-anod-horn-pair-unique}).
  Then exactly one of the following holds:
  \begin{enumerate}
    \item\label{item:lem-esd-esdp-conn-inner-anod-horn-pair-order:eq}%
    we have $p=(\sigma, \tau)$;
    \item\label{item:lem-esd-esdp-conn-inner-anod-horn-pair-order:lt}%
    the simplex $\tau$ is contained in a horn pair $q\in\mathcal{H}$ such that $q\prec p$;
    \item\label{item:lem-esd-esdp-conn-inner-anod-horn-pair-order:img}%
    the simplex $\tau$ is in the image of $f$.
  \end{enumerate}
\end{itemize}
\end{lemma}
\begin{proof}
  As $Y$ is finite, it is automatic that $\mathcal{H}$ is finite.
  Note that the required property \property\ 
  is preserved under extension of the partial order $\prec$.
  Since any order may be linearly extended and any finite linear order is well-ordered,
  it suffices to define an irreflexive \emph{partial} order $\prec$
  on $\mathcal{H}$ that satisfies the property \property.
  
  Consider the set $\metanats^3=\metanats$ of ordered triples of natural numbers, and
  equip this set with the (irreflexive) lexicographic order $<$, where each coordinate
  $\metanats$ is ordered in the usual way. We define a map $\phi\colon\mathcal{H}\to\metanats^3$
  by sending each horn pair $p=(\sigma, \tau)$ of horn position $i$ to the triple
  $(\dim\sigma,\, u_\sigma(\dim\sigma,0),\, i)$.
  We define the irreflexive partial order $\prec$ on $\mathcal{H}$
  by pulling back the lexicographic order $<$ on $\metanats^3$ under $\phi$.
  
  We shall now show that $\prec$ satisfies the property \property.
  Let $\sigma$ be a non-degenerate outsider, and $\tau$ be a proper face of $\sigma$.
  We first remind the reader that $Y$ is nerve of a poset, so that $\tau$, being a
  face of non-degenerate $\sigma$, is also non-degenerate.
  Also, notice that the conditions
  \labelcref{item:lem-esd-esdp-conn-inner-anod-horn-pair-order:eq,%
  item:lem-esd-esdp-conn-inner-anod-horn-pair-order:lt,%
  item:lem-esd-esdp-conn-inner-anod-horn-pair-order:img}
  in the statement are mutually exclusive; we only need to show that $\sigma$ and $\tau$
  enjoy at least one of the three conditions
  \labelcref{item:lem-esd-esdp-conn-inner-anod-horn-pair-order:eq,%
  item:lem-esd-esdp-conn-inner-anod-horn-pair-order:lt,%
  item:lem-esd-esdp-conn-inner-anod-horn-pair-order:img}.
  
  If $\tau$ is in the image of $f$, then there is nothing to prove;
  hence we may assume that $\tau$ is an outsider.
  Let $p$ and $q$ be the horn pairs containing $\sigma$ and $\tau$, respectively.
  Write $p=(\sigma_0, \sigma_1)$ and $q=(\tau_0, \tau_1)$. If either $\sigma=\sigma_1$ or $\tau=\tau_0$
  holds, then $\dim \tau < \dim \sigma$, $\dim\sigma_0 = \dim\sigma_1 + 1$, and $\dim\tau_0 = \dim\tau_1 + 1$
  imply that $\dim\tau_0 < \dim\sigma_0$, and hence that
  $q\prec p$. Therefore, we may assume that $\sigma=\sigma_0$, i.e., $\sigma$ is the core of $p$,
  and that $\tau=\tau_1$, i.e., $\tau$ is the periphery of $q$.
  
  Let $k=\dim\sigma$, and let $0<i<k$ be the horn position of $p$. If $\dim\tau<k-1$, then
  we have $\dim\tau_0 < \dim\sigma_0$, and hence $q\prec p$. Therefore, we may assume that
  $\dim\tau=k-1$, which gives $\dim \tau_0 = k = \dim\sigma_0$.
  Say that the $\tau$ is the $j$-th facet of $\sigma$. Since $\tau$ is assumed to
  be a horn periphery,
  by \cref{lem:esd-esdp-conn-inner-anod-horn-core-unique,lem:esd-esdp-conn-inner-anod-core-not-periphery},
  we only have three cases to consider: $j=i-1$, $j=i$, and $j=k$.
  If $j=i$, then we have $p=q=(\sigma, \tau)$, and we are done.
  
  Next, consider the case $j=i-1$. In this case, we have:
  \[
  u_{\tau_0}(k,0) = u_{\tau}(k-1,0) = u_{\sigma}(k,0).
  \]
  The horn position of $q$, equal to the anticipated horn position $i-1$ of $\tau$,
  is smaller than $i$. Therefore we have $q\prec p$.
  
  Finally, consider the case $j=k$. In order for $\tau$ to be a periphery,
  $u_{\tau_0}(k,0) = u_\tau(k-1,0) = u_\sigma(k-1,0)$ must be strictly less than
  $u_\sigma(k,0)$. In this case, we have $q\prec p$ by definition.
  
  We have analyzed all possible cases, and the claim has been proven.
\end{proof}
}

Now we have paired the simplices out of the image of $f$ so that they may be added
via inner horn fillings, so we may now proceed to the proof of the lemma in question.

\begin{proof}[Proof of \cref{lem:esd-esdp-conn-inner-anod}]
  \label{proof:lem:esd-esdp-conn-inner-anod}
  We employ \cref{not:esd-esdp-conn-inner-anod} throughout this proof.
  We shall show that the map $f\colon X\hookrightarrow Y$ is inner anodyne by
  decomposing $f$ into the composition of a finite sequence of inner horn fillings.
  Here, an inner horn filling stands for the pushout along any simplicial map of any
  inner horn inclusion: $\mtsHorn{k}{l} \hookrightarrow \deltaBra{k}$ with
  $0<l<k$.

  Let $\prec$ be an irreflexive well-order on the finite set $\mathcal{H}$ 
  that has been asserted to exist in
  \cref{lem:esd-esdp-conn-inner-anod-horn-pair-order}. Write $\mathcal{H}$ as
  \[
    \mathcal{H} = \set{p_0 \prec p_1 \prec \dotsb \prec p_{m-1}}.
  \]
  Write $p_i=(\sigma_i, \tau_i)$. For any $0\le i \le m$, we define $X_i \subseteq Y$
  as the smallest subcomplex of $Y$ that contains the image of $f$ and the simplices
  $\sigma_j$ for $j < i$. Then there is the following filtration:
  \[
    \operatorname{Im} f 
    = X_0 \subseteq X_1 \subseteq \dotsb \subseteq X_{m-1} \subseteq X_m
  \]
  Notice that $X_m = Y$, because every outsider is a degeneracy of a non-degenerate outsider, and
  $X_m$ contains all non-degenerate outsiders.

  It suffices to show that the inclusion $X_{j-1}\hookrightarrow X_j$ is an inner horn filling
  for every $0 < j \le m$. Write $p_{j-1} = (\sigma_{j-1}, \tau_{j-1})$. Let $k=\dim\sigma_{j-1}$, and let
  $i$ be the horn position of $p_{j-1}$. Then, by the property of the well-order $\prec$ that
  \cref{lem:esd-esdp-conn-inner-anod-horn-pair-order} guarantees, we have the following pushout
  square:
  \[
    \begin{tikzcd}
      \mtsHorn{k}{i} \arrow[d, hook] \arrow[r, hook]
      & \mtyndSpx\deltaBra{k} \arrow[d, hook, "\sigma_{j-1}"] \\
      X_{j-1} \arrow[r, hook] & X_j
    \end{tikzcd}
  \]
  This is exactly what we wanted to show.
\end{proof}

\subsection{Properties of the transformations 3: another inner anodyne map}
\label{subsec:down-last-infty-loc-shape-inner-anod2}

There is another natural transformation that we need to show are pointwise inner anodyne:

\begin{proposition}\label{prop:esdi-esdpi-conn-inner-anod}
  Let $X\in\mtcSSet$ be a simplicial set. Consider the following commutative
  diagram:
  \[
    \begin{tikzcd}
      \mtsESd X \arrow[r, hookrightarrow] \arrow[d, hookrightarrow] &
      \mtsESdp X \arrow[d, hookrightarrow] \\
      \mtsESdI X \arrow[r, hookrightarrow] &
      \mtsESdpI X
    \end{tikzcd}
  \]
  Here, the vertical maps are from \cref{def:dcp-esd-esdp-incl-to-i}, and the
  horizontal maps are from \cref{def:esd-esdp-connecting-inj}.
  Then the corresponding simplicial map
  \[
    (\mtsESdp X) \cup_{\mtsESd X} (\mtsESdI X) \to \mtsESdpI X
  \]
  out of the pushout is inner anodyne.
\end{proposition}

Before providing the proof of this proposition, we state a simple corollary:

\begin{corollary}\label{cor:esd-esdp-conn-pushout-inner-anod}
  Let $f\colon X\to Y$ be any map of simplicial sets. 
  Consider the following morphism of spans of simplicial sets:
  \[
  \begin{tikzcd}
    \mtsESd Y \arrow[d, hook]
    & \mtsESd X \arrow[d, hook] \arrow[l, "f"'] \arrow[r, hook]
    & \mtsESdI X \arrow[d, hook] \\
    \mtsESdp Y
    & \mtsESdp X \arrow[l, "f"] \arrow[r, hook]
    & \mtsESdpI X
  \end{tikzcd}
  \]
  Then, the following induced map of pushouts is inner anodyne:
  \[
  (\mtsESd Y)\cup_{\mtsESd X}(\mtsESdI X)
  \to (\mtsESdp Y) \cup_{\mtsESdp X} (\mtsESdpI X).
  \]
\end{corollary}
\begin{proof}[Proof assuming \cref{prop:esdi-esdpi-conn-inner-anod}]
  The previous \cref{prop:esdi-esdpi-conn-inner-anod} states that the
  morphism of the spans in question is Reedy inner anodyne, which immediately implies
  the claim.
\end{proof}

The proof of \cref{prop:esdi-esdpi-conn-inner-anod} will be conducted similarly to that of
\cref{prop:esd-esdp-conn-inner-anod}; it will be simple with the presence of
the following lemma:

\begin{lemma}\label{lem:esdi-esdpi-conn-inner-anod}
  Let $\deltaBra{n}\in\mtcSimpx$ be a simplex. Consider the following commutative cube:
  \begin{equation}\label{eq:lem-esdi-esdpi-conn-inner-anod:cube}
    \begin{tikzcd}[column sep=small, row sep=small]
      & \mtsESdI\mtBdrySpx\deltaBra{n} \arrow[hook, rr] \arrow[hook, dd]
      && \mtsESdI\mtyndSpx\deltaBra{n} \arrow[hook, dd]
      \\
      \mtsESd\mtBdrySpx\deltaBra{n} \arrow[hook, dd] \arrow[hook, rr, crossing over] \arrow[hook, ru]
      && \mtsESd\mtyndSpx\deltaBra{n} \arrow[hook, ru]
      &\\
      & \mtsESdpI\mtBdrySpx\deltaBra{n} \arrow[hook, rr]
      && \mtsESdpI\mtyndSpx\deltaBra{n}
      \\
      \mtsESdp\mtBdrySpx\deltaBra{n} \arrow[hook, rr] \arrow[hook, ru]
      && \mtsESd\mtyndSpx\deltaBra{n} \arrow[hook, ru] \arrow[hook, from=uu, crossing over]
      &
    \end{tikzcd}
  \end{equation}
  Let $X$ denote the colimit of the cube except $Y\coloneqq\mtsESdpI\mtyndSpx\deltaBra{n}$.
  Then the canonical simplicial map $X\to Y$ induced by the cube is inner anodyne.
\end{lemma}

\begin{proof}[Proof of \cref{prop:esdi-esdpi-conn-inner-anod} from \cref{lem:esdi-esdpi-conn-inner-anod}]
  By \cref{lem:cospx-kan-ext-reedy-nat}, it suffices to prove that the natural transformation
  \[
    (\mtsESdp \mtyndSpx\deltaBra{n}) \cup_{\mtsESd \mtyndSpx\deltaBra{n}} (\mtsESdI \mtyndSpx\deltaBra{n})
    \to \mtsESdpI \mtyndSpx\deltaBra{n}
  \]
  in $\deltaBra{n}\in\mtcSimpx$ is Reedy inner anodyne, which is proved in
  \cref{lem:esdi-esdpi-conn-inner-anod}.
\end{proof}

The rest of this section will be devoted to the proof of \cref{lem:esdi-esdpi-conn-inner-anod}.
As with the previous subsection, we shall first establish some preliminary lemmas. The concluding proof
will appear in \cpageref{proof:lem:esdi-esdpi-conn-inner-anod} 
after \cref{lem:esdi-esdpi-conn-inner-anod-horn-pair-order}.
We again start with some notations:

\begin{notation}\label{not:esdi-esdpi-conn-inner-anod}
  The notations listed here will be used solely for the proof of \cref{lem:esdi-esdpi-conn-inner-anod}.
  This applies to the rest of this subsection.
  \begin{itemize}
    \item We use some notations from \cref{not:esd-esdp-conn-inner-anod} with varying values of $n$.
      Since \cref{not:esd-esdp-conn-inner-anod} fixes $n$, we need to rename some terminologies and notations
      for each case where $n=m$:
      \begin{itemize}
        \item A horn pair shall be called a $m$-horn pair. The set of all $m$-horn pairs
          will be denoted by $\mathcal{H}^{m}$.
        \item An outsider shall be called a $(m,\bullet)$-outsider; a $k$-outsider a $(m,k)$-outsider.
          The set of all non-degenerate $(m,k)$-outsiders will be denoted by $\mathcal{O}^{m}_k$;
          the set of all non-degenerate $(m,\bullet)$-outsiders will be denoted by $\mathcal{O}^{m}$.
        \item The notation $u_\sigma\colon \deltaBra{k}\times\deltaBra{1}\to\deltaBra{m}$, 
          where $\sigma$ is a simplex of $\mtsESdp\mtyndSpx\deltaBra{m}$,
          will be used and does not need to be renamed.
        \item The terms \emph{horn position}, \emph{core}, and \emph{periphery} of an $m$-horn pair
          will be used as is.
        \item Other notations and terminologies will not be used; in particular, $X$, $Y$, and
          $f$ are reserved for definitions of this \namecref{not:esdi-esdpi-conn-inner-anod}'s own.
      \end{itemize}
    \item Apart from $m$, which varies for the use of \cref{not:esd-esdp-conn-inner-anod},
      we shall fix a simplex $\deltaBra{n}\in\mtcSimpx$.
    \item The map in question, which we aim to show is inner anodyne,
      will be denoted by $f\colon X \to Y$.
    \item A $k$-simplex $\sigma$ of $Y\cong\nerve(\mtpESdpI\deltaBra{n})$
     may be considered as a map of posets
      \[
        \sigma\colon\deltaBra{k}\to \mtpESdpI\deltaBra{n}
        = (\set{0}\times \functCat{\deltaBra{1}}{\deltaBra{n}}) \cup
        (\set{1}\times \deltaBra{n}),
      \]
      the order of whose codomain is given in \cref{def:esdp-esdpi}.
      Let $\deltaBra{k'} \subseteq \deltaBra{k}$ in $\mtcSimpxAug\subseteq\mtcPoset$
      be the preimage of $\set{0}\times \functCat{\deltaBra{1}}{\deltaBra{n}}$ under $\sigma$;
      let $-1 \le n' \le n$ be the smallest integer such that
      $\sigma(\deltaBra{k'}) \subseteq \set{0}\times\functCat{\deltaBra{1}}{\deltaBra{n'}}$.
      Then:
      \begin{itemize}
        \item The integer $-1 \le k' \le k$ is called the \emph{switch position} of $\sigma$;
          the integer $-1 \le n' \le n$ is called the \emph{switch value} of $\sigma$.
        \item The \emph{zeroth part} of $\sigma$ is the map
        \[
          \sigma_0\colon\deltaBra{k'} \to \set{0}\times\functCat{\deltaBra{1}}{\deltaBra{n'}}
          \cong \functCat{\deltaBra{1}}{\deltaBra{n'}},
        \]
        where the left arrow is the restriction of $\sigma$. Note that, if $k'\ge 0$, then
        $n'\ge 0$ and we may consider $\sigma_0$ as a $k'$-simplex
        of $\mtsESdp\mtyndSpx\deltaBra{n'}$. 
        \item The \emph{first part} of $\sigma$ is the map
        \[
          \hspace{5em}\sigma_1\colon\deltaBra{k}\setminus\deltaBra{k'} \to
          \set{1}\times\set{n', n'+1, \dotsc, n}
          \cong \set{n', n'+1, \dotsc, n},
        \]
        where the left arrow is the restriction of $\sigma$.
        \item We say that $\sigma$ \emph{has a non-empty zeroth part} if $k'\ge 0$;
          \emph{has an empty zeroth part} if $k'=-1$;
          \emph{has a non-empty first part} if $k' < k$;
          \emph{has an empty first part} if $k' = k$.
        \item We say that $\sigma$ \emph{has a straight first part} if
          $\sigma_1$ is an order embedding, and its image either
          is the full codomain $\set{n', n'+1, \dotsc, n}$ or misses only the smallest
          number: $\set{n'+1, n'+2, \dotsc, n}$.
      \end{itemize}
    \item An \emph{I-outsider} is a simplex of $Y$ that is not in the image of the map $f$.
      A $k$-\emph{I-outsider} is an I-outsider of dimension $k$.
    \item The set of all non-degenerate $k$-I-outsiders will be denoted by
      $\mathcal{O}^{\mathrm{I}}_k$; that of all non-degenerate I-outsiders by $\mathcal{O}^{\mathrm{I}}$.
    \item An \emph{I-horn pair} is a pair $(\sigma, \tau)$ of simplices of $Y$ satisfying the followings:
      \begin{itemize}
        \item The simplex $\tau$ is some facet of $\sigma$.
        \item The simplex $\sigma$ has a non-empty and straight first part.
        \item The switch value $n'$ of $\sigma$ and $\tau$ are equal and non-negative, implying that
          the two simplices have non-empty zeroth parts.
        \item Let $\sigma_0$ and $\tau_0$ denote the zeroth parts of $\sigma$ and $\tau$, respectively.
          Then $(\sigma_0, \tau_0)$ is an $n'$-horn pair.
      \end{itemize}
    \item The set of all I-horn pairs will be denoted by $\mathcal{H}^{\mathrm{I}}$.
    \item If $p=(\sigma, \tau)$ is an I-horn pair, then $\sigma$ is called the \emph{core} of $p$,
      and $\tau$ the \emph{periphery} of $p$.
  \end{itemize}
\end{notation}

Except that we separately prove that $f$ is injective,
the proof of \cref{lem:esdi-esdpi-conn-inner-anod} will be conducted in a similar manner to that of
\cref{lem:esd-esdp-conn-inner-anod}: we show that every I-outsider is contained in a unique I-horn pair,
and construct a well-order on the set of I-horn pairs to determin the order in which the
horns are to be filled.

\begin{lemma}\label{lem:esdi-esdpi-conn-inner-anod-pushout-inj}
  We invoke \cref{not:esdi-esdpi-conn-inner-anod}. The map $f\colon X\to Y$ is injective.
\end{lemma}
\begin{proof}
  Consider the commutative cube \labelcref{eq:lem-esdi-esdpi-conn-inner-anod:cube}
  in the statement of \cref{lem:esdi-esdpi-conn-inner-anod}.
  The desired injectivity of $f$ follows from the injectivity of all arrows in the cube and
  the fact that every face of the cube is a pullback square.
\end{proof}

\begin{lemma}\label{lem:esdi-esdpi-conn-inner-anod-outsider-cond}
  Let $\sigma$ be a $k$-simplex of $Y$.
  Let $\sigma_0\colon \deltaBra{k'}\to\functCat{\deltaBra{1}}{\deltaBra{n'}}$ and
  $\sigma_1\colon \deltaBra{k}\setminus\deltaBra{k'}\to\set{n', \dotsc, n}$ be the zeroth and first parts,
  respectively, of $\sigma$.
  Then we have the following:
  \begin{enumerate}
    \item\label{item:lem-esdi-esdpi-conn-inner-anod-outsider-cond:nd}%
      The simplex $\sigma$ is non-degenerate if and only if both $\sigma_0$ and $\sigma_1$ are injective.
    \item\label{item:lem-esdi-esdpi-conn-inner-anod-outsider-cond:nd0}%
      Given that $\sigma$ has a non-empty zeroth part, the zeroth part $\sigma_0$
      is injective if and only if it is a non-degenerate simplex in $\mtsESdp\mtyndSpx\deltaBra{n'}$.
    \item\label{item:lem-esdi-esdpi-conn-inner-anod-outsider-cond:esdp}%
      The simplex $\sigma$ is not in the injective image of $\mtsESdp\mtyndSpx\deltaBra{n}$
      if and only if it has a non-empty first part.
    \item\label{item:lem-esdi-esdpi-conn-inner-anod-outsider-cond:esdi}%
      The simplex $\sigma$ is not in the injective image of $\mtsESdI\mtyndSpx\deltaBra{n}$
      if and only if it has a non-empty zeroth part and the zeroth part
      $\sigma_0$ satisfies $u_{\sigma_0}(0,1) < u_{\sigma_0}(k',0)$.
    \item\label{item:lem-esdi-esdpi-conn-inner-anod-outsider-cond:bdry}%
      The simplex $\sigma$ is not in the injective image of $\mtsESdpI\mtBdrySpx\deltaBra{n}$
      if and only if the adjunct $u_{\sigma_0}\colon \deltaBra{k'}\times\deltaBra{1}\to\deltaBra{n'}$
      of the zeroth part is surjective and the image of $\sigma_1$ includes
      $\set{n'+1, n'+2, \dotsc, n}$.
    \item\label{item:lem-esdi-esdpi-conn-inner-anod-outsider-cond:ndout}%
      The simplex $\sigma$ is a non-degenerate I-outsider if and only if it has
      a non-empty and straight first part and a non-empty zeroth part $\sigma_0$
      that is a non-degenerate $(n',k')$-outsider.
  \end{enumerate}
\end{lemma}
\begin{proof}
  The first five claims follows from constrution.
  The last claim is a direct consequence of the first five: see
  \cref{lem:esd-esdp-conn-inner-anod-outsider-cond}.
\end{proof}

\begin{lemma}\label{lem:esdi-esdpi-conn-inner-anod-ihorn-pair-propert}
  Let $(\sigma, \tau)$ be an I-horn pair.
  Let $k$ be the dimension of $\sigma$;
  $n'$ be the switch value of $\sigma$ and $\tau$;
  $0\le k' < k$ be the switch position of $\sigma$;
  $\sigma_i$ and $\tau_i$ be the $i$-th part of $\sigma$ and $\tau$, respectively, for $i=0,1$.
  Then we have the following:
  \begin{enumerate}
    \item\label{item:lem-esdi-esdpi-conn-inner-anod-ihorn-pair-propert:swp}%
      The switch position of $\tau$ is $k'-1\ge 0$.
    \item\label{item:lem-esdi-esdpi-conn-inner-anod-ihorn-pair-propert:pos}%
      The simplex $\tau$ is the $i$-th facet of $\sigma$ for some unique
      $0\le i\le k$, which satisfies $0<i<k'$.
    \item\label{item:lem-esdi-esdpi-conn-inner-anod-ihorn-pair-propert:1st-eq}%
      Let $q$ denote the unique order isomorphism 
      \[
        q\colon\deltaBra{k-1}\setminus\deltaBra{k'-1}\to\deltaBra{k}\setminus\deltaBra{k'};
        \quad i\mapsto i+1.
      \]
      Then the following triangle is commutative:
      \[
        \begin{tikzcd}[column sep=0]
          \deltaBra{k-1}\setminus\deltaBra{k'-1} \arrow[rr, "q"', "\sim"] \arrow[rd, "\tau_1"']
          && \deltaBra{k}\setminus\deltaBra{k'} \arrow[ld, "\sigma_1"] \\
          & \set{n', n'+1, \dotsc, n} &
        \end{tikzcd}
      \]
    \item\label{item:lem-esdi-esdpi-conn-inner-anod-ihorn-pair-propert:1st-p}%
      The simplex $\tau$ has a non-empty and straight first part.
    \item\label{item:lem-esdi-esdpi-conn-inner-anod-ihorn-pair-propert:ndout}%
      The simplices $\sigma$ and $\tau$ are non-degenerate I-outsiders.
  \end{enumerate}
\end{lemma}
\begin{proof}
  \eqref{item:lem-esdi-esdpi-conn-inner-anod-ihorn-pair-propert:swp}:
  By definition, we have $(\sigma_0,\tau_0)\in\mathcal{H}^{n'}$;
  therefore, $\tau_0$ must be a facet of $\sigma_0$,
  which proves the claim.

  \eqref{item:lem-esdi-esdpi-conn-inner-anod-ihorn-pair-propert:pos}:
  By inspecting each facet of $\sigma$, we see that $i$ must be the horn position of
  $(\sigma_0, \tau_0)$.

  \eqref{item:lem-esdi-esdpi-conn-inner-anod-ihorn-pair-propert:1st-eq}:
  Follows from \eqref{item:lem-esdi-esdpi-conn-inner-anod-ihorn-pair-propert:pos}.

  \eqref{item:lem-esdi-esdpi-conn-inner-anod-ihorn-pair-propert:1st-p}:
  Follows from \eqref{item:lem-esdi-esdpi-conn-inner-anod-ihorn-pair-propert:1st-eq} and
  the fact that $\sigma$ satisfies the same property.

  \eqref{item:lem-esdi-esdpi-conn-inner-anod-ihorn-pair-propert:ndout}:
  We apply \cref{lem:esdi-esdpi-conn-inner-anod-outsider-cond}~%
  \eqref{item:lem-esdi-esdpi-conn-inner-anod-outsider-cond:ndout} to $\sigma$ and $\tau$.
  Since $(\sigma_0, \tau_0)$ is an $n'$-horn pair,
  we see that $\sigma$ and $\tau$ has non-empty zeroth parts,
  that $\sigma_0\in\mathcal{O}^{n'}_{k'}$, and that $\tau_0\in\mathcal{O}^{n'}_{k'-1}$.
  By \eqref{item:lem-esdi-esdpi-conn-inner-anod-ihorn-pair-propert:1st-p} and
  the definition of I-horn pairs, $\sigma$ and $\tau$ has non-empty and straight first parts.
  The claim follows from these observations.
\end{proof}

\begin{lemma}\label{lem:esdi-esdpi-conn-inner-anod-horn-pair-unique}
  Every non-degenerate I-outsider
  belongs to a unique I-horn pair.
\end{lemma}
\begin{proof}
  Let $\sigma\in\mathcal{O}^{\mathrm{I}}_k$ be a non-degenerate $k$-I-outsider.
  By \cref{lem:esdi-esdpi-conn-inner-anod-outsider-cond}~%
  \eqref{item:lem-esdi-esdpi-conn-inner-anod-outsider-cond:ndout},
  $\sigma$ has a non-empty and straight first part and a non-empty zeroth part $\sigma_0$
  that is a non-degenerate $(n',k')$-outsider, which implies $k \le 2$ and $0 < k' < k$.
  Let $p$ be the $n'$-horn pair
  containing $\sigma_0$. We have two cases to consider: that where
  $\sigma_0$ is the core of $p$, and that where $\sigma_0$ is the periphery of $p$.

  For the first case, let us assume that $p=(\sigma_0, \tau_0)$ for some $\tau_0$.
  Define an order-preserving map
  $\tau\colon \deltaBra{k-1}\to \mtpESdpI\deltaBra{n}$ by:
  \[
    \tau(i) \coloneqq \begin{cases}
      (0,\tau_0(i)) & \text{if }0\le i < k'; \\
      (1,\sigma(i+1)) & \text{if }k' \le i<k.
    \end{cases}
  \]
  We consider $\tau$ as a $(k-1)$-simplex of $\mtpESdpI\deltaBra{n}$. Now, we may
  easily check that $(\sigma, \tau)$ is an I-horn pair containing $\sigma$:
  consult the definitions of I-horn pairs and $n'$-horn pairs.

  For the second case, let us assume that $p=(\tau_0, \sigma_0)$ for some $\tau_0$.
  Define an order-preserving map
  $\tau\colon \deltaBra{k+1}\to \mtpESdpI\deltaBra{n}$ by:
  \[
    \tau(i) \coloneqq \begin{cases}
      (0,\tau_0(i)) & \text{if }0\le i \le k' + 1; \\
      (1,\sigma(i)) & \text{if }k' + 1 < i \le k+1.
    \end{cases}
  \]
  We consider $\tau$ as a $(k+1)$-simplex of $\mtpESdpI\deltaBra{n}$. Now, by the
  definitions of I-horn pairs and $n'$-horn pairs, it is easy to see that
  $(\tau, \sigma)$ is an I-horn pair containing $\sigma$.

  We have shown that every non-degenerate I-outsider belongs to an I-horn pair.
  It only remains to show that the I-horn pair is unique. Assume that $\sigma$ is contained
  in I-horn pairs $p,p'$. Let $\tau$ and $\tau'$ be the other consituent of $p$ and $p'$,
  respectively. Write $\sigma_i, \tau_i, \tau'_i$ for the $i$-th part of $\sigma, \tau, \tau'$,
  respectively, for $i=0,1$. By definition, $\sigma_0$ and $\tau_0$, with an appropriate order,
  form an $n'$-horn pair $p_0$, and $\sigma_0$ and $\tau'_0$ form an $n'$-horn pair $p'_0$.
  By \cref{cor:esd-esdp-conn-inner-anod-horn-pair-unique}, we have either:
  \begin{align*}
    (\sigma_0,\tau_0) &= p_0 = p'_0 = (\sigma_0,\tau'_0), \text{ or} \\
    (\tau_0,\sigma_0) &= p_0 = p'_0 = (\tau'_0,\sigma_0).
  \end{align*}
  Either way we get $\tau_0 = \tau'_0$. This also implies that $\sigma$ is either the core of
  both $p$ and $p'$, or the periphery of the both. Therefore, in order to conclude the proof,
  it suffices to show that $\tau_1 = \tau'_1$, which is a direct consequence of
  \cref{lem:esdi-esdpi-conn-inner-anod-ihorn-pair-propert}~%
  \eqref{item:lem-esdi-esdpi-conn-inner-anod-ihorn-pair-propert:1st-eq}.
\end{proof}

{
\newcommand{\property}{($\star$)}
\begin{lemma}\label{lem:esdi-esdpi-conn-inner-anod-horn-pair-order}
  We employ \cref{not:esdi-esdpi-conn-inner-anod}. There is an irreflexive well-order
  $\prec^{\mathrm{I}}$ on the finite set $\mathcal{H}^{\mathrm{I}}$ of I-horn pairs that satisfies the following property
  \property:
  \begin{itemize}
    \item[\property] \label{item:lem-esdi-esdpi-conn-inner-anod-horn-pair-order:prop}%
    Let $\sigma$ be a non-degenerate I-outsider, and $\tau$ be a proper face of $\sigma$.
    Let $p$ be the I-horn pair containing $\sigma$ (see \cref{lem:esdi-esdpi-conn-inner-anod-horn-pair-unique}).
    Then exactly one of the following holds:
    \begin{enumerate}
      \item\label{item:lem-esdi-esdpi-conn-inner-anod-horn-pair-order:eq}%
      $p=(\sigma, \tau)$;
      \item\label{item:lem-esdi-esdpi-conn-inner-anod-horn-pair-order:lt}%
      the simplex $\tau$ is contained in a horn pair $q\in\mathcal{H}^{\mathrm{I}}$ such that $q\prec^{\mathrm{I}} p$;
      \item\label{item:lem-esdi-esdpi-conn-inner-anod-horn-pair-order:img}%
      the simplex $\tau$ is in the image of $f$.
    \end{enumerate}
  \end{itemize}
\end{lemma}
\begin{proof}
  The finiteness of $\mathcal{H}^{\mathrm{I}}$ follows from the finiteness of $Y$.

  In order to define an irreflexive well-order $\prec^{\mathrm{I}}$ on $\mathcal{H}^{\mathrm{I}}$,
  we need some auxiliary orders on other sets.
  For each $0\le m\le n$, take an irreflexive well-order $\prec^m$ on $\mathcal{H}^m$
  guaranteed to exist by \cref{lem:esd-esdp-conn-inner-anod-horn-pair-order}.
  Consider the following disjoint union of sets:
  \[
    \mathcal{H}^{\bullet} \coloneqq \bigcup_{0\le m\le n} \mathcal{H}^m 
    = \bigsqcup_{0\le m\le n} \mathcal{H}^m.
  \]
  Equip this with an irreflexive well-order $\prec^\bullet$ by considering this union as the join or
  the ordinal sum $H^0 \star H^1 \star \dotsb \star H^n$.
  That is, if $p\in\mathcal{H}^m$ and $q\in\mathcal{H}^{m'}$,
  then $p\prec^\bullet q$ precisely if either $m < m'$ or $m=m'$ and $p\prec^m q$.
  Consider the irreflexively well-ordered set $P \coloneqq \metanats \times \mathcal{H}^{\bullet}$,
  where the order is the lexicographic order.

  We define a map $\psi\colon \mathcal{H}^{\mathrm{I}}\to P$ as follows: for each I-horn pair
  $p=(\sigma, \tau)$ with the zeroth parts of the components forming an $n'$-horn pair $p_0$,
  we set $\psi(p) \coloneqq (\dim\sigma, p_0)$. We define the order $\prec^{\mathrm{I}}$ on $\mathcal{H}^{\mathrm{I}}$
  by $p\prec^{\mathrm{I}} q$ if and only if $\psi(p)\prec^\bullet\psi(q)$. 
  By \cref{lem:esdi-esdpi-conn-inner-anod-horn-pair-unique,%
  lem:esdi-esdpi-conn-inner-anod-ihorn-pair-propert,lem:esdi-esdpi-conn-inner-anod-outsider-cond},
  We see that $\psi$ is injective;
  hence $\prec^{\mathrm{I}}$ is an irreflexive well-order.

  We now show that $\prec^{\mathrm{I}}$ satisfies \property. Let $\sigma$ be a non-degenerate I-outsider,
  and $\tau$ be a proper face of $\sigma$. Since $Y$ is nerve of a poset, the simplex $\tau$
  is non-degenerate.
  By applying the same argument as in the proof of \cref{lem:esd-esdp-conn-inner-anod-horn-pair-order}
  to \cref{lem:esdi-esdpi-conn-inner-anod-horn-pair-unique,%
  lem:esdi-esdpi-conn-inner-anod-ihorn-pair-propert}, we see
  that the conditions \labelcref{item:lem-esdi-esdpi-conn-inner-anod-horn-pair-order:eq,%
  item:lem-esdi-esdpi-conn-inner-anod-horn-pair-order:lt,%
  item:lem-esdi-esdpi-conn-inner-anod-horn-pair-order:img} are mutually exclusive, and that
  the only non-trivial case to consider is when
  $\sigma$ is the core of some I-horn pair $p=(\sigma,\sigma')$ and $\tau$ is a facet of $\sigma$ and
  the periphery of another I-horn pair $q=(\tau',\tau)$.

  Let $k'$ be the switch position of $\sigma$ and $m$ be the switch value of $\sigma$.
  Say that $\tau$ is the $i$-th facet of $\sigma$. By the definition of I-horn pairs and
  \cref{lem:esd-esdp-conn-inner-anod-core-not-periphery}, in order that $\tau$ is the
  periphery of $q$, we must have $0 \le i \le k'$. If the switch value $m'$ of $\tau$ is not
  equal to $m$, then we have $i = k'$ and $m' < m$; therefore, $q\prec^{\mathrm{I}} p$.
  We consider the case $m' = m$. Let $p_0=(\sigma_0,\sigma'_0)$ and $q_0=(\tau'_0,\tau_0)$ be the
  $m$-horn pairs formed by the zeroth parts of the components of $p$ and $q$, respectively.
  Since $\tau_0$ is a facet of $\sigma_0$, the defining property of $\prec^m$ from
  \cref{lem:esd-esdp-conn-inner-anod-horn-pair-order} applies and we have exactly one of
  $p_0 = (\sigma_0, \tau_0) = q_0$ or $q_0\prec^m p_0$. In the former case, we have
  $p = (\sigma, \tau) = q$; in the latter, we have $q\prec^{\mathrm{I}} p$.
\end{proof}
}

Now we conclude this subsection by proving \cref{lem:esdi-esdpi-conn-inner-anod}:

\begin{proof}[Proof of \cref{lem:esdi-esdpi-conn-inner-anod}]\label{proof:lem:esdi-esdpi-conn-inner-anod}
  By \cref{lem:esdi-esdpi-conn-inner-anod-pushout-inj}, it suffices to prove that
  the inclusion of the image of $f$ into $Y$ is inner anodyne.
  By virtue of \cref{lem:esdi-esdpi-conn-inner-anod-horn-pair-unique,%
  lem:esdi-esdpi-conn-inner-anod-horn-pair-order}, the rest of the proof goes exactly
  the same as that of \cref{lem:esd-esdp-conn-inner-anod} in
  \cpageref{proof:lem:esd-esdp-conn-inner-anod}.
\end{proof}

\subsection{Main lemma for \texorpdfstring{$(\infty,1)$}{(∞,1)}-localization}
\label{subsec:down-last-infty-loc-shape-main-lem}

The purpose of this subsection is \cref{cor:dcpc-esdpc-loc-homotopy-extension}, which is
a summarized form of the results of this section. For the sake of conciseness, we begin with
a new notation:

\begin{definition}\label{def:dcp-dcpi-esdp-esdpi-pushout}
  Let $f\colon X\to Y$ be any map of simplicial sets. We shall define the simplicial sets
  $\mtsDcpC f$ and $\mtsESdpC f$ (C for ``(mapping) cylinder'')
  as those that fills in the following pushout squares:
  \[
    \begin{tikzcd}
      \mtsDcp X \arrow[d, "\mtsDcp(f)"'] \arrow[r, hook] & \mtsDcpI X \arrow[d] \\
      \mtsDcp Y \arrow[r, hook] & \mtsDcpC f
    \end{tikzcd}
    \quad\text{and}\quad
    \begin{tikzcd}
      \mtsESdp X \arrow[d, "\mtsESdp(f)"'] \arrow[r, hook] & \mtsESdpI X \arrow[d] \\
      \mtsESdp Y \arrow[r, hook] & \mtsESdpC f
    \end{tikzcd}
  \]
  Note that there is a canonical map $\mtsDcpC f\to\mtsESdpC f$. We shall write
  $W^C_{f}\subseteq (\mtsDcpC f)_1$ for the set of edges that are mapped to degenerate edges
  in $\mtsESdpC f$.
\end{definition}

\begin{proposition}\label{prop:dcp-esd-pushout-localization}
  Let $f\colon X\to Y$ be any map of simplicial sets. 
  Then the canonical map $\mtsDcpC f\to \mtsESdpC f$ exhibits $\mtsESdpC f$ as the
  localization of $\mtsDcpC f$ at $W^C_{f}$.
\end{proposition}
\begin{proof}
  Follows from \cref{cor:esd-esdp-conn-pushout-inner-anod,cor:dcp-esd-conn-surj-pushout-univ-loc};
  note that an edge in $\mtsESdpC f$ is degenerate precisely if it is the image of
  a degenerate edge in $\mtsESd Y \cup_{\mtsESd X} \mtsESdI X$, for the map between these
  two simplicial sets is a monomorphism and an isomorphism on the set of vertices.
  See \href{https://kerodon.net/tag/01MX}{Remark 01MX} and
  \href{https://kerodon.net/tag/02M1}{Proposition 02M1} in \cite{kerodon} for more details.
\end{proof}

\begin{corollary}\label{cor:dcpc-esdpc-loc-homotopy-extension}
  Let $u\colon X\to Y$ be any map of simplicial sets, and $A \subseteq B$ be
  an inclusion of simplicial sets. Set $K\coloneqq B\times\mtsDcpC u$ and $L\coloneqq B\times\mtsESdpC u$.
  Consider the following unions of simplicial subsets:
  \begin{alignat*}{3}
    K' &\coloneqq{}&(A\times\mtsDcpC u) &\cup (B\times (X\times\set{1})) &&\subseteq K, \\
    L' &\coloneqq{}&(A\times\mtsESdpC u) &\cup (B\times (X\times\set{1})) &&\subseteq L.
  \end{alignat*}
  Let $Q\in\mtcSSet$ be a quasi-category, and assume that there is a following commutative
  square of simplicial sets:
  \begin{equation} \label{eq:cor-dcpc-esdpc-loc-homotopy-extension:lkq}
    \begin{tikzcd}
      K' \arrow[d, hook] \arrow[r] & L' \arrow[d, "g"] \\
      K \arrow[r, "f"'] & Q
    \end{tikzcd}
  \end{equation}
  Here, unlabeled edges represent canonical maps (see
  \cref{def:dcp-dcpi-esdp-esdpi-pushout} for $K'\to L'$),
  and the labeled morphisms $f$ and $g$ are
  arbitrary. Suppose that $f$ sends the edges in the following subset
  to equivalences in $Q$:
  \[
  W_K\coloneqq\set{(\idmor[b], e)}[b \in B_0, e \in W^C_{u}]
  \subseteq B_1 \times (\mtsDcpC u)_1 = K_1.
  \]
  Then there exists an
  extension $h\colon L\to Q$ of $g$, such that the composite
  $K \to L \overset{h}{\to} Q$ is naturally equivalent to $f$ relatively to $K'$.
\end{corollary}
\begin{proof}
  By \cref{prop:dcp-esd-pushout-localization} and some simple computation,
  we see that the canonical $K\to L$ exhibits $L$ as the localization
  of $K$ at $W_K$. Moreover, we claim that its restriction $K'\to L'$ also exhibits $L'$ as the
  localization of $K'$ at:
  \[
  W_{K'}\coloneqq W_K \cap K' = (A\times W^C_{u}) \cup \set{\idmor[p]}[{p \in (B\times(X\times\set{1}))_0}].
  \]
  To see the claim, note that the subcomplexes $K'$ and $L'$ are the following
  pushouts, and that $K'\to L'$ is induced by the obvious transformation of the pushout-defining spans:
  \begin{align*}
    K' &= (A\times\mtsDcpC u) \cup_{A\times(X\times\set{1})} (B\times(X\times\set{1})), \\
    L' &= (A\times\mtsESdpC u) \cup_{A\times(X\times\set{1})} (B\times(X\times\set{1})).
  \end{align*}
  Each of the three consituent maps of the transformation, say $U \to V$, exhibits $V$ as the localization
  of $U$ at $W_K \cap U$. Since the two arrows in each of the two spans are inclusions,
  we derive from \cite[\href{https://kerodon.net/tag/01N7}{Proposition 01N7}]{kerodon}
  that the localization property is inherited by the pushout, as we have claimed.

  Apply $\operatorname{Fun}(\bullet,Q)$ to the canonical commutative square
  \eqref{eq:cor-dcpc-esdpc-loc-homotopy-extension:lk} below to get another one.
  The latter square, by the definition of localizations, admits the following factorization
  \eqref{eq:cor-dcpc-esdpc-loc-homotopy-extension:fun}
  through subcomplexes, where the arrows with $\sim$ are equivalences of quasi-categories:
  \begin{gather}
    \begin{tikzcd}[ampersand replacement=\&]
      L \arrow[from=r] \arrow[from=d, hook] \& K \arrow[from=d, hook] \\
      L' \arrow[from=r] \& K'
    \end{tikzcd}\label{eq:cor-dcpc-esdpc-loc-homotopy-extension:lk}\\
    \begin{tikzcd}[ampersand replacement=\&]
      \operatorname{Fun}(L,Q) \arrow[r, "\sim"] \arrow[d]
      \&\operatorname{Fun}((K, W_K), Q^\natural) \arrow[r, hook] \arrow[d]
      \&\operatorname{Fun}(K,Q) \arrow[d] \\
      \operatorname{Fun}(L',Q) \arrow[r, "\sim"]
      \&\operatorname{Fun}((K', W_{K'}), Q^\natural) \arrow[r, hook]
      \&\operatorname{Fun}(K',Q)
    \end{tikzcd}\label{eq:cor-dcpc-esdpc-loc-homotopy-extension:fun}
  \end{gather}

  From $L' \subseteq L$, $K' \subseteq K$, and $W_{K'} \subseteq W_K$, we see that the
  vertical maps in \eqref{eq:cor-dcpc-esdpc-loc-homotopy-extension:fun} are isofibrations of quasi-categories.
  To detail the derivation, for the left and the right vertical arrows, you may use
  \cite[\href{https://kerodon.net/tag/01F3}{Corollary 01F3}]{kerodon}, $L' \subseteq L$,
  and $K' \subseteq K$. Next, we derive from \cite[\href{https://kerodon.net/tag/01MN}{Remark 01MN}]{kerodon}
  that $\operatorname{Fun}((K', W_{K'}), Q^\natural) \hookrightarrow
  \operatorname{Fun}(K',Q)$ is an isofibration of quasi-categories. Now note that the outer rectangle in the
  following commutative diagram is a pullback:
  \[
  \begin{tikzcd}
    \operatorname{Fun}((K', W_{K'}), Q^\natural) \arrow[rr] \arrow[d, equal]
    && \operatorname{Fun}((K, W_K), Q^\natural) \arrow[d, hook] \\
    \operatorname{Fun}((K', W_{K'}), Q^\natural) \arrow[r, hook]
    & \operatorname{Fun}(K',Q) \arrow[r]
    & \operatorname{Fun}(K,Q)
  \end{tikzcd}
  \]
  Since the lower horizontal composite is an isofibration of quasi-categories, 
  \cite[\href{https://kerodon.net/tag/01H4}{Corollary 01H4}]{kerodon} implies that 
  $\operatorname{Fun}((K, W_K), Q^\natural) \to \operatorname{Fun}((K', W_{K'}), Q^\natural)$
  is an isofibration of quasi-categories.

  We have assumed in \eqref{eq:cor-dcpc-esdpc-loc-homotopy-extension:lkq} that
  $g\in\operatorname{Fun}(L',Q)$ goes to $\left. f \right|_{K'} \in \operatorname{Fun}((K', W_{K'}), Q^\natural)$
  in the diagram \eqref{eq:cor-dcpc-esdpc-loc-homotopy-extension:fun}. 
  Let us write $F_g$ for the fiber of $\operatorname{Fun}(L,Q)\to\operatorname{Fun}(L',Q)$ over $g$,
  and $F_f$ for the fiber of $\operatorname{Fun}((K, W_K), Q^\natural)\to\operatorname{Fun}((K', W_{K'}), Q^\natural)$
  over $\left. f \right|_{K'}$.
  It suffices to show that the restriction $F_g\to F_f$ of the middle vertical map
  in \eqref{eq:cor-dcpc-esdpc-loc-homotopy-extension:fun} is an equivalence of quasi-categories, for the map $h$ we seek
  is a vertex of $F_g$ that maps to an object equivalent to $f$ in $F_f$.

  Consider the following 1-categorical pullback diagrams in $\mtcSSet$ that define $F_g$ and $F_f$:
  \begin{align*}
    &
    \begin{tikzcd}[ampersand replacement=\&]
      F_g \arrow[d] \arrow[r, hook] \& \operatorname{Fun}(L,Q) \arrow[d] \\
      \mtyndSpx\deltaBra{0} \arrow[r, hook, "g"'] \& \operatorname{Fun}(L',Q)
    \end{tikzcd}
    &
    \begin{tikzcd}[ampersand replacement=\&]
      F_f \arrow[d] \arrow[r, hook] \& \operatorname{Fun}((K, W_K), Q^\natural) \arrow[d] \\
      \mtyndSpx\deltaBra{0} \arrow[r, hook, "\left. f \right|_{K'}"'] \& \operatorname{Fun}((K', W_{K'}), Q^\natural)
    \end{tikzcd}
  \end{align*}
  In these squares, the vertices are all quasi-categories, and as we have seen, the right vertical arrows are both
  isofibrations of quasi-categories. Therefore these pullbacks are in fact homotopy pullback squares with respect to
  the Joyal model structure on $\mtcSSet$, which are equivariant under categorical equivalences,
  which shows our sufficient claim for our corollary.
\end{proof}

\section{The proof of \texorpdfstring{$(\infty,1)$}{(∞,1)}-localization}
\label{sec:down-last-infty-loc}

Now, we are finally ready to show that $\last[\Gamma]$ is a localization
map, where $\Gamma$ is
one of the four categories $\int\catNerve(C)$, $\int\catNerve^{{-},{+}}(C)$, $\Downstar(C)$, and $\Down(C)$.
We begin with the easiest case, $\Gamma=\int\catNerve(C)$.

\begin{proposition}
  The simplicial map $\nerve(\last)\colon\nerve(\int\catNerve(C))\to\nerve(C)$ exhibits the
  domain $\nerve(\int\catNerve(C))$ as the localization of $\nerve(C)$ at $\last$-weak equivalences.
\end{proposition}
\begin{proof}
  This follows from the proof of \cref{lem:groth-tot-last-localiz}, which shows that
  ${\last}\colon\int\catNerve(C)\to C$ is a refective localization.
\end{proof}

We now treat the remaining cases, which are of our main interest. 

\begin{notation}
  In the following discussion, we set $S=\nerve(C)$, where $C$ is the Reedy category
  that was fixed throughout this paper, and $T=\nerve(\Gamma)$, where $\Gamma$ is
  any of the three categories $\int\catNerve^{{-},{+}}(C)$, $\Downstar(C)$,
  and $\Down(C)$. For the purpose of aesthetics,
  We shall abbreviate as $\lambda={\last}$
  the functor $\nerve(\last)\colon T\to S$ of quasi-categories corresponding to
  $\last\colon\Gamma\to C$ from \cref{def:down-fctr-last}.
\end{notation}

We begin with defining the $(\infty,1)$-diagrams used in the proof of our
$(infty,1)$-localization theorem.

\begin{lemma}\label{lem:down-and-orig-dcp-esdp-maps}
  There are simplicial maps
  $D\colon \mtsDcp S \to T$, $D^I\colon \mtsDcpI T \to T$,
  $E\colon \mtsESdp S \to S$, and $E^I\colon \mtsESdpI T \to S$
  that satisfy the following properties:
  \begin{enumerate}
    \item \label{item:lem-down-and-orig-dcp-esdp-maps:comm}%
    The following diagram commutes:
    \[
      \begin{tikzcd}[row sep=small, column sep=small]
        T \arrow[rrrr, "\lambda={\last}"] \arrow[dddd, equal]
        &&&& S \arrow[dddd, equal]\\
        & \mtsDcp S \arrow[ul, "D"]\arrow[r, twoheadrightarrow]
        & \mtsESd S \arrow[r, hook]
        & \mtsESdp S \arrow[ur, "E"']
        & \\
        & \mtsDcp T \arrow[u, "\lambda"'] \arrow[d, hook] \arrow[r, twoheadrightarrow]
        & \mtsESd T \arrow[u, "\lambda"'] \arrow[d, hook] \arrow[r, hook]
        & \mtsESdp T \arrow[u, "\lambda"'] \arrow[d, hook]
        & \\
        & \mtsDcpI T \arrow[ld, "D^I"'] \arrow[r, twoheadrightarrow]
        & \mtsESdI T \arrow[r, hook]
        & \mtsESdpI T \arrow[rd, "E^I"]
        & \\
        T \arrow[rrrr, "\lambda={\last}"']
        &&&& S
      \end{tikzcd}
    \]
    \item \label{item:lem-down-and-orig-dcp-esdp-maps:retr}%
    The maps $D^I$ and $E$ are retractions of the canonical
    injections.
    \item \label{item:lem-down-and-orig-dcp-esdp-maps:cm-last}%
    We have the following commutative square:
    \[
      \begin{tikzcd}
        T \times \mtyndSpx\deltaBra{1} \arrow[r, hook] \arrow[d, two heads, "\mathrm{proj}"']
        & \mtsESdpI T \arrow[d, "E^I"] \\
        T \arrow[r, "\lambda=\last"']
        & S
      \end{tikzcd}
    \]
    \item \label{item:lem-down-and-orig-dcp-esdp-maps:weq}%
    The map $D\colon \mtsDcp S \to T$ sends the inverse image
    of degenerate edges under $\mtsDcp S \twoheadrightarrow \mtsESd S$
    to $\last$-weak equivalences in $T$.
    \item \label{item:lem-down-and-orig-dcp-esdp-maps:weqI}%
    Let $e\in(\mtsDcpI T)_1$ be an edge.
    If the image of $e$ under $\mtsDcpI T \twoheadrightarrow \mtsESdI T$
    is either degenerate or 
    in the subcomplex $\mtsESdI(\mtsSk{0} T)\subseteq\mtsESdI T$,
    then the map $D^I$ sends $e$ to a $\last$-weak equivalence in $T$.
  \end{enumerate}
\end{lemma}
\begin{proof}
  Let $\Phi \colon \functCat{\deltaBra{1}}{C} \to \functCat{\deltaBra{2}}{C}$
  denote the functorial factorization that corresponds to the Reedy factorization of $C$.
  That is: $\Phi$ is a strict section of the composition functor
  $\backwardinduce{(\delta^2_1)}\colon \functCat{\deltaBra{2}}{C} \to \functCat{\deltaBra{1}}{C}$;
  for each $u\in\Ob\functCat{\deltaBra{1}}{C} = \Mor C$, we have
  $\backwardinduce{(\delta^2_2)}(\Phi(u)) \in C_{-}$
  and $\backwardinduce{(\delta^2_0)}(\Phi(u)) \in C_{+}$. Write
  $Q\coloneqq \backwardinduce{(\iota_1)}\compos \Phi \colon \functCat{\deltaBra{1}}{C} \to C$ for the
  ``midpoint'' functor of the factorization $\Phi$. The natural transformation that corresponds to
  $\backwardinduce*{\delta^2_0}\compos\Phi\colon \functCat{\deltaBra{1}}{C} \to \functCat{\deltaBra{1}}{C}$
  shall be denoted by $\eta\colon Q\Rightarrow \mathrm{cod}=\backwardinduce*{\iota^2_2}\colon
  \functCat{\deltaBra{1}}{C}\to C$: the ``second arrow.''
  If $(f,g)\colon u\to v$ is a morphism
  in $\functCat{\deltaBra{1}}{C}$, then $Q$ and $\eta$ fit into the following commutative
  diagram:
  \[
  \begin{tikzcd}
    w \arrow[r, "u"] \arrow[d, "f"']
    & x \arrow[d, "g"] \\
    y \arrow[r, "v"']
    & z
  \end{tikzcd}
  \quad \overset{\Phi}{\longmapsto} \quad
  \begin{tikzcd}
    w \arrow[rr, bend left, "u"] \arrow[r, two heads] \arrow[d, "f"']
    & Q(u) \arrow[r, tail, "{\eta_u}"'] \arrow[d, "{Q(f,g)}"]
    & x \arrow[d, "g"] \\
    y \arrow[rr, bend right, "v"'] \arrow[r, two heads]
    & Q(v) \arrow[r, tail, "{\eta_v}"]
    & z
  \end{tikzcd}
  \]
  Note that if $f \in C_{-}$, then $Q(f,g) \in C_{-}$, and if $g \in C_{+}$, then $Q(f,g) \in C_{+}$.

  We shall first construct $E$. With the following definitions in mind:
  \[ \mtsESdp S = \mtColim_{\mtyndSpx\deltaBra{n} \to S} \nerve(\mtpESdp\deltaBra{n}), \quad
  S = \nerve(C), \]
  we see that it suffices to compatibly define
  $E_\varphi\colon \mtpESdp\deltaBra{n}\to C$ for each $n$-simplex
  $\varphi\colon\deltaBra{n}\to C$ of $S=\nerve(C)$.
  Therefore we set $E_\varphi$ as the composite of:
  \[
  \begin{tikzcd}
    \mtpESdp\deltaBra{n} = \functCat{\deltaBra{1}}{\deltaBra{n}}
    \arrow[r, "\forwardinduce{\varphi}"]
    & \functCat{\deltaBra{1}}{C} \arrow[r, "Q"]
    & C,
  \end{tikzcd}
  \]
  which is compatible with the structure maps
  of the colimit-defining diagram for $\mtsESdp S$. 
  Combining, we obtain a simplicial map $E\colon\mtsESdp S\to S$.

  The construction of $E^I$ is similar; for each $n$-simplex $\varphi\colon\deltaBra{n}\to T$ of $T=\nerve(\Gamma)$,
  we set $E^I_\varphi$ as the composite of:
  \[
  \begin{tikzcd}[column sep=small]
    \mtpESdpI\deltaBra{n} \ar[r, hook] & \deltaBra{1} \times \functCat{\deltaBra{1}}{\deltaBra{n}} \ar[r, two heads]
    & \functCat{\deltaBra{1}}{\deltaBra{n}} \ar[rr, "\forwardinduce{\varphi}"]
    && \functCat{\deltaBra{1}}{\Gamma} \ar[rr, "\forwardinduce{\last}"]
    && \functCat{\deltaBra{1}}{C} \ar[r, "Q"]
    & C.
  \end{tikzcd}
  \]
  Here, the first inclusion is mentioned in the definition of $\mtpESdpI$ itself in \cref{def:esdp-esdpi}.
  The construction is natural in $\varphi$, so bundles together to yield a simplicial map $E^I\colon\mtsESdpI T\to S$.

  The construction of $D$ requires a little harder work. Remember $T = \nerve(\Gamma)$; we may focus
  on the case $\Gamma = \int\catNerve^{{-},{+}}(C)$, for the cases
  $\Gamma = \Downstar(C)$ and $\Gamma = \Down(C)$ are treated through the
  post-composition of the quotient functor $\int\catNerve^{{-},{+}}(C)\to\Downstar(C)$
  and the quasi-inverse of the equivalence $\Down(C)\hookrightarrow\Downstar(C)$.

  Let $\varphi\colon\deltaBra{n}\to C$ be an $n$-simplex of $S=\nerve(C)$. Similarly to above, 
  we wish to compatibly define a functor $D_\varphi\colon\mtcDcp\deltaBra{n}\to\Gamma$. 
  Let $(x,\alpha\colon\deltaBra{m}\to\deltaBra{n})\in\Ob(\mtcDcp\deltaBra{n})$. We set
  $D_\varphi(x,\alpha) \coloneqq (\deltaBra{m}, Q^\varphi_{(x,\alpha)})$, where
  $Q^\varphi_{(x,\alpha)}\colon\deltaBra{m}\to C_{-}$ is the following functor:
  \begin{align*}
    \deltaBra{m}\ni i &\longmapsto Q(\varphi(\uniqmor_{x,\alpha(i)}))\in\Ob C;\\
    (\uniqmor_{i,j}\colon i\le j) &\longmapsto Q(\idmor[\varphi(x)],\varphi(\uniqmor_{\alpha(i),\alpha(j)})).
  \end{align*}
  Let the following be an arbitary morphism in $\mtcDcp\deltaBra{n}$:
  \[
  (\uniqmor_{x, x'}, \beta)\colon (x, \alpha\colon\deltaBra{m}\to\deltaBra{n})
  \to (x', \alpha'\colon\deltaBra{m'}\to\deltaBra{n}).
  \]
  Remember, since $\beta\colon \alpha \to \alpha'$ is a morphism in $\overcat{\mtcSimpx}{\deltaBra{n}}$,
  it is a morphism $\beta\colon \deltaBra{m}\to\deltaBra{m'}$ in $\mtcSimpx$. We shall set:
  \[ 
   D_\varphi(\uniqmor_{x,x'},\beta)\coloneqq (\beta, \theta^{\varphi xx'\alpha})
   \colon (\deltaBra{m}, Q^\varphi_{(x,\alpha)}) \to
   (\deltaBra{m'}, Q^\varphi_{(x',\alpha')}).
  \]
  Here, the natural transformation
  \[
  \theta^{\varphi xx'\alpha}\colon Q^\varphi_{(x,\alpha)} \Rightarrow Q^\varphi_{(x',\alpha)} = Q^\varphi_{(x',\alpha')}\compos\beta
  \colon\deltaBra{m}\to C
  \]
  is defined by, for each $i\in\deltaBra{m}$:
  \[
  \theta^{\varphi xx'\alpha}_i \coloneqq Q(\varphi(\uniqmor_{x,x'}),\idmor[\varphi(\alpha(i))])
  \colon Q(\varphi(\uniqmor_{x,\alpha(i)})) 
  \to Q(\varphi(\uniqmor_{x',\alpha(i)})) = Q(\varphi(\uniqmor_{x',\alpha'(\beta(i))})),
  \]
  which belongs to $C_{+}$ by the property of $Q$. The naturality of $\theta^{\varphi xx'\alpha}$ follows
  from the structure of the domain $\functCat{\deltaBra{1}}{C}$ of $Q$ and the functoriality of $Q$;
  thus we have declared the mapping $D_\varphi$ on objects and morphisms. The functoriality of $D_\varphi$
  follows from that of $Q$ and $\varphi$,
  so we obtain a simplicial map $\nerve(D_\varphi)\colon\mtsDcp\mtyndSpx\deltaBra{n}\to T$.
  Since the naturality of $\nerve(D_\varphi)$ in $\varphi$ may be straightforwardly proven using the definition
  of the functor ${\mtcDcp}\colon \mtcSimpx\to\mtcCat$, we obtain a map
  $D\colon\mtsDcp S\to T$ of simplicial sets.

  We finally construct $D^I\colon\mtsDcpI T\to T$. We begin with the case $\Gamma=\int\catNerve^{{-},{+}}(C)$.
  Given an $n$-simplex $\varphi\colon\deltaBra{n} \to \Gamma$
  of $T=\nerve(\Gamma)$, we need to compatibly define a functor $D^I_{\varphi}\colon\mtcDcpI\deltaBra{n}\to\Gamma$.
  In this construction of $D^I_\varphi$, we will use the following symbols:
  \begin{alignat*}{2}
    (\deltaBra{m_x}, X_x) &\coloneqq \varphi(x)\in\Ob\Gamma
    &\quad&\text{for } x\in\deltaBra{n};\\
    (\beta_{xy}, \theta^{xy}) &\coloneqq \varphi(\uniqmor_{xy})\colon \varphi(x)\to\varphi(y)
    &\quad&\text{for } x\le y.
  \end{alignat*}

  First remember that $\mtcDcpI\deltaBra{n}$ has
  $\set{0}\times(\mtcDcp\deltaBra{n})$ and $\set{1}\times\deltaBra{n}$ as disjoint full subcategories.
  We set:
  \begin{align*}
    \left. D^I_{\varphi}\right|_{\set{0}\times(\mtcDcp\deltaBra{n})} 
    &\coloneqq D_{\lambda(\varphi)} = D_{{\last}\compos\varphi},\\
    \left. D^I_{\varphi}\right|_{\set{1}\times\deltaBra{n}} &\coloneqq \varphi.
  \end{align*}
  Since every object of $\mtcDcpI\deltaBra{n}$ belongs to one of these two full subcategories,
  it only remains to define how $D^I_{\varphi}$ acts on morphisms that connect objects in the distinct
  subcategories.
  
  By the definition of $\mtcDcpI\deltaBra{n}$, we only need to consider the unique
  morphism $\uniqmor_{(0,(x,\alpha)),(1,y)}\colon (0,(x,\alpha))\to(1,y)$, where 
  $(x,\alpha\colon\deltaBra{m}\to\deltaBra{n})\in\Ob(\mtcDcp\deltaBra{n})$ and $y\in\deltaBra{n}$
  are objects satisfying $\alpha(m) \le y$. We wish to construct
  \[
    D^I_{\varphi}(\uniqmor_{(0,(x,\alpha)),(1,y)})=(\beta, \theta)\colon
    (\deltaBra{m}, Q^{\lambda(\varphi)}_{(x,\alpha)}) \to \varphi(y)=(\deltaBra{m_y}, X_y).
  \]
  We would like to define $\beta\colon\deltaBra{m}\to\deltaBra{m_y}$ by, for each $i\in\deltaBra{m}$:
  \[
    \beta(i) \coloneqq \beta_{\alpha(i),y}(m_{\alpha(i)}) \in\deltaBra{m_y}.
  \]
  This is indeed a morphism in $\mtcSimpx$, as we have, for each pair $i\le j$ in $\deltaBra{m}$:
  \[
    \beta(i) = \beta_{\alpha(i),y}(m_{\alpha(i)}) = \beta_{\alpha(j),y}(\beta_{\alpha(i),\alpha(j)}(m_{\alpha(i)}))
      \le \beta_{\alpha(j),y}(m_{\alpha(j)}) = \beta(j).
  \]
  We next have to define a natural transformation
  $\theta\colon Q^{\lambda(\varphi)}_{(x,\alpha)}\Rightarrow X_y\compos\beta$.
  Let $i\in\deltaBra{m}$. We have the following morphisms in $C_{+}$:
  \begin{alignat*}{2}
    \eta_{\last(\varphi(\uniqmor_{x,\alpha(i)}))} 
    &\colon& Q^{\lambda(\varphi)}_{(x,\alpha)}(i) = Q(\last(\varphi(\uniqmor_{x,\alpha(i)})))
    &\rightarrowtail \last(\varphi(\alpha(i))),\\
    \theta^{\alpha(i)y}_{m_{\alpha(i)}} &\colon& \last(\varphi(\alpha(i))) = X_{\alpha(i)}(m_{\alpha(i)})
    &\rightarrowtail X_y(\beta_{\alpha(i)y}(m_{\alpha(i)})) = X_y(\beta(i)).  
  \end{alignat*}
  We set $\theta_i$ as the composite of these two morphisms, which are both natural in $i$.

  Now we have defined $D^I_{\varphi}$ on objects and morphisms, and it follows from the naturality of
  $\eta$ and the functoriality of $\varphi$ that $D^I_{\varphi}$ is a functor. The naturality in $\varphi$
  of the simplicial map $\nerve(D^I_{\varphi})\colon\mtsDcpI\mtyndSpx\deltaBra{n}\to T$ is
  straightforward to prove, so we obtain a simplicial map $D^I\colon\mtsDcpI T\to T$.

  Now we need to consider the cases $\Gamma=\Downstar(C)$ and $\Gamma=\Down(C)$, but we may obtain
  $D^I$ from the following commutative diagram:
  \[
  \begin{tikzcd}
    \mtsDcpI \nerve(\int\catNerve^{{-},{+}}(C)) \ar[d, "D_I"'] \ar[r, two heads]
    & \mtsDcpI\nerve(\Downstar(C)) \ar[d, dotted, "{}^{\exists!}D_I"]
    & \mtsDcpI\nerve(\Down(C)) \ar[d, dotted, "{}^{\exists!}D_I"] \ar[l, hook']\\
    \nerve(\int\catNerve^{{-},{+}}(C)) \ar[r, two heads]
    & \nerve(\Downstar(C)) \ar[r, two heads]
    & \nerve(\Down(C))
  \end{tikzcd}
  \]
  Here, the horizontal arrows are the canonical ones: two in the left are induced by the quotient functor 
  $\int\catNerve^{{-},{+}}(C)\twoheadrightarrow\Downstar(C)$, the other two in the right comes from the
  equivalence $\Down(C)\hookrightarrow\Downstar(C)$ of categories in \cref{lem:eqvce-down-c-downstar}.

  The required properties are now easy to verify, and we have obtained the desired lemma.
\end{proof}

We can now prove our theorem in the form of a lifting property, as follows:

\begin{lemma}\label{lem:down-last-infty-localiz-lift-ext-property}
  Let $A\subseteq B$ be a simplicial set and a simplicial subset, and let
  $Q$ be a quasi-category. Assume that we have the following (strictly) commutative
  square of simplicial sets, where $f$ and $g$ are arbitrary:
  \[
    \begin{tikzcd}[column sep=7.2em]
      A \times T \arrow[d, hook]
      \arrow[r, "{\idmor[A]\times\lambda=\idmor[A]\times{\last}}"]
      & A \times S \arrow[d, "g"] \\
      B \times T \arrow[r, "f"']
      & Q
    \end{tikzcd}
  \]
  Suppose that $f$ takes
  any edge of the form $(\idmor[b],t)\in B_1\times T_1=(B\times T)_1$,
  with $\idmor[b] \in B_1$ a degenerate edge and $t\in T_1$ a $\last$-weak equivalence, 
  to an equivalence in $Q$.
  Then, there exists a strict extension $h\colon B\times S \to Q$ of $g$
  such that its pre-composition $h\compos (\idmor[B]\times\lambda)\colon B\times T \to Q$ and $f$
  are naturally equivalent relative to $A\times T$.
\end{lemma}
\begin{proof}
  Take $D$, $D^I$, $E$, and $E^I$ as in \cref{lem:down-and-orig-dcp-esdp-maps}.
  The universality of pushouts and \cref{lem:down-and-orig-dcp-esdp-maps}~%
  \labelcref{item:lem-down-and-orig-dcp-esdp-maps:comm} allow us construct
  $\tilde{D}\colon \mtsDcpC\lambda\to T$ and
  $\tilde{E}\colon \mtsESdpC\lambda\to S$ from
  these four maps. In terms of $\tilde{D}$ and $\tilde{E}$,
  \cref{lem:dcp-esd-esdp-i-comm} and the properties
  \labelcref{item:lem-down-and-orig-dcp-esdp-maps:comm,item:lem-down-and-orig-dcp-esdp-maps:retr,%
  item:lem-down-and-orig-dcp-esdp-maps:cm-last} from \cref{lem:down-and-orig-dcp-esdp-maps}
  can be wrapped up in the following commutative diagram:
  \begin{equation}\label{eq:lem-down-last-infty-localiz-lift-ext-property:tildeDE}
    \begin{tikzcd}
      T \times \set{1}\arrow[dd, equal] \arrow[rrr, hook]\arrow[dr, hook]
      &&&[-3.6em] (S\times\set{0})\cup_{T\times\set{0}}(T\times\mtyndSpx\deltaBra{1})
      \arrow[d, "\mathrm{proj}", two heads] \arrow[dl, hook'] \\
      & \mtsDcpC\lambda\arrow[dl, "\tilde{D}"']\arrow[r]
      & \mtsESdpC\lambda\arrow[dr, "\tilde{E}"]
      & S\cup_T T\arrow[d, equal] \\
      T \arrow[rrr, "\lambda={\last}"']
      &&& S
    \end{tikzcd}
  \end{equation}
  
  We see that the following maps
  $\tilde{f}$ and $\tilde{g}$ are well-defined and satisfy the
  constraint of \cref{cor:dcpc-esdpc-loc-homotopy-extension},
  using our assumption, the diagram~\eqref{eq:lem-down-last-infty-localiz-lift-ext-property:tildeDE},
  and the properties~\labelcref{item:lem-down-and-orig-dcp-esdp-maps:weq,item:lem-down-and-orig-dcp-esdp-maps:weqI}
  from \cref{lem:down-and-orig-dcp-esdp-maps}: 
  \begin{align*}
    \tilde{f}&\colon\hspace{-0.45em}\begin{tikzcd}[ampersand replacement=\&]
      B\times\mtsDcpC\lambda \arrow[r, "{\idmor[B]\times\tilde{D}}"] \&[2.0em] B\times T \arrow[r, "f"] \& Q;
    \end{tikzcd}\\
    \hat{g}&\colon\hspace{-0.45em}\begin{tikzcd}[ampersand replacement=\&]
      A\times\mtsESdpC\lambda \arrow[r, "{\idmor[A]\times\tilde{E}}"] \&[2.0em] A\times S \arrow[r, "g"] \& Q;
    \end{tikzcd}\\
    \tilde{g}&\coloneqq \hat{g} \cup f\colon (A\times\mtsESdpC\lambda) \cup (B\times(T\times\set{1})) \to Q.
  \end{align*}
  Therefore, the map $\tilde{g}$ can be extended to $\tilde{h}\colon B\times \mtsESdpC\lambda\to Q$
  whose precomposition with the canonical $B\times\mtsDcpC\lambda\to B\times\mtsESdpC\lambda$
  is homotopic to $\tilde{f}$ relative to $A\times\mtsDcpC\lambda$.

  Let us define $h\colon B\times S\to Q$ as the following composite:
  \[
  \begin{tikzcd}[column sep=small]
    B\times (S\times\set{0}) \arrow[r, hook]
    & B\times ((S\times\set{0})\cup_{T\times\set{0}}(T\times\mtyndSpx\deltaBra{1})) \arrow[r, hook]
    & B\times \mtsESdpC\lambda \arrow[r, "\tilde{h}"] & Q.
  \end{tikzcd}
  \]
  If we write $u\subseteq v$ to mean a simplicial map $u$ is a restriction of $v$, we have: 
  \[
    g \subseteq \hat{g} \subseteq \tilde{g} \subseteq \tilde{h} \supseteq h,
  \]
  where the first $\subseteq$ is due to the commutativity 
  of the rightmost triangle in the diagram~\eqref{eq:lem-down-last-infty-localiz-lift-ext-property:tildeDE}.
  Since $\dom g = A \times S \times \set{0} \subseteq B\times S \times \set{0} = \dom h$, we see $g \subseteq h$.

  Now, define $H\colon (B\times T)\times\mtyndSpx\deltaBra{1}\to Q$ as the following composite:
  \[
  \begin{tikzcd}[column sep=small]
    B\times(T\times\mtyndSpx\deltaBra{1}) \arrow[r]
    & B\times((S\times\set{0})\cup_{T\times\set{0}}(T\times\mtyndSpx\deltaBra{1})) \arrow[r, hook]
    & B\times \mtsESdpC\lambda \arrow[r, "\tilde{h}"] & Q.
  \end{tikzcd}
  \]
  By definition, we see $\left. H \right|_{(B\times T) \times \set{0}} = h\compos(\idmor[B]\times\lambda)$.
  We also have $\left. H \right|_{(B\times T) \times \set{1}} = f$, because
  $f\subseteq\tilde{g}\subseteq\tilde{h}\supseteq \left. H \right|_{(B\times T) \times \set{1}}$.
  As $\hat{g}\subseteq\tilde{g}\subseteq\tilde{h}$, the map $H$ restricts to the
  $A\times T\times\mtyndSpx\deltaBra{1} \to Q$ in the following commutative diagram:
  \[
  \begin{tikzcd}[column sep=small]
    A\times (T\times\mtyndSpx\deltaBra{1}) \arrow[r] \arrow[d, two heads, "\mathrm{proj}"']
    & A\times((S\times\set{0})\cup_{T\times\set{0}}(T\times\mtyndSpx\deltaBra{1})) \arrow[r, hook]
      \arrow[d, two heads, "\mathrm{proj}"']
    & A\times \mtsESdpC\lambda \arrow[r, "\hat{g}"] \arrow[d, "{\idmor[A]\times\tilde{E}}"']
    & Q \arrow[d, equal] \\
    A\times T \arrow[r]
    & A\times (S\cup_T T) \arrow[r, equal]
    & A\times S \arrow[r, "g"]
    & Q
  \end{tikzcd}
  \]
  Here we have used the commutativity of the rightmost triangle in the 
  diagram~\eqref{eq:lem-down-last-infty-localiz-lift-ext-property:tildeDE} and the definition of $\hat{g}$.
  Thus we see that $H$ is a simplicial homotopy, or a natural transformation,
  from $h\compos(\idmor[B]\times\lambda)$ to $f$ relative to $A\times T$.

  It only remains to demonstrate that the natural transformation $H$ is a natural equivalence.
  We wish to prove, for each vertex $p\in (B\times T)_0$, that the edge
  \[
    e_p\coloneqq p\times\idmor[\mtyndSpx\deltaBra{1}]\colon\mtyndSpx\deltaBra{1}=
    \mtyndSpx\deltaBra{0}\times\mtyndSpx\deltaBra{1} \to
    (B\times T)\times\mtyndSpx\deltaBra{1}
  \]
  in $(B\times T)\times\mtyndSpx\deltaBra{1}$ is sent by $H$ to an equivalence in $Q$.
  If we write $\lambda_0\colon \mtsSk{0} T \to \mtsSk{0} S$ for the restriction of $\last$, then
  the image of $e_p$ in $B\times\mtsESdpC\lambda$ lies within the subcomplex
  $B\times\mtsESdpC\lambda_0\subseteq B\times\mtsESdpC\lambda$. 
  By simple inspection we see that the
  canoncial map from $\mtsDcpC\lambda_0$ surjects onto $\mtsESdpC\lambda_0$; therefore it reduces to show
  that the following composite sends every edge in the domain to an equivalence in $Q$:
  \[
  \begin{tikzcd}
    \mtsSk{0} B \times \mtsDcpC\lambda_0 \arrow[r, hook]
    & B \times \mtsDcpC\lambda \arrow[r]
    & B \times \mtsESdpC\lambda \arrow[r, "\tilde{h}"]
    & Q.
  \end{tikzcd}
  \]
  The restriction $\tilde{f}_0$ of $\tilde{f}$ to $\mtsSk{0} B\times\mtsDcpC\lambda_0$ is naturally equivalent to the
  composite above, by the construction of $\tilde{h}$. By
  \cref{lem:down-and-orig-dcp-esdp-maps}~\labelcref{item:lem-down-and-orig-dcp-esdp-maps:weqI}
  and our assumption on $f$, we see that
  $\tilde{f}_0$ sends every edge to an equivalence in $Q$; hence
  the same holds for the composite above, as desired.
\end{proof}

Now, our theorem is now simple to prove.

\begin{theorem}\label{thm:down-last-infty-localiz}
  The functor $\lambda=\last\colon T \to S$ exhibits $S$ as the localization of $T$ at the
  $\last$-weak equivalences.
\end{theorem}
\begin{proof}
  Follows from the previous \cref{lem:down-last-infty-localiz-lift-ext-property};
  see \cite[\href{https://kerodon.net/tag/01MS}{Proposition 01MS}]{kerodon}.
\end{proof}

\section{Conclusion: proof of the main theorem}
\label{sec:wrapping-up}

In this short section, we shall summarize the results of this paper into the theorem
stated in \cref{sec:introduction}.

\begin{theorem}[Restatement of \cref{thm:main}]
  Let $C$ be a small Reedy category. Then there is a concrete construction of
  a direct category $\Down(C)$, a set $\weqlast\subseteq\Mor\Down(C)$ of morphisms,
  and a functor ${\last}\colon \Down(C) \to C$
  that exhibits $C$ as the localization of $\Down(C)$ at $\weqlast$.
  Furthermore, if $C$ is finite, then $\Down(C)$ is also finite.
\end{theorem}
\begin{proof}
  The category $\Down(C)$ is constructed in \cref{def:down-c}; the functor $\last$
  is constructed in \cref{def:down-fctr-last}; and the set $\weqlast$ is constructed
  in \cref{def:down-weq-last}. The directness of $\Down(C)$ is proven in
  \cref{prop:down-c-direct}. It is shown that $(\Down(C),\last)$ is a 1-localization
  of $C$ at $\weqlast$ in \cref{thm:down-last-localiz} constructively; and it is demonstrated
  that the pair is $(\infty, 1)$-categorical localization in \cref{thm:down-last-infty-localiz}.
  Finally, the finiteness of $\Down(C)$ is proven in \cref{lem:down-c-finite}.
\end{proof}

\printbibliography

\end{document}